\documentclass[review,onefignum,onetabnum]{siamart171218}

\usepackage{geometry}
\usepackage{commath}
\usepackage{tcolorbox}

\usepackage{lipsum}
\usepackage{amsfonts}
\usepackage{graphicx}
\usepackage{epstopdf}
\usepackage{algorithmic}
\usepackage{subcaption}
\usepackage{mathtools}
\usepackage{mathrsfs}  
\usepackage{amsmath}
\usepackage{amssymb}
\ifpdf
  \DeclareGraphicsExtensions{.eps,.pdf,.png,.jpg}
\else
  \DeclareGraphicsExtensions{.eps}
\fi

\newsiamremark{remark}{Remark}
\newsiamremark{hypothesis}{Hypothesis}
\crefname{hypothesis}{Hypothesis}{Hypotheses}
\newsiamthm{claim}{Claim}

\headers{PARAMETERIZED WASSERSTEIN HAMILTONIAN FLOW}{H. Wu, S. Liu, X. Ye, and H. Zhou}

\title{Parameterized Wasserstein Hamiltonian Flow 
}

\author{Hao Wu\thanks{School of Mathematics, Georgia Institute of Technology, Atlanta, GA, USA
 (\email{hwu406@gmail.com}).}
 \and
 Shu Liu\thanks{Department of Mathematics, University of California, Los Angles, CA, USA(\email{shuliu@math.ucla.edu}).}
 \and 
 Xiaojing Ye\thanks{Department of Mathematics and Statistics, Georgia State University, Atlanta, GA, USA(\email{xye@gsu.edu}).}
 \and
 Haomin Zhou\thanks{School of Mathematics, Georgia Institute of Technology, Atlanta, GA, USA(\email{hmzhou@gatech.edu}).}
 }
\usepackage{amsopn}

\makeatletter
\newcommand*{\addFileDependency}[1]{
  \typeout{(#1)}
  \@addtofilelist{#1}
  \IfFileExists{#1}{}{\typeout{No file #1.}}
}
\makeatother

\newcommand*{\myexternaldocument}[1]{%
    \externaldocument{#1}%
    \addFileDependency{#1.tex}%
    \addFileDependency{#1.aux}%
}

\ifpdf
\hypersetup{
  pdftitle={Parameterized Wasserstein Hamiltonian Flow},
  pdfauthor={robot}
}
\fi

\myexternaldocument{ex_supplement}

\newcommand{\rhotheta}{\rho_{\theta}}
\newcommand{\thetadot}{\dot{\theta}}

\newtheorem{example}[theorem]{Example}
\newtheorem{assumption}{Assumption}

\begin{document}

\maketitle

\begin{abstract}

 In this work, we propose a numerical method to compute the Wasserstein Hamiltonian flow (WHF), which is a Hamiltonian system on the probability density manifold. Many well-known PDE systems can be reformulated as WHFs. We use parameterized function as push-forward map to characterize the solution of WHF, and convert the PDE to a finite-dimensional ODE system, which is a Hamiltonian system in the phase space of the parameter manifold. We establish theoretical error bounds for the continuous time approximation scheme in Wasserstein metric. For the numerical implementation, neural networks are used as push-forward maps. We design an effective symplectic scheme to solve the derived Hamiltonian ODE system so that the method preserves some important quantities such as Hamiltonian. The computation is done by fully deterministic symplectic integrator without any neural network training. Thus, our method does not involve direct optimization over network parameters and hence can avoid errors introduced by the stochastic gradient descent (SGD) or methods alike, which is usually hard to quantify and measure in practice. The proposed algorithm is a sampling-based approach that scales well to higher dimensional problems. In addition, the method also provides an alternative connection between the Lagrangian and Eulerian perspectives of the original WHF through the parameterized ODE dynamics. 
\end{abstract}


\section{Introduction}
Wasserstein Hamiltonian flow (WHF) describes the time evolution of a Hamiltonian system on a Wasserstein manifold. It can be formulated as the following first-order Hamiltonian system of dual coordinates on the Wasserstein manifold, which is the space of probability densities equipped with optimal transport distance \cite{villani2009optimal},
\begin{subequations}
\label{WHF in dual coordinates}\begin{align}
    &\partial_t\rho = \frac{\delta}{\delta\Phi}\mathcal{H}(\rho, \Phi),\\
    &\partial_t\Phi = -\frac{\delta}{\delta\rho}\mathcal{H}(\rho, \Phi),
\end{align}
\end{subequations}
with given initial values
\begin{equation}
    \label{eq:WHF-init}
    \rho(0, x)=\rho_0(x) \qquad \mbox{and} \qquad \Phi(0, x)=\Phi_0(x).
\end{equation}
In \eqref{WHF in dual coordinates}, $x\in \mathbb{R}^d$ (our theory applies to any Riemannian manifold $M$ without boundary but for simplicity we only consider $M=\mathbb{R}^{d}$ in this work) and $\frac{\delta}{\delta\rho}$ is the $L^2$ first variation, $\rho$ is the probability density, i.e., a non-negative function with $\int_{\mathbb{R}^d} \rho(x) dx =1$, $\Phi$ is called the dual function, whose gradient provides the vector field transporting $\rho$ on the Wasserstein manifold, and $\rho_0$ and $\Phi_0$ are their initial values respectively.
We consider the following general class of Hamiltonian:
\begin{align}
\label{eq:WHF-H}
\mathcal{H}(\rho, \Phi) = \int_{\mathbb{R}^d} \frac{1}{2}|\nabla \Phi(x)|^2\rho(x) dx+\mathcal{F}(\rho),  
\end{align}
where the first term is the kinetic energy associated with the 2-Wasserstein metric, and $\mathcal{F}(\rho)$ is a potential functional defined on the Wasserstein manifold, which typically is one or a combination of the three terms appeared in the following formula,
\begin{align}
\label{general potential energy}
    \mathcal{F}(\rho) = \int _{\mathbb{R}^d} V(x)\rho(x)dx + \int_{\mathbb{R}^d} U(\rho,x)dx + \iint_{\mathbb{R}^d\times \mathbb{R}^d} W(x-y)\rho(x)\rho(y)dxdy.
\end{align} 
Here the first term is determined by the linear potential $V$, the second one is a nonlinear functional $U$ of $\rho$ such as entropy or Fisher information, and the third is an interactive potential $W$ between particles whose population density is given by $\rho$.
Recent work \cite{chow2020wasserstein} reveals that WHF has deep connections to many well-known partial differential equations (PDEs), such as Wasserstein geodesic, Vlasov and Schr\"odinger equations, just to name a few. WHF provides an alternative framework and a set of new tools originated from optimal transport that potentially can be used to study those PDEs and relevant applications.
However, computation of WHF remains a challenging problem, especially in higher dimensions.
In this work, we develop a computational framework to solve WHF by leveraging several techniques together, including generative models, neural networks, symplectic integrator, and Wasserstein metric on density manifold. In particular, our method is readily scalable to solve WHFs in high-dimensional spaces.

There are two main objectives in this paper. The first one is to derive an effective finite-dimensional approximation of WHF \eqref{WHF in dual coordinates} by using reduced-order models. While the theory developed in this paper is applicable to general reduced-order models, we use a special class known as neural networks due to their excellent empirical approximation power in this study. Our derivation is conducted on a parameter space equipped with an induced Wasserstein metric and a submanifold of the probability density space where a density function $\rho$ is determined by a push-forward map parameterized by a neural network. For convenience, we call the induced Wasserstein metric on the parameter space as the \textit{pullback Wasserstein metric} in this paper. The resulting Hamiltonian system is a set of coupled ordinary differential equations (ODEs) for the neural network parameter and its dual, which is a finite-dimensional approximation to the infinite-dimensional WHF \eqref{WHF in dual coordinates}. We call this new Hamiltonian system the \textit{parameterized} WHF (PWHF). 

Our second objective is to develop a symplectic numerical scheme to solve the PWHF. This is accomplished by introducing an approximation to the pullback Wasserstein metric, which can be efficiently computed by recently developed machine learning techniques, such as residual neural networks \cite{he2016deep, 10.5555/3495724.3495951} or continuous normalizing flows \cite{chen2018neural, grathwohl2018ffjord}. The algorithm is designed by using samples only so that it is readily scalable to high-dimensional problems. Moreover, the proposed method allows effective computation of both particle motion of the Hamiltonian system in the classical phase space and the density evolution on the Wasserstein manifold simultaneously.

Here we highlight several main features of the proposed method:
\begin{itemize}
    \item (Dimension reduction) The PDEs (\ref{WHF in dual coordinates}), which can also be viewed as an infinite-dimensional dynamical system, is effectively approximated by a finite-dimensional ODE system and solved by a customized symplectic numerical scheme.
    \item (Computation effectiveness) A simplified Wasserstein metric is introduced to greatly reduce  the computational cost when compared to that of the pullback Wasserstein metric on the parameter space.
    \item (Training free) The proposed method does \textit{not} involve non-convex optimization algorithms like stochastic gradient descent (SGD) methods which are commonly adopted in machine learning. This avoids errors introduced by those optimization methods that are usually difficult to control and analyze.  

    \item (Symplectic structure preservation) The proposed scheme preserves the symplectic structure of PWHF. Thus the Hamiltonian is conserved, even for large time horizon.
    \item (Error estimation) The convergence of the proposed scheme is guaranteed by error estimates obtained in the Wasserstein metric. 
    \item (Eulerian and Lagrangian formulation) PWHF provides a natural connection bridging the Eulerian and Lagrangian formulations of the underlying Hamiltonian system. 
\end{itemize}

The remainder of the paper is organized as follows. In Section \ref{sec: relatedwork}, we describe the related work to this study. We briefly introduce the WHF, and its equivalent formulations in Section \ref{Background WHF}. We derive the PWHF and its simplified dynamics in Sections \ref{sec: derive theta dynamics} and \ref{sec: derive relaxed PWHF} respectively. In Section \ref{pseudo inverse error analysis}, we show that the density $\rho_{\theta}$ obtained by PWHF is a good approximation to the true solution with provable error bound. Then we provide a numerical algorithm to effectively solve the PWHF, with details about the simplified pullback Wasserstein metric tensor in Section \ref{sec: numerical method}. Numerical results are given in Section \ref{sec:experiments}. We provide a discussion about potential applications of PWHF on other types of problems in Section \ref{sec:discussion} and conclude the paper in Section \ref{sec:conclusion}.

\section{Related work} \label{sec: relatedwork}
The formulation of WHF studied here is first introduced in the paper \cite{chow2020wasserstein} where a derivation framework based on Lagrangian functional for general WHF on density manifold is proposed.  This work also reveals the connections between WHF and several well-known PDEs through examples. Numerical methods have been developed for solving the WHF in recent works \cite{cui2022continuation, cui2022time}, in which the classical finite difference and shooting techniques are used to solve WHF in lower dimensions.

We note that the idea of introducing the metric defined on probability manifold to parameter space originates from \cite{amari1998natural} in which the Fisher metric is discussed. Later, the case of Wasserstein metric was introduced and studied in \cite{li2018natural} and \cite{chen2020optimal}. Soon after, the Wasserstein gradient flows defined on the parameter space of the generative model were introduced in \cite{li2019affine, liliuzhouzha, lin2021wasserstein, liu2022neural}. 

The present study is mostly inspired by a recent work on parametric Fokker-Planck equation (PFPE) \cite{liliuzhouzha, liu2022neural} which establishes a finite-dimensional approximation of the Fokker-Planck equation (FPE) by using push-forward maps, neural networks, and Wasserstein metric. Leveraging the viewpoint that the FPE is the gradient flow of relative entropy functional on Wasserstein manifold \cite{jordan1998variational, doi:10.1081/PDE-100002243}, PFPE is derived by taking the gradient flow of relative entropy projected onto the parameter space equipped with the pullback Wasserstein metric. The resulting PFPE is a system of ODEs for the parameters. Our work follows a similar strategy. We use the same parameter space defined by push-forward maps and neural networks, and a similar pullback Wasserstein metric on the parameter space. Different from PFPE, our aim is to establish PWHF on the parameterized Wasserstein submanifold. In addition, we introduce a new metric, which can be viewed as a close approximation to the one introduced in \cite{liu2022neural}. Such a new metric does not require $\rho_\theta$-weighted Helmholtz projection of vector fields $\partial_{\theta} T_\theta$
in \cite{liu2022neural}. As a result, using the new metric enables us to directly compute PWHF via customized symplectic scheme with provable accuracy and significantly reduced computational cost.

Since introduced in the seminal works \cite{lasry2006jeux, lasry2007mean} to describe the limiting behavior of stochastic differential games, the mean field games (MFGs) have been studied extensively including numerical methods \cite{achdou2010mean, achdou2020mean, camilli2012semi} and machine learning based approaches \cite{pmlr-v162-lauriere22a, pmlr-v130-cui21a, ruthotto2020machine}. The works reported in \cite{ruthotto2020machine, doi:10.1073/pnas.2024713118} provide methods for computing MFG in high dimensional cases. The WHF is closely related to MFG systems at least in their mathematical forms, i.e., the MFG systems with quadratic kinetic energy can be treated as WHFs with boundary conditions.

Hamiltonian Monte Carlo algorithms introduced in \cite{DUANE1987216} aim at generating samples from a given probability distribution by evolving an associated Hamiltonian system in the phase space. We refer readers to \cite{betancourt2017conceptual, girolami2011riemann} and the references therein for more details. Instead of sampling from a fixed terminal distribution, our research in the paper computes the entire density evolution of Hamiltonian flow.

There are also numerous pieces of research \cite{greydanus2019hamiltonian, toth2019hamiltonian, chen2021data, so2022data, so2022data} focusing on recovering the Hamiltonian, and predicting the dynamics of certain physical systems based on observed trajectories. This is called inverse problem in computing Hamiltonian system in some literature.  Neural networks are widely utilized in those studies to make the computation scalable for high-dimensional settings. Nevertheless, there are significant differences between our problem and theirs, with the most prominent one being that we aim at solving for the entire probability flow while the aforementioned researches always focus on particle-wised computation.

In a broader sense, PWHF and the proposed numerical method provide an alternative approach that can potentially be applied to solve some PDEs in higher dimension by using neural networks. In the past few years, various machine learning methods have been developed for solving PDEs. For example, a deep learning method based on backward stochastic differential equations (SDEs) has been designed to solve high dimensional parabolic PDEs in \cite{han2017deep}. Deep Ritz method (DRM) is studied to solve PDEs whose solutions can be reformulated as the minimizers of variational forms \cite{11a35c8971ea4ca6af3db3c5545af9ed}. Physics-informed neural network (PINN) is proposed as a general framework to solve PDEs by minimizing the residual in least squares sense \cite{raissi2019physics, haghighat2021physics}. Weak adversarial network (WAN) solves PDEs in weak forms through a minimax approach \cite{zang2020weak, bao2020numerical}. More recently, Fourier neural operator \cite{li2020fourier}, DeepONet \cite{wang2021learning}, and Neural control \cite{gaby2023neural} are constructed to approximate the solution operators by neural networks so that the computation can be carried out more efficiently when the same PDEs must be solved repetitively with different initial or boundary conditions. Those and many more studies have shown that deep neural networks (DNN) possess great potentials in handling high-dimensional PDEs with various non-linearities.  

\section{Parameterization of Wasserstein Hamiltonian flow}
In this section, we first briefly review the Wasserstein metric and WHF, then we derive the parameterization of WHF and suggest a strategy to speed up its computation by using an approximate Wasserstein metric. We provide a comprehensive error analysis of PWHF in the end.

\subsection{Formulation of Wasserstein Hamiltonian flow}
\label{Background WHF}
The review here follows the formulation detailed in \cite{chow2020wasserstein}. 
For simplicity, let $M$ be a smooth manifold without boundary. Let us consider the space of smooth density functions supported on $M$ with finite second moment:
\begin{align}
    \mathcal{P}(M)=\cbr[2]{\rho\in C^{\infty}(M)\, : \, \rho\geq 0, \int_M\rho  \,dx=1,\ \int_M |x|^2\rho \,dx < \infty },
\end{align} 
and its tangent space at $\rho\in \mathcal{P}(M)$:
\begin{align}
    T_{\rho}\mathcal{P}(M)=\cbr[2]{\sigma\in C^{\infty}(M):\int_M\sigma \,dx=0}.
\end{align}
We also denote the interior of $\mathcal{P}(M)$ as $\mathcal{P}_{+}(M) := \mathcal{P}(M) \cap \{\rho > 0\}$.

We introduce the tangent bundle and the cotangent bundle of $\mathcal{P}$ by denoting 
\begin{equation}
    \mathcal{T}\mathcal{P} = \bigcup_{\rho\in\mathcal{P}} \{\rho\}\times T_\rho \mathcal{P} , \label{def: tangent bundle}
\end{equation}
as the tangent bundle of $\mathcal{P}$ and 
\begin{equation}
    \mathcal{T}^*\mathcal{P} = \bigcup_{\rho\in\mathcal{P}} \{\rho\}\times T_\rho^* \mathcal{P}  ,\label{def: cotangentbundle}
\end{equation}
as the cotangent bundle of $\mathcal{P}$. 
Here for each $\rho$ the cotangent spaces $T^*_\rho\mathcal{P}$ is taken as $$C^\infty(\mathbb{R}^d)/\mathbb{R} = \{[\Phi]:\Phi\in C^\infty(\mathbb{R}^d)\},$$ where $[\Phi]$ is the equivalent class of functions that are identical to $\Phi$ up to a constant, i.e., $[\Phi] = \{\Phi+c: c\in\mathbb{R}\}$. In the following discussion, we always write the equivalent class $[\Phi]$ as $\Phi$ for simplicity. It is clear that $\nabla \phi=\nabla\Phi$ for any $\phi\in[\Phi]$. Thus 

we also denote $\nabla [\Phi]$ as $\nabla\Phi$ for convenience. 

The space $\mathcal{P}(M)$ becomes a metric space when equipped with the Wasserstein distance. For any $\rho_1, \rho_2\in \mathcal{P}(M)$, the 2-Wasserstein distance (we call it Wasserstein distance for short hereafter) between $\rho_1$ and $\rho_2$ is given by \cite{villani2009optimal}
\begin{equation*}
\label{W2 distance}
W_2(\rho_1, \rho_2) = \Big(\inf_{\pi\in \Pi (\rho_1, \rho_2)}\iint |x-y|^2d\pi(x, y)
\Big)^{1/2},
\end{equation*}
where $\Pi (\rho_1, \rho_2)$ is the set of joint distributions on $\mathbb{R}^d\times \mathbb{R}^d$ with $\rho_1$ and $\rho_2$ as the marginals. This distance naturally induces a metric on $\mathcal{P}(M)$. In fact,  for any $\rho \in \mathcal{P}(M)$ and $\sigma \in T_{\rho}\mathcal{P}(M)$, let us denote $\Delta_{\rho}:=\nabla\cdot(\rho\nabla)$ and $\Delta_{\rho}^{\dagger}$ be its pseudo inverse operator, i.e., $\Phi =(-\Delta_{\rho})^{\dagger}\sigma$ implies $\sigma = -\Delta_{\rho} \Phi$. 

It is shown that $\Phi$ is unique up to a constant for any given $\sigma$ \cite{chow2020wasserstein}. Then the Wasserstein metric is defined by $g^W({\rho})(\cdot, \cdot):T_{\rho}\mathcal{P}(M)\times T_{\rho}\mathcal{P}(M)\rightarrow \mathbb{R}$ as follows,
\begin{align}
\label{Wasserstein metric}
    g^W({\rho})(\sigma_1, \sigma_2)&=\int_M\sigma_1(x) (-\Delta_{\rho})^{\dagger}\sigma_2(x)\,dx
    =\int_M \nabla \Phi_{1}(x) \cdot \nabla \Phi_{2}(x) \rho(x) \,dx, 
\end{align}
where $-\Delta_{\rho} \Phi_{i}(x) = - \nabla \cdot (\rho(x) \nabla \Phi_{i}(x)) = \sigma_i(x)$ for any $x \in M$ and $i=1,2$.

It is known that, equipped with the Wasserstein distance, the density manifold $\mathcal{P}(M)$ becomes a Riemannian manifold on which various differential operators and geometric flows can be established. In particular, WHF is derived by considering the following variational problem, 

\begin{align}
    \label{vari WHF}
    \mathcal{I}(\rho)= \inf_{\rho}\cbr[2]{\int_0^T \mathcal{L}(\rho, \partial_t\rho)\,dt: \rho|_{t=0}=\rho_0, \rho|_{t=T}=\rho_T },
\end{align}
where $\mathcal{L}(\rho, \partial_t\rho) = \frac{1}{2}g^W({\rho})(\partial _t\rho, \partial _t\rho) -\mathcal{F}(\rho)$ is a functional defined on $\mathcal{T}\mathcal{P}$ known as the Lagrangian, and $\rho_0$ and $\rho_T$ are some given initial and terminal densities respectively. The solution of \eqref{vari WHF} satisfies the  Euler-Lagrange equation which can be written as a second-order PDE
\begin{align}
\label{2nd order WHF}
    \partial _{tt}\rho + \Gamma_W(\partial_t\rho, \partial_t\rho)=-\textrm{grad}_W\mathcal{F}(\rho),
\end{align}
where $\Gamma_W$ is a quadratic function of $\partial_t \rho$ called the Christopher symbol given by
\begin{align}
    \Gamma_W(\partial_t\rho, \partial_t\rho)=-\cbr[2]{\Delta_{\partial_t\rho}\Delta_{\rho}^{\dagger}\partial_t\rho+\frac{1}{2}\Delta_{\rho}(\nabla\Delta _{\rho}^{\dagger}\partial_t\rho)^2}, 
\end{align}
and $\textrm{grad}_W$ is the gradient operator on Wasserstein manifold, which is defined by following the standard Riemannian geometry: for any curve $\{\rho(t,\cdot)\}_{t\in (-\delta, \delta)}$ with $\rho|_{t=0} =\rho_0, \ \frac{d}{dt}\rho|_{t=0} = \dot{\rho}|_{t=0} = \dot{\rho}_{0}$, and $\delta>0$ on $\mathcal{P}_+(M)$, the gradient of $\mathcal{F}$ at $\rho_0$ in the sense of Wasserstein metric is defined by the unique tangent vector $\textrm{grad}_W\mathcal{F}(\rho_0)$ such that the following identity holds:

\begin{align*}
\label{Wass gradient}
    \frac{d}{dt}\mathcal{F}(\rho(t,\cdot))\bigg|_{t=0}=g^W(\rho_0)(\textrm{grad}_W\mathcal{F}(\rho_0), \dot{\rho}_0).
\end{align*}

By direct calculation, it can be shown that at any specific $\rho \in \mathcal{P}_{+}(M)$ there is
\begin{align*}
    \textrm{grad}_W\mathcal{F}(\rho)=g^W(\rho)^{-1}\del[2]{\frac{\delta \mathcal{F}}{\delta \rho}}=-\nabla \cdot \del[2]{\rho(x)\nabla\frac{\delta \mathcal{F}}{\delta \rho}(x)}.
\end{align*}
Furthermore, the second-order PDE (\ref{2nd order WHF}) can be reformulated as a system of first-order PDEs given in the following theorem. 
\begin{theorem}[\cite{chow2020wasserstein} Hamiltonian flow in dual coordinates]
Consider $\Phi=(-\Delta_{\rho})^{\dagger}\partial_t\rho\in\mathcal{T}^*_\rho\mathcal{P}$, then equation \eqref{2nd order WHF} is equivalent to \eqref{WHF in dual coordinates} which can be treated as a Hamiltonian system on $\mathcal{T}^* \mathcal{P}$.
\end{theorem}

WHF \eqref{WHF in dual coordinates} describes the evolution of density  $\rho$ as a function of space and time. This can be viewed as the Eulerian formulation if using the language of classical fluid mechanics. Likewise, the dynamics can be written in the Lagrangian formulation, which describes the particle motion, i.e., the evolution of particle position, $\boldsymbol{X}$ as a function of time $t$. $\boldsymbol{X}$ is a random variable whose distribution follows the density $\rho$ governed by \eqref{WHF in dual coordinates}. The connections between two formulations are summarized in the following theorem.

\begin{theorem}[\cite{chow2020wasserstein}]
\label{theorem: particle whf}
    Let $(\boldsymbol{X}(t))_{0\leq t<t_0}$ be a random process in $\mathbb{T}^d$ with density $\rho$. Suppose $\boldsymbol{X}(t)$ satisfies
\begin{equation} \label{pWHF-1}
\begin{split}
        &\frac{d^2}{dt^2}\boldsymbol{X}(t)=-\nabla\frac{\delta}{\delta \rho(t, \boldsymbol{X})}\mathcal{F}(\rho(t, \boldsymbol{X}), \boldsymbol{X}), \quad \textrm{for any }\boldsymbol{X}_0\in \mathbb{T}^d,\\
        &\dot{\boldsymbol{X}}(0)=\nabla \Phi_0(\boldsymbol{X}_0).
\end{split}
\end{equation}
Then the density $\rho(t, \cdot)$ of ${\boldsymbol{X}}$ is a solution of the WHF (\ref{WHF in dual coordinates}).
\end{theorem}

For simplicity, this theorem was presented with periodic boundary condition or the underlying manifold being $\mathbb{T}^d$. 
By introducing a new momentum variable $v(t, \boldsymbol{X})$,  \eqref{pWHF-1} 
is converted into a system of first-order equations: 
\begin{subequations}
  \label{particle dynamics}
  \begin{align}
      & \frac{d}{dt}{\boldsymbol{X}} = v(t, \boldsymbol{X}), \quad \boldsymbol{X}(0) = \boldsymbol{X}_0,\\
      & \frac{d}{dt}v(t, \boldsymbol{X}) = -\nabla\frac{\delta}{\delta \rho}\mathcal{F}(\rho(t, \boldsymbol{X}), \boldsymbol{X}),  \quad v(0, \boldsymbol{X}_0) = \nabla\Phi(0, \boldsymbol{X}_0).
  \end{align}
\end{subequations}
For convenience, we call the system \eqref{particle dynamics} the \emph{particle} WHF.

\subsection{Parameterized WHF}
\label{sec: derive theta dynamics}
As one of the main goals of this paper, we introduce the PWHF in this subsection. The adopted strategy is to project the Lagrangian $\mathcal{L}$ in \eqref{vari WHF} onto a parameter space defined by the push-forward maps, and then derive the corresponding Euler-Lagrange equation in the parameter space.

\subsubsection{Parameter space defined by push-forward maps}\label{section : param space} Let $T:\mathbb{R}^d\rightarrow \mathbb{R}^d$ be a measurable map, also called push-forward map in $\mathbb{R}^d$. Given a reference distribution $\lambda$, the induced push-forward distribution, denoted by $T_{\sharp}\lambda$, is defined as, 
\begin{align*}
    T_{\sharp}\lambda (E)=\lambda(T^{-1}(E))\textrm{ for all measurable } E\subset \mathbb{R}^d,
\end{align*}
where $T^{-1}(E)$ is the pre-image of $E$. 

Let us take $T$ as parameterized map, namely for any $\theta\in \Theta$, $T_{\theta}: \mathbb{R}^d\rightarrow\mathbb{R}^d$ is a parametric function with parameter $\theta$. 
Here $\Theta$, as a subset of $\mathbb{R}^m$, is called the \emph{parameter space}, where $m$ is the number of parameters of $T_{\theta}$ (i.e., the dimension of $\theta$). Typical examples of $T_{\theta}$ include Fourier expansion, finite element approximation, and neural networks. 

The map $T_{(\cdot)\#}:\Theta\rightarrow\mathcal{P}$ given by $\theta\mapsto T_{\theta\#}\lambda$ naturally defines an immersion map from $\Theta$ to the probability manifold $\mathcal{P}$. Collecting all parameterized distributions together, i.e.,
\begin{equation*}
  \mathcal{P}_\Theta = \left\{\rho_\theta = T_{\theta\sharp}\lambda~:~\theta\in\Theta\right\}, 
\end{equation*}
we obtain a finite-dimensional submanifold $\mathcal{P}_\Theta$ of $\mathcal{P}$. We can define the tangent space of $\mathcal{P}_\Theta$ at each $\theta$ as
$T_{\rho_\theta}\mathcal{P}_\Theta = \mathrm{span}\{\frac{\partial\rho_\theta}{\partial\theta_1},\cdots, \frac{\partial\rho_\theta}{\partial\theta_m}\}.$ The tangent bundle is then
$\mathcal{T}\mathcal{P}_\Theta = \cup_{\theta\in\Theta}\{\rho_\theta\}\times T_{\rho_\theta}\mathcal{P}_\Theta.$ On the other hand, the cotangent space $T_{\rho_\theta}^*\mathcal{P}_\Theta$ is the dual space of $T_{\rho_\theta}\mathcal{P}_\Theta$, and the cotangent bundle is $\mathcal{T}^*\mathcal{P}_\Theta = \cup_{\theta\in\Theta}\{\rho_\theta\}\times T_{\rho_\theta}^*\mathcal{P}_\Theta.$

A counterpart to the Wasserstein metric defined on $\mathcal{P}$ can be introduced on the parameter space $\Theta\in \mathbb{R}^m$ by using the pullback operator through $T_{\theta}$, i.e., $G(\theta) = {T_{\theta \sharp}}^* g^W$, where $g^W$ is Wasserstein metric tensor given in (\ref{Wasserstein metric}). This is the pullback Wassertein metric on the parameter space. It turns out that $G(\theta)$ is an $m \times m$ positive semi-definite matrix which defines a bilinear form on the tangent space of $\Theta$ at $\theta$, $\mathcal{T}_\theta\Theta\simeq\mathbb{R}^m$ (rigorously speaking $\mathcal{T}_\theta\Theta$ may be a subspace of $\mathbb{R}^{m}$ depending on the choice of $T_{\theta}$ as addressed in Remark \ref{rmk:singular-G} below). For any $\theta\in\Theta$ and $\xi_1,\xi_2 \in \mathcal{T}_\theta \Theta$, we have
\begin{equation}
  G(\theta)(\xi_1, \xi_2) = g^W(\rho_\theta)((T_{\theta \sharp})_*\xi_1, (T_{\theta \sharp})_*\xi_2), \label{metric tensor}
\end{equation}
where $(T_{\theta \sharp})_*\xi_i$ is the tangent vector at $T_{\theta\sharp}\lambda$ on the Wasserstein manifold due to the push-forward of $\xi_i$ by the map $T_{\theta\sharp}$ for $i=1,2$.

Following the study detailed in \cite{liu2022neural}, the metric tensor $G(\theta)$ takes the following form 

\begin{equation}
\label{eq:G}
  G(\theta) = \int  \nabla\Psi_\theta(T_\theta(z))\nabla \Psi_\theta(T_\theta(z))^{\top}~d \lambda(z),
\end{equation}
where $\Psi_\theta=(\psi_{\theta, 1},\cdots,\psi_{\theta, m})^\top: \mathbb{R}^{d} \to \mathbb{R}^{m}$ and $\nabla\Psi_\theta$ is the $m\times d$ Jacobian of $\Psi_\theta$. For each $j=1,2,\cdots,m$, $\psi_{\theta, j}$ solves the following equation:
\begin{equation}
  \nabla\cdot(\rho_\theta\nabla\psi_{\theta, j}(x)) = \nabla\cdot(\rho_\theta~\partial_{\theta_j} T_\theta(T^{-1}_\theta(x))), \label{Hodge Dcom}
\end{equation}
with condition $\lim_{x\rightarrow\infty}\rho_\theta(x)\nabla\psi_{\theta, j}(x)=0.$
We omit the derivation and the properties of $G(\theta)$. Interested readers are referred to Section 3.1 of \cite{liu2022neural} for further details.

\subsubsection{Parameterization of WHF}
\label{sec: PWHF}
We introduce the parameterization of WHF in this section, which is the first contribution of this work. Our treatment is outlined in the following flowchart:
\begin{tcolorbox}
\begin{gather}
 \textrm{Starting with a given Hamiltonian } ~\mathcal{H}(\rho,\Phi)  \nonumber \\
 \Downarrow  \nonumber \\
 \textrm{By taking Legendre Transform of $\mathcal{H}$, we obtain Lagrangian }~\mathcal{L}(\rho, \dot{\rho}) \nonumber \\
 \Downarrow \nonumber \\
 \textrm{Using $\mathcal{L}$, we can define $L$ on $\mathcal{T}\Theta$ as }~ L(\theta,\dot\theta) = \mathcal{L}((T_{\cdot \sharp}\lambda)(\theta), (T_{\theta \sharp}\lambda)_*\dot\theta) \nonumber \\
 \Downarrow \nonumber \\
 \textrm{By applying Legendre transform to $L$, we obtain the Hamiltonian in parameter space } ~ H(\theta, p) \nonumber \\
 \Downarrow \nonumber \\
 \textrm{We formulate the PWHF as }\nonumber\\
 \dot\theta(t) = \partial_p H (\theta(t), p(t)), \nonumber \\
 \dot p(t) = -\partial_\theta H(\theta(t), p(t)).\nonumber \\
 \nonumber
\end{gather}
\end{tcolorbox}

 Following this procedure, we derive the PWHF by leveraging the perspective of the Lagrangian mechanics. 
 To be more specific, as introduced in Section \ref{Background WHF}, we consider the Lagrangian 
\begin{equation*}
  \mathcal{L}(\rho,\dot\rho) = \frac{1}{2}g^W(\dot \rho, \dot \rho) - \mathcal{F}(\rho),
\end{equation*}
where $g^W$ is defined in (\ref{Wasserstein metric}) and $\mathcal{F}(\rho)$ takes the general form in \eqref{general potential energy}.
We define the counterpart Lagrangian $L$ of $\mathcal{L}$ on $\mathcal{T}\Theta$ as 
\begin{equation*}
  L(\theta, \dot\theta)=\mathcal{L}(T_{\theta \sharp}\lambda, (T_{\theta \sharp})_*\dot\theta).
\end{equation*}
More precisely, denote $\rho_\theta=T_{\theta\sharp}\lambda$, then $(T_{\theta \sharp})_*\dot \theta = \frac{\partial\rho_\theta}{\partial \theta} \dot{\theta}$, the Lagrangian $L$ takes the following form
\begin{equation}
\label{para Lagrangian}
    L(\theta, \dot\theta) = \mathcal{L} \del[2]{\rho_\theta, \frac{\partial \rho_\theta}{\partial \theta}\dot\theta}=\frac{1}{2}\dot\theta^\top G(\theta)   \dot\theta-F(\theta),
\end{equation}
where $G(\theta)$ 
is defined in \eqref{eq:G}, and $F(\theta):=\mathcal{F}(\rho_\theta)$. The detailed calculation of \eqref{para Lagrangian} is given in Appendix \ref{calculate L}.

\begin{theorem}[Euler-Lagrange equation in parameter space]
\label{Theorem exact PWHF}
Consider the Lagrangian $L$ defined in (\ref{para Lagrangian}), as well as the variational problem 
\begin{align}
    \label{vari exact PWHF}
    \mathcal{I}^\Theta (\theta)= \inf_{\theta}\cbr[2]{\int_0^T L(\theta, \dot{\theta})dt: \rho_{\theta}|_{t=0}=\rho_0, \rho_{\theta}|_{t=T}=\rho_T }.
\end{align}
The Euler-Lagrange equation of the above variational problem is the following second-order ODE,
\begin{equation}
\label{theta 2nd order PWGF}
G(\theta)\ddot{\theta} + \sum_{k=1}^m \dot{\theta}_{k}\partial_{\theta_k}G(\theta) \dot{\theta} - \frac{1}{2} [\dot{\theta}^{\top} \partial_{\theta_k} G(\theta) \dot{\theta}]_{k=1}^{m} = -\nabla_{\theta}F(\theta)
\end{equation}
with $\dot{\theta} = [\dot{\theta}_k]_{k=1}^{m}$. Here $[a_k]_{k=1}^{m}$ represents an $m$-dimensional vector with $a_k$ as its $k$th component.

\end{theorem}
\begin{proof}
Recall the Euler-Lagrange equation of the parameterized variational problem \eqref{vari exact PWHF} is
\begin{equation}
\label{eq:EL}
    \frac{d}{dt}\frac{\partial}{\partial\dot\theta}L(\theta, \dot\theta)=\frac{\partial}{\partial\theta}L(\theta, \dot\theta).
\end{equation}
The left-hand side of \eqref{eq:EL} is 
\begin{align*}
    \frac{d}{dt}\frac{\partial}{\partial\dot\theta}L(\theta, \dot\theta)=G(\theta)\ddot{\theta} + \sum_{k=1}^m \dot{\theta}_{k}\partial_{\theta_k}G(\theta) \dot{\theta}, 
\end{align*}
and the right-hand side of \eqref{eq:EL} is
\begin{align*}
\label{parametric WHF}
\frac{\partial}{\partial\theta}L(\theta, \dot{\theta})=\frac{1}{2} [\dot{\theta}^{\top} \partial_{\theta_k} G(\theta) \dot{\theta}]_{k=1}^{m}-\nabla_{\theta}F(\theta)  .
\end{align*}
Plugging them into \eqref{eq:EL} yields \eqref{theta 2nd order PWGF}. 
\end{proof}

Now we temporarily assume that $G(\theta)$ is non-singular for any $\theta\in\Theta$ (the more general case where $G(\theta)$ can be singular will be discussed in Remark \ref{rmk:singular-G}), we can introduce the associated Hamiltonian via Legendre transform. Specifically, we denote $\mathcal{T}^*\Theta$ as the phase space (cotangent bundle) of $\Theta$, then define $H(\cdot,\cdot): \mathcal{T}^*\Theta\rightarrow \mathbb{R}$ as
\begin{equation}
  H(\theta, p) = \sup_{\dot{\theta}}\,\{\dot\theta^\top p - L(\theta, \dot\theta)\} = \frac{1}{2} p^\top G(\theta)^{-1} p + F(\theta).  \label{def: H(theta,p)}
\end{equation}
Following the convention in classical mechanics \cite{goldstein2011classical}, we introduce the momentum 
\begin{equation}
  p = \frac{\partial L(\theta, \dot\theta)}{\partial \dot \theta} =  G(\theta)\dot\theta.   \label{introduce momemtun}
\end{equation}
Then the Hamiltonian system associated with \eqref{theta 2nd order PWGF} can be formulated as
\begin{subequations}
\label{eq:pwhf}
\begin{align}
   & \dot \theta = \frac{\partial H(\theta, p)}{\partial p} = G(\theta)^{-1} p ,  \label{pwhf 1}\\
   & \dot p = -\frac{\partial H(\theta, p)}{\partial \theta} = \frac{1}{2}[p^\top G(\theta)^{-1} \partial_{\theta_k}G(\theta)G(\theta)^{-1}p]_{k=1}^m - \nabla_\theta F(\theta). \label{pwhf 2}
\end{align}
\end{subequations}
We call the ODE system \eqref{eq:pwhf} the \textit{parameterized Wasserstein Hamiltonian flow} (PWHF).

\begin{remark}[Existence and uniqueness of PWHF]
    Under the assumption that $G(\theta)$ is non-singular on $\Theta$, one can verify that both $\frac{\partial H(\theta, p)}{\partial \theta}$  and  $\frac{\partial H(\theta, p)}{\partial p}$ are locally Lipschitz. Thus by the standard ODE theory, the PWHF \eqref{eq:pwhf} must have a unique solution over a finite time interval $[0, t^*)$ for some $t^*>0$ from any given initial value. However, determining $t^*$ is a challenging problem due to the complex structure of $T_{\theta}$ and geometry of $\mathcal{P}_{\Theta}$. We leave this for future investigations.
\end{remark}

\begin{remark}[Singular $G(\theta)$]\label{rmk:singular-G}
    In our derivation of the PWHF \eqref{eq:pwhf}, the metric tensor $G$ is assumed to be non-singular. This assumption can be relaxed. If $G$ is singular, the PWHF can be derived similarly with the following modifications: we restrict $\mathcal{T}_{\theta}\Theta = \mathcal{T}_{\theta}^*\Theta= \mathcal{R}(G(\theta)) \subset \mathbb{R}^m$ at each $\theta$, where $\mathcal{R}(\cdot)$ denotes the range (i.e., column space) of its argument matrix. Then $G(\theta)$ is positive definite, hence a non-degenerate inner product, on the tangent and cotangent spaces with corresponding bundles denoted by $\{(\theta, \mathcal{T}_{\theta}\Theta):\theta \in \Theta\}$ and $\{(\theta, \mathcal{T}_{\theta}^*\Theta):\theta \in \Theta\}$ respectively. We can define $H(\cdot,\cdot): \mathcal{T}^*\Theta \to \mathbb{R}$ by
\begin{equation}
  H(\theta, p) = \sup_{\dot{\theta}\in \mathcal{R}(G)} \{\dot\theta^\top p - L(\theta, \dot\theta)\} = \frac{1}{2} p^\top G(\theta)^{\dagger} p + F(\theta),
\end{equation}
where $G(\theta)^{\dagger}$ is the Penrose-Moore pseudo inverse of $G(\theta)$. In this case, the Legendre transform is well-defined, because the maximizer $\dot{\theta}=G(\theta)^{\dagger} p$ can be attained as long as $p \in \mathcal{T}^*_{\theta}\Theta = \mathcal{R}(G(\theta))$, which is always true given the definition of momentum $p=G(\theta)\dot{\theta} \in \mathcal{R}(G(\theta))$.
The resulting Hamiltonian system on $(\theta,p)$ is the same as \eqref{pwhf 1} and \eqref{pwhf 2} except for two modifications: $G^{-1}$ is replaced by $G^{\dagger}$; and \eqref{pwhf 2} has an additional term due to the derivative of the pseudo inverse $G(\theta)^{\dagger}$. Detailed derivation procedure about this additional term can be found in Section \ref{relaxedhd-pio}.  
\end{remark}

\subsubsection{Transformation between $\mathcal{T}^*\Theta$ and $\mathcal{T}^*\mathcal{P}_\Theta$ induced by the push-forward map}

In the previous section, we derive a Hamiltonian system PWHF \eqref{eq:pwhf} on the phase space $\mathcal{T}^*\Theta$. But how the parameter-momentum pair $(\theta, p)$ in PWHF relates to the probability-potential pair $(\rho, \Phi)$ in WHF is not clearly illustrated. In this section, we connect $(\theta, p)$ to $(\rho, \Phi)\in\mathcal{T}^*\mathcal{P}_\Theta$ by deriving a transformation $\tau$ that maps every $(\theta, p)\in\mathcal{T}^*\Theta$ to $(\rho_\theta, \Phi_{\theta,p})\in\mathcal{T}^*\mathcal{P}_\Theta$. We again assume $G(\theta)$ is non-singular here for simplicity, and the general singular case can be handled similarly by replacing $G(\theta)^{-1}$ with $G(\theta)^{\dagger}$ as described in Remark \ref{rmk:singular-G}.

To determine the map from $(\theta, p)$ to $\Phi_{\theta, p}$, we recall that in the derivation of WHF \eqref{WHF in dual coordinates} detailed in \cite{chow2020wasserstein}, the relation between $\Phi$ and $\dot\rho$ is given by
\begin{equation}
    \Phi = \partial_{\dot\rho}\mathcal{L}(\rho, \dot\rho). \label{rel on P}
\end{equation}
If restricting $(\rho, \dot\rho)$   
on $\mathcal{T}\mathcal{P}_\Theta$, i.e., setting $(\rho, \dot\rho)$ as $(\rho_\theta, (T_{\theta \sharp})_*\dot\theta)=(\rho_\theta, \frac{\partial\rho_\theta}{\partial\theta}\cdot\dot\theta)$, we obtain 
\begin{equation}
\label{rel psitilde}
    \Phi = \partial_{\dot\rho} \mathcal{L} \del[2]{\rho_\theta, \frac{\partial\rho_\theta}{\partial\theta}\cdot\dot\theta} =  -\Delta_{\rho_\theta}^\dagger(-\nabla\cdot(\rho_\theta\nabla\Psi_\theta^\top\dot\theta))=\Psi_\theta^\top\dot\theta, 
\end{equation}
where $\Psi_\theta:\mathbb{R}^d\rightarrow \mathbb{R}^m$ is defined in \eqref{Hodge Dcom}. On the other hand, we define the momentum $p\in \mathcal{T}^*_\theta \Theta$ through $p = G(\theta)\dot\theta$ in \eqref{introduce momemtun}. Thus, $\dot\theta = G(\theta)^{-1} p $, and by plugging this into \eqref{rel psitilde}, we obtain
\begin{equation}
 \Phi_{\theta, p} = \Psi_\theta^\top G(\theta)^{-1}  p. \label{transfm cotspc}
\end{equation}

By combining $\rho_\theta$ and $\Phi_{\theta,p}$ together, we obtain the following transformation $\tau$ from $\mathcal{T}^*\Theta$ to $T^*\mathcal{P}_\Theta$: 
\begin{align}
  \tau:\quad & ~   \mathcal{T}^*\Theta \longrightarrow \mathcal{T}^*\mathcal{P}_\Theta,\nonumber \\
        & (\theta,~p) \longmapsto (T_{\theta\sharp}\lambda,~ \Psi_\theta^\top G(\theta)^{-1}p ) .\label{def tau}
\end{align}
Further discussions on the geometric properties, such as whether or not $\tau$ preserves the symplectic form, are provided in Appendix \ref{append a }.

Once the solution $\{\theta(t), p(t)\}$ to PWHF is computed, the transformation $\tau$
gives a valid approximation to the solution $(\rho(t,\cdot), \Phi(t,\cdot))$ of WHF. However, in order to significantly improve computation efficiency of PWHF, we will introduce a simplified version of the pullback Wasserstein metric $G(\theta)$ in the next section, and establish a bound for the approximation error, measured by the  Wasserstein metric,  between our numerical solution $(\rho_{\theta(t)}, \partial_\theta T_{\theta(t)}\circ T_{\theta(t)}^{-1}(\cdot)\dot\theta(t))$ and the solution $(\rho(t,\cdot), \nabla \Phi(t,\cdot))$ of the original WHF. 

\subsection{PWHF with a simplified metric}
\label{sec: derive relaxed PWHF}
Theorem \ref{Theorem exact PWHF} reduces the PDE in density space to a parameterized system in finite-dimensional space, hence potentially providing a way to compute the WHF by numerical algorithms. However, the computational cost to solve (\ref{theta 2nd order PWGF}), as well as \eqref{eq:pwhf}, is still high. The main difficulty comes from the computation of the metric tensor $G$. More precisely, directly evaluating $G$ requires solving $m$ different elliptic PDEs where $m$ is the number of parameters in the pushforward map $T_{\theta}$, in which $m$ can be very large if we choose $T_{\theta}$ to be neural networks. In \cite{liu2022neural}, a bi-level minimization scheme is proposed to circumvent this challenge. By introducing several auxiliary functions, the term $G(\theta)^{-1}\theta$ is calculated as the critical point of a min-max problem. However, it may still be expensive to solve such optimization problems in general. In this paper, we develop another strategy by introducing a simplified metric $\widehat{G}$ and use it to replace $G$ in the derivation. This new metric not only yields much simpler implementation and more effective computations, but also enables us to establish a  theoretical estimate in Wasserstein metric to quantify the error of the approximation. Furthermore, our investigation shows that both the computation and theory can be extended to the general case where $\widehat{G}$ is not necessarily invertible, but positive semi-definite with constant rank. Numerical results also demonstrate excellent approximation accuracy of this new metric.

\begin{definition}[Simplified pullback Wasserstein metric in $\Theta$]
Let $T_{\theta}$ be the pushforward map and $\lambda$ be the reference distribution, we define the simplified pullback Wasserstein metric on $\mathcal{P}_\Theta$ as:
\begin{align}
    \widehat{G}(\theta)=\int  \partial_{\theta} T_\theta(z)^{\top} \partial_{\theta} T_\theta(z)~d\lambda(z).\label{relaxed metric tensor}
\end{align}
\end{definition}

It is worth mentioning that this definition is directly inspired by \eqref{eq:G} and \eqref{Hodge Dcom}. A more in-depth motivation is influenced by the work of Otto \cite{doi:10.1081/PDE-100002243} in which the Wasserstein metric is defined through an isometric submersion from  the space of push-forward operators $\mathcal{O}$ onto the Wasserstein manifold of density $\mathcal{P}(\mathbb{R}^d)$. Here by the space of push-forward operators we meant that $\mathcal{O}$ is the set of smooth transformations $T$ on $\mathbb{R}^{d}$. Let us consider the map $\mathscr{T}:\Theta\ni\theta\mapsto T_\theta\in\mathcal{O}$. We define the pullback metric on 
$\Theta$ as
\[ \widehat{G}(\theta) = \mathscr{T}^* g_{L^2(\lambda)}, \]
where $\mathscr{T}^*$ is the pullback operation induced by $\mathscr{T}$ and $g_{L^2(\lambda)}$ is the metric on the tangent space of $\mathcal{O}$.  To evaluate $\widehat{G}(\theta)$, we consider any curve $\{\theta(t)\}_{-\epsilon \leq t \leq \epsilon}$ in $\Theta$, and denote $\dot{\theta}(0) = \frac{d}{dt}\theta(t)|_{t=0}$. By the definition of pullback operation, we have
\[ \widehat{G}(\dot\theta(0), \dot\theta(0)) = g_{L^2(\lambda)}\left(\frac{d}{dt}T_{\theta(t)}|_{t=0}, \frac{d}{dt}T_{\theta(t)}|_{t=0}\right) 
= \dot\theta(0)^\top \left( \int_{\mathbb{R}^d} \partial_\theta T_\theta(z)^\top \partial_\theta T_\theta(z) \, d\lambda(z)\right)\dot\theta(0). \]
Thus we obtain \eqref{relaxed metric tensor}. More precisely, we have $\widehat{G}(\theta)=(\widehat{G}(\theta)_{ij})_{1\leq i,j\leq m}$, and  
\begin{equation*}
\widehat{G}(\theta)_{ij}=\sum_{k=1}^d\int_{\mathbb{R}^d}\partial_{\theta_i}T^{(k)}_\theta(z)\cdot\partial_{\theta_j}T^{(k)}_\theta(z)\,d\lambda(z),
\end{equation*}
where $T^{(k)}_{\theta}:\mathbb{R}^{d}\to \mathbb{R}$ is the $k$th component of $T_{\theta}:\mathbb{R}^{d} \to \mathbb{R}^{d}$.
The form of $\widehat{G}(\theta)$ indicates it as an $m\times m$ semi-positive definite matrix for any $\theta\in\Theta.$

Replacing $G$ by $\widehat{G}$ in the expression of $L$, we can establish similar results as those stated in Theorem \ref{Theorem exact PWHF}. 

\begin{theorem}
Consider the variational problem 
\begin{align}
    \label{vari relaxed PWHF}
    \widehat{\mathcal{I}}^\Theta(\theta)= \inf_{\theta}\cbr[2]{\int_0^T \widehat{L}(\theta, \dot{\theta})dt: \rho_{\theta}|_{t=0}=\rho_0, \rho_{\theta}|_{t=T}=\rho_T },
\end{align}
where $\widehat{L}$ is defined as 
\begin{equation}
    \widehat{L}(\theta, \dot{\theta})=\frac{1}{2}\dot\theta^\top \widehat{G}(\theta)   \dot\theta-F(\theta).
\end{equation}
The Euler-Lagrange equation for the variation formulation 
is
\begin{equation}
\label{relaxed theta 2nd order PWGF}
\widehat{G}(\theta)\ddot{\theta} + \sum_{k=1}^m \dot{\theta}_{k}\partial_{\theta_k}\widehat{G}(\theta) \dot{\theta} - \frac{1}{2} [\dot{\theta}^{\top} \partial_{\theta_k} \widehat{G}(\theta) \dot{\theta}]_{k=1}^{m} = -\nabla_{\theta}F(\theta).
\end{equation}

\end{theorem}

\begin{remark}
In the $1$-dimensional case, the simplified metric $\widehat{G}$ coincides with the exact matrix $G$, see \cite{liu2022neural} for a proof.
\end{remark}

The following theorems state conditions for the matrix $\widehat{G}$ to be invertible.

\begin{theorem}[Positive definiteness of $\widehat{G}$]\label{pos def relaxed metric}
The metric $\widehat{G}(\theta)$ defined in \eqref{relaxed metric tensor} is positive definite if and only if the $m$ vectors $\{\partial_{\theta_k} T_{\theta}: k=1, \cdots, m\}$ are linearly independent in $L^2(\mathbb{R}^d;\mathbb{R}^d,\lambda)$.
\end{theorem}

The proof of Theorem \ref{pos def relaxed metric} is trivial and hence omitted. In what follows, we also provide a sufficient condition for $\widehat{G}$ to be invertible.

  \begin{theorem} \label{metric positive definiteness}
  If the metric $G(\theta)$ defined in \eqref{metric tensor} is positive definite, then $\widehat{G}(\theta)$ is positive definite.
  \end{theorem}
  \begin{proof}
  Denote $w_i = \frac{\partial T_{\theta}}{\partial \theta_i}\circ T_{\theta}^{-1} - \nabla \psi_i$ where $-\Delta_{\rho_{\theta}} \psi_i = - \nabla \cdot (\rho_{\theta} \frac{\partial T_{\theta}}{\partial \theta_i}\circ T_{\theta}^{-1})$ for each $i=1,\dots,m$. Then $\nabla \cdot (\rho_{\theta}w_i) = 0$. Denote $W = [w_1,\dots,w_m] \in \mathbb{R}^{d\times m}$. For any $\thetadot \in \mathbb{R}^{m}$, we have 
  \begin{align*}
      \thetadot^{\top} \widehat{G}(\theta) \thetadot
      & = \thetadot^{\top} G(\theta) \thetadot + \sum_{i,j} \int \nabla \psi_i \cdot w_j \rhotheta dx \thetadot_i \thetadot_j + \int |W \thetadot|^2 \rhotheta dx \\
      & = \thetadot^{\top} G(\theta) \thetadot  + \int |W \thetadot|^2 \rhotheta dx \\
      & \ge \thetadot^{\top} G(\theta) \thetadot,
\end{align*}
where we used the fact $\int \nabla \psi_i \cdot w_j \rhotheta dx = - \int \psi_i \nabla \cdot (\rhotheta w_j) dx = 0$ for all $i,j$ in the second equality. Hence $G(\theta) \succ 0$ implies $\widehat{G}(\theta) \succ 0$.
  \end{proof} 
  
However, the converse of Theorem \ref{metric positive definiteness} is not necessarily true, as shown in the following counter example. 
  \begin{example}
  On $\mathbb{R}^4$, let us consider $T_{\theta}(x)=x+\theta_1 \vec{v}_1(x)+\theta_2\vec{v}_2(x)$, where $\vec{v}_1(x)=(-x_2, x_1, 0, 0)^\top$, $\vec{v}_2(x)=(0, 0, -x_4, x_3)^\top$ are two rotational fields. One can verify that $T_\theta$ is invertible for any $\theta=(\theta_1,\theta_2)$. We set reference $\lambda = \mathcal{N}(0,I_4)$. Direct calculation shows that $\frac{\partial T_{\theta}}{\partial \theta_1} = \vec{v}_1$, $\frac{\partial T_{\theta}}{\partial \theta_2} = \vec{v}_2$ are linearly independent. By Theorem \ref{pos def relaxed metric}, $\widehat{G}(\theta)$ is positive definite for any $\theta\in\mathbb{R}^2$. On the other hand, we examine the positive definiteness of $G(\theta)$ at $\theta=(0,0)$. To find it, we first compute $\rho_\theta = \mathcal{N}(0, \textrm{diag}((1+\theta_1^2)I_2,(1+\theta_2^2)I_2))$. Then we solve
    \begin{equation}
      -\nabla\cdot(\rho_\theta\nabla\psi_1(x)) = -\nabla\cdot(\rho_\theta\frac{\partial T_\theta}{\partial\theta_1}\circ T_\theta^{-1}(x)), \quad \theta=(0,0).\label{hodge decomposition at theta=0,0}
    \end{equation}
    The right hand side equals
    \begin{equation*}
      -\nabla\rho_\theta(x) \cdot \vec{v}_1(T_\theta^{-1}(x)) - \rho_\theta(x)\nabla\cdot\vec{v}_1(T_\theta^{-1}(x)).
    \end{equation*}
    Recall at $\theta=(0,0)$, $\rho_\theta = \mathcal{N}(0, I_4)$, we verify that both $\nabla\rho_{\theta}(x)\cdot\vec{v}_1(T_\theta^{-1}(x)) = 0$ and $\nabla\cdot \vec{v}_1(T_\theta^{-1}(x)) = 0$. Thus the right hand side of \eqref{hodge decomposition at theta=0,0} equals $0$, so does $\nabla\psi_1=0$. By similar argument, $\nabla\psi_2 = 0$. Therefore, the metric tensor $G((0,0))=O_2$, the $2 \times 2$ zero matrix, which is not positive definite.
  \end{example}

\subsection{Error bound for the continuous time PWHF}
\label{pseudo inverse error analysis}
In this subsection, we give error estimates on the continuous time dynamics \eqref{relaxed theta 2nd order PWGF}. We assume that the matrix $\widehat{G}$ has constant rank for $\theta\in \Theta$. The main results of this section are Theorem \ref{theorem: W2 error estimation} (error bound on $\rho$) and Theorem \ref{theorem: error estimation on Phi} (error bound on $\Phi$). We shall express \eqref{relaxed theta 2nd order PWGF} as a Hamiltonian system first.

Let us start by recalling some properties of the pseudo inverse operator for a positive semi-definite matrix.

As introduced before, $\widehat{G}(\theta)^{\dagger}$ is the Penrose-Moore pseudo inverse of the matrix $\widehat{G}(\theta)$. We write them as $\widehat{G}^{\dagger}$ and $\widehat{G}$ respectively for notation simplicity below. The Penrose-Moore pseudo inverse operator is a well-defined, one-to-one linear mapping.  
In addition, by \cite[Theorem 4.3]{doi:10.1137/0710036}, we know that if $\widehat{G}$ has constant rank, then
\begin{align}
\label{derivative pseudo inverse}
        \partial_{\theta_k} \widehat{G}^{\dagger}=-\widehat{G}^{\dagger}(\partial_{\theta_k} \widehat{G})\widehat{G}^{\dagger}+\widehat{G}^{\dagger}\widehat{G}^{\dagger}(\partial_{\theta_k}\widehat{G})(I-\widehat{G}\widehat{G}^{\dagger})+(I-\widehat{G}^{\dagger}\widehat{G})(\partial_{\theta_k} \widehat{G})\widehat{G}^{\dagger}\widehat{G}^{\dagger}.
\end{align}%
Further, Penrose-Moore pseudo inverse satisfies the following estimate for any $\eta \in \mathbb{R}^{m}$:
    \begin{align}
    \label{norm pseudo inverse}
        |\widehat{G}(\theta)^{\dagger} \eta |\leq \frac{1}{\lambda_{\min}(\widehat{G}(\theta))}|\eta|,
    \end{align}
where $|\cdot|$ is the standard Euclidean norm of vectors and $\lambda_{\min}(\widehat{G}(\theta))$ is the smallest \emph{nonzero} eigenvalue of $\widehat{G}(\theta)$.


\subsubsection{Parameterization of the potential energy} To obtain the error estimates, we need to express the potential energy in terms of the pushforward map, which is given in this subsection. 

If we have $\rho=T_{\sharp}\lambda$ for some reference density $\lambda$ and push-forward map $T$, the connection between $\rho$ and $T$ is explicitly given by
\begin{align}
    \label{eq: continuity eq for map}
    \rho(x)=\lambda\circ T^{-1}(x)\textrm{det}(\frac{d}{dx}T^{-1}(x)).
\end{align}

Consider the parameterized push-forward map $T_{\theta}$ as well as the density function $\rho_{\theta}=T_{\theta\sharp}\lambda$, we have
\begin{align}
    \left[\partial_{\theta}\rho_{\theta}+\textrm{div}_X \left(\rho_{\theta}\cdot\partial_{\theta}T_{\theta}\circ T_{\theta}^{-1}\right)\right]\circ T_{\theta}=0,
\end{align}
where $\textrm{div}_X$ denotes the divergence operator with respect to $x$.

In the following, we consider the variation of $\mathcal{F}(\rho)$, and denote its function value at $x$ as $\frac{\delta}{\delta \rho}\mathcal{F}(\rho, x)$. Assume $\frac{\delta}{\delta\rho}\mathcal{F}(\rho, \cdot)$ is smooth for each $\rho$, then for $F(\theta)=\mathcal{F}(\rho_{\theta})$ we can verify:
\begin{align}
\label{eq: para grad f}
    \nabla_{\theta} F(\theta) 
    & = \int \frac{\delta \mathcal{F}}{\delta \rho}(\rho_{\theta}(x),x) \partial_{\theta}\rho_{\theta}(x) \, dx \nonumber\\
    & = \int \frac{\delta \mathcal{F}}{\delta \rho}(\rho_{\theta}(T_{\theta}(z)),T_{\theta}(z)) \partial_{\theta}\rho_{\theta}(T_{\theta}(z)) \det(\nabla_z T_{\theta}(z)) \, dz \nonumber\\
    & = - \int \frac{\delta \mathcal{F}}{\delta \rho}(\rho_{\theta}(T_{\theta}(z)),T_{\theta}(z)) \textrm{div}_X\left[\rho_{\theta}(T_{\theta}(z))\partial_{\theta}T_{\theta}(z)\right] \det(\nabla_z T_{\theta}(z)) \, dz \nonumber\\
    & = - \int \frac{\delta \mathcal{F}}{\delta \rho}(\rho_{\theta}(x),x) \textrm{div}_X\left[\rho_{\theta}(x)\partial_{\theta}T_{\theta}\circ T_{\theta}^{-1}(x)\right]  \, dx \nonumber\\
    & = \int \nabla \frac{\delta \mathcal{F}}{\delta \rho}(\rho_{\theta}(T_{\theta}(z)),T_{\theta}(z))^\top \partial_{\theta}T_{\theta}(z) \rho_{\theta}(T_{\theta}(z))\det(\nabla_z T_{\theta}(z)) \, dz\nonumber \\
    &=\int \partial_{\theta}T_{\theta}(z)^\top \nabla \frac{\delta}{\delta\rho}\mathcal{F}( T_{\theta\sharp}\lambda(\cdot), \cdot)\circ T_{\theta}(z)~d\lambda(z).
\end{align}

\subsubsection{Simplified Hamiltonian dynamics with the pseudo inverse operator} \label{relaxedhd-pio}
In this part, we study properties of the parameterized dynamics \eqref{relaxed theta 2nd order PWGF} in detail. From the equivalence between the Lagrange and Hamiltonian mechanics, the second-order ODE \eqref{relaxed theta 2nd order PWGF} is equivalent to a first order Hamiltonian system with Hamiltonian
\begin{align}
    \label{pseudo Hamiltonian}
    H(\theta, p)=\frac{1}{2}p^{\top}\widehat{G}^{\dagger}p + F(\theta).
\end{align}

\begin{remark}
To derive the Hamiltonian system from the Lagrange mechanics, we require $\dot\theta\in \mathcal{R}(\widehat{G})$. This condition can be verified as long as the initial value $\dot\theta (0)$ lies in the range $\widehat{G}(\theta(0))$, which shall be proved in Lemma \ref{res derivative}.    
\end{remark}

In the remaining part of this section, we make the following assumption:
\begin{assumption}
\label{assumtion: pseudo inverse sv}
Assume that $\widehat{G}(\theta)$ is a smooth function of $\theta$ and $$\lambda_{\min, \Theta}:=\inf_{\theta\in \Theta}\lambda_{\min}(\widehat{G}(\theta))>0,$$ where $\lambda_{\min}(\widehat{G}(\theta))$ is the smallest positive eigenvalue of matrix $\widehat{G}(\theta)$. We also assume $$C_{\Theta}:=\sup_{\theta\in \Theta}\max_{k}\|\partial_{\theta_k}\widehat{G}(\theta)\|<\infty,$$ where $\|\partial_{\theta_k}\widehat{G}(\theta)\|$ is the standard matrix 2-norm of $\partial_{\theta_k}\widehat{G}(\theta)$.
\end{assumption}
\begin{remark}
    Under assumption \ref{assumtion: pseudo inverse sv}, the smallest nonzero singular of $\widehat{G}(\theta)$ is a smooth function of $\theta$, and $\lambda_{min, \Theta}>0$ implies that $\widehat{G}(\theta)$ is constant rank, hence \eqref{derivative pseudo inverse} holds.
\end{remark}

\begin{proposition}
\label{prop: hamiltonian eq exact}
    Under Assumption \ref{assumtion: pseudo inverse sv}, the following is the Hamiltonian dynamics of \eqref{pseudo Hamiltonian}:
    \begin{subequations}
\label{general PWGF pseudo}
    \begin{align}
        \dot\theta &= \widehat{G}^{\dagger}p, \label{eq:gPWHF-theta}\\
        \dot p &= \frac{1}{2}[(\widehat{G}^{\dagger}p)^{\top} (\partial_{\theta_k}\widehat{G})\widehat{G}^{\dagger}p]_{k=1}^{m}-\nabla_{\theta}F(\theta)-S(\theta, p), \label{eq:gPWHF-p}
    \end{align}
\end{subequations}
    where
    \begin{align}
    \label{eq:S-def}
        S(\theta, p)=\frac{1}{2}\left[p^{\top}\widehat{G}^{\dagger}\widehat{G}^{\dagger}(\partial_{\theta_k}\widehat{G})(I-\widehat{G}\widehat{G}^{\dagger})p+p^{\top}(I-\widehat{G}^{\dagger}\widehat{G})(\partial_{\theta_k}\widehat{G})\widehat{G}^{\dagger}\widehat{G}^{\dagger}p\right]_{k=1}^{m}.
    \end{align}
    Moreover, \eqref{general PWGF pseudo} is equivalent to \eqref{relaxed theta 2nd order PWGF}.
\end{proposition}
\begin{proof}
    First of all, we have
    \begin{align*}
        \frac{d}{dt}\theta&=\nabla_{p}H(\theta, p)=\widehat{G}^{\dagger}p, 
    \end{align*}
    which gives the equation in \eqref{eq:gPWHF-theta}. As for \eqref{eq:gPWHF-p}, we apply \eqref{derivative pseudo inverse} to
    \begin{align*}
        \frac{d}{dt}p=-\frac{1}{2}[p^{\top} \partial_{\theta_k}(\widehat{G}^{\dagger})p]_{k=1}^{m}-\nabla_{\theta}F(\theta)
    \end{align*}
\end{proof}


Similarly, we can also prove that the dynamics \eqref{theta 2nd order PWGF} is equivalent to the Hamiltonian system \eqref{general PWGF pseudo} with $\widehat{G}$ replaced by $G$ under the assumption that $G$ is constant rank. We note that the term $S(\theta, p)$ in \eqref{eq:S-def} is the extra term mentioned in Remark \ref{rmk:singular-G} if replacing $\widehat{G}$ by $G$. $S$ vanishes if $\widehat{G}$ or $G$ is invertible.
The Hamiltonian structure guarantees the boundedness of $|\dot{\theta}|$, which further allows us to give a prior estimation on the error.

\begin{lemma}
    \label{lemma: bounded dot theta}
    Assume the potential $F(\theta)$ can be bounded from below, i.e., $F_{\min} := \inf_{\theta\in\Theta}F(\theta)>-\infty$. Suppose $(\theta, p)$ is solved from \eqref{general PWGF pseudo} with initial values $(\theta(0),p(0))$, and we denote $H_0 = H(\theta(0),p(0))$. Under Assumption \ref{assumtion: pseudo inverse sv}, $|\dot{\theta}|$ can be uniformly upper bounded by
 \begin{equation}
 \label{ineq: bound dot theta}
 |\dot{\theta}|\leq \sqrt{\frac{2(H_0-F_{\min})}{\lambda_{\min,\Theta}}}. 
 \end{equation}
\end{lemma}
\begin{proof}
    Since the value of the Hamiltonian $H(\theta,p)$ is conserved for any time $t$ when $(\theta,p)$ solves the Hamiltonian system \eqref{general PWGF pseudo}, we have
\begin{equation}
\label{eq:H0}
  \frac{1}{2}\dot\theta^\top\widehat{G}(\theta)\dot\theta + F(\theta) = H(\theta(0),p(0)), \quad \textrm{for any } t\geq 0.
\end{equation}
By \eqref{norm pseudo inverse} and the fact that $\dot\theta=\widehat{G}^{\dagger}p\in \mathcal{R}(\widehat{G})$, we have
\begin{equation}
\label{eq:dottheta-bound}
    \frac{1}{2}\dot\theta^\top\widehat{G}(\theta)\dot\theta + F(\theta) \geq \frac{1}{2}\lambda_{\min}(\widehat{G}(\theta))|\dot\theta|^2 + F_{\min}\geq \frac{1}{2}\lambda_{\min,\Theta}|\dot\theta|^2 + F_{\min}.
\end{equation}
Combining \eqref{eq:H0} and \eqref{eq:dottheta-bound} yields \eqref{ineq: bound dot theta}.
\end{proof}

\subsubsection{Simplification of PWHFs}
In this subsection, we show that the Hamiltonian system \eqref{general PWGF pseudo} can be simplified even when $\widehat{G}$ is not invertible. Specifically, we shall show that $\widehat{G}(\theta(t))\dot{\theta}(t)-p(t)$ remains zero as long as its initial value is zero. In this case, we can also show that the term $S(\theta(t), p(t))=0$ for all $t$. To this end, we need to investigate properties of the metric $\widehat{G}$. From its definition \eqref{relaxed metric tensor}, we can see that $\widehat{G}$ is the inner product matrix for the functions $\{\partial_{\theta_k} T_{\theta}: k=1, \cdots, m\}$ in the $L^2(\mathbb{R}^d;\mathbb{R}^d,\lambda)$ space. We define $\mathcal{Q}^{\theta}:=\mathrm{span}\{\partial_{\theta_k} T_{\theta}(\cdot): k=1, \cdots, m\}$ to be the subspace of $L^2(\mathbb{R}^d;\mathbb{R}^d, \lambda)$. The following lemma gives the orthogonal projection operator from $L^2(\mathbb{R}^d;\mathbb{R}^d, \lambda)$ onto $\mathcal{Q}^{\theta}$.
\begin{lemma}
\label{lemma: K theta proj}
Define kernel $K_{\theta}(\cdot, \cdot):\mathbb{R}^d\times \mathbb{R}^d\rightarrow \mathcal{M}(\mathbb{R}^d)$, where $\mathcal{M}(\mathbb{R}^d)$ is the space of $d\times d $ matrices,
\begin{align}
    \label{kernel func}
         K_\theta(z', z) = \partial_{\theta} T_\theta(z') \widehat{G}^{\dagger}(\theta)  \partial_{\theta} T_\theta(z)^\top,
\end{align}
and the linear operator $\mathcal{K}_\theta$ on $L^2(\mathbb{R}^d;\mathbb{R}^d, \lambda)$ as 
  \begin{align}
  \label{kernel op}
           \mathcal{K}_\theta[f](\cdot) &= \partial_{\theta} T_\theta(\cdot)\widehat{G}^{\dagger}(\theta) \int \partial_{\theta} T_\theta(z)^\top f(z)~d\lambda(z)\\
           &=\int K_\theta(\cdot , z) f(z)~d\lambda(z).
  \end{align}
  where $f\in L^2(\mathbb{R}^d;\mathbb{R}^d, \lambda)$.
    Then the operator $\mathcal{K}_\theta$ is the orthogonal projection from $L^2(\mathbb{R}^d;\mathbb{R}^d, \lambda)$ onto $\mathcal{Q}^{\theta}\subset L^2(\mathbb{R}^d;\mathbb{R}^d, \lambda)$.
\end{lemma}
\begin{proof}
    Assume $\langle f, \partial_{\theta_k} T_{\theta}\rangle_{L^2(\lambda)}=0, $ for $k=1, \cdots, m$, then $\int \partial_{\theta} T_\theta(z)^\top f(z)~d\lambda(z)$ is the zero vector, as a result $\mathcal{K}_{\theta}[f]=0$.
    
    On the other side, for $\partial _{\theta_k}T_{\theta}\in \mathcal{Q}^{\theta}$ with $k=1, \cdots m$, the vector $ \vec{v}=\int\partial_{\theta} T_\theta(z)^\top \partial_{\theta_k}T_{\theta}(z)~d\lambda(z)$ is the $k$-th column vector of the matrix $\widehat{G}$, so $\vec{v}=\widehat{G}e_k$, which implies that we can decompose $e_k=\widehat{G}^{\dagger}\vec{v}+\eta$ with $\eta\perp \mathcal{R}(\widehat{G})$. We first claim that $\partial_{\theta}T_{\theta}(\cdot)\eta$ is the zero function. In fact we have
    \begin{align*}
        0&=\widehat{G}\eta=\int \partial_{\theta} T_\theta(z)^\top \partial _{\theta}T_{\theta}(z)\eta~d\lambda(z),
    \end{align*}
    which shows that $\partial _{\theta}T_{\theta}(z)\eta\in \mathcal{Q}^{\theta}$ is orthogonal to the linear space $\mathcal{Q}^{\theta}$, hence equals to zero.
    
    We can check
    \begin{align*}
        \mathcal{K}_{\theta}[\partial _{\theta_k}T_{\theta}](\cdot)&=\partial_{\theta} T_\theta(\cdot)\widehat{G}^{\dagger}(\theta) \int \partial_{\theta} T_\theta(z)^\top \partial _{\theta_k}T_{\theta}(z)~d\lambda(z)\\
        &=\partial_{\theta} T_\theta(\cdot)\widehat{G}^{\dagger} \vec{v}\\
        &=\partial_{\theta} T_\theta(\cdot)(e_k-\eta)=\partial _{\theta_k}T_{\theta}(\cdot)
    \end{align*}
    So $\mathcal{K}_\theta$ is the orthogonal projection from $L^2(\mathbb{R}^d;\mathbb{R}^d, \lambda)$ onto $\mathcal{Q}^{\theta}$.
  \end{proof}  
    
The following proposition is a direct result of Lemma \ref{lemma: K theta proj}:
\begin{proposition}
\label{prop: map in range}
    For any $f\in L^2(\mathbb{R}^d;\mathbb{R}^d, \lambda)$, we have
    \begin{align}
        \int \partial_{\theta} T_\theta(z)^\top f(z)~d\lambda(z)\in \mathcal{R}(\widehat{G}).
    \end{align}
\end{proposition}
\begin{proof}
By the property of projection operator $\mathcal{K}_{\theta}$, we know that $\mathcal{K}_{\theta}[f]=\partial_{\theta}T_{\theta}\gamma^*$ with
\begin{align}
\label{eq: opt gamma}
\gamma^*=\textrm{argmin}_{\gamma} \int \left|\partial_{\theta}T_{\theta}(z)\gamma-f(z)\right|^2~d\lambda(z).
\end{align}
Through the normal equation of \eqref{eq: opt gamma}, we have
\begin{align*}
    \int \partial_{\theta} T_\theta(z)^\top f(z)~d\lambda(z)=\left(\int \partial_\theta T_\theta(z)^\top \partial_\theta T_\theta(z)~d\lambda(z)\right)\gamma ^*=\widehat{G}(\theta)\gamma^*\in \mathcal{R}(\widehat{G}).
\end{align*}
\end{proof}

With Proposition \ref{prop: map in range} and the integral form \eqref{eq: para grad f} of $\nabla_{\theta}F(\theta)$, we conclude:
\begin{proposition}
\label{prop: grad f in range}
        For energy $\mathcal{F}$ with smooth $L^2$ first variation, we always have $\nabla_{\theta}F(\theta)\in \mathcal{R}(\widehat{G}(\theta))$.
\end{proposition}

Now we go back to the parameterized system \eqref{general PWGF pseudo}. Under Assumption \ref{assumtion: pseudo inverse sv}, $\{\theta\}$ is a $C^2$ curve on $\Theta$, hence we can define the following function $\Gamma^{\theta}(\cdot,t)\in L^2(\lambda)$ associated to $\theta$,
\begin{equation}
\label{def: Gamma func}
  \Gamma^{\theta}(z,t) = \sum_{k=1}^m\dot\theta_k\partial_{\theta_k}\partial_{\theta} T_{\theta}(z)\dot\theta+\partial_{\theta} T_{\theta}(z)\ddot\theta.
\end{equation}
The following lemma provides an explicit expression on the time derivative of $\widehat{G}(\theta)\dot\theta-p$.
\begin{lemma}
\label{res derivative}
    If $(\theta, p)$ solves the parameterized system \eqref{general PWGF pseudo},
    then there is 
    \begin{align}
        \frac{d}{dt}[\widehat{G}(\theta)\dot\theta-p]
        =\int \partial_{\theta} T_{\theta}(z)^\top  \left[\Gamma^{\theta}(z, t)+\nabla \frac{\delta}{\delta\rho}\mathcal{F}( T_{\theta\sharp}\lambda(\cdot), \cdot)\circ T_{\theta}(z)\right]~d\lambda(z)+S(\theta, p)
        \label{eq:dr}
    \end{align}
\end{lemma}
\begin{proof}
We first have
\begin{align*}
    \widehat{G}(\theta)\ddot\theta=\int \partial_{\theta} T_{\theta}(z)^\top  \partial_{\theta} T_{\theta}(z)\ddot\theta~d\lambda(z).
\end{align*}
In addition, there is
\begin{align*}
    \frac{d}{dt}(\widehat{G}(\theta)\dot\theta) & = \int \sum_{k=1}^m\dot\theta_k\partial_{\theta_k}\partial_{\theta} T_{\theta}(z)^\top\partial T_{\theta}(z)\dot\theta d\lambda(z) + \int \partial_{\theta} T_{\theta}(z)^\top  \left(\sum_{k=1}^m\dot\theta_k\partial_{\theta_k}\partial_{\theta} T_{\theta}(z)\dot\theta\right)~d\lambda(z)+\widehat{G}(\theta)\ddot\theta\\
    &=\frac{1}{2} [\dot{\theta}^{\top} (\partial_{\theta_k}\widehat{G})\dot{\theta}]_{k=1}^{m}+\int \partial_{\theta} T_{\theta}(z)^\top  \left(\sum_{k=1}^m\dot\theta_k\partial_{\theta_k}\partial_{\theta} T_{\theta}(z)\dot\theta+\partial_{\theta} T_{\theta}(z)\ddot\theta\right)~d\lambda(z)\\
    &=\frac{1}{2} [\dot{\theta}^{\top} (\partial_{\theta_k}\widehat{G})\dot{\theta}]_{k=1}^{m}+\int \partial_{\theta} T_{\theta}(z)^\top  \Gamma^{\theta}(z, t)~d\lambda(z),
\end{align*}
where
\begin{align*}
    \frac{1}{2} [\dot{\theta}^{\top} (\partial_{\theta_k}\widehat{G})\dot{\theta}]_{k=1}^{m}&=\frac{1}{2}\left[
    \int \dot\theta\partial_{\theta_k}\partial_{\theta} T_{\theta}(z)^\top\partial T_{\theta}(z)\dot\theta d\lambda(z)+\int \dot\theta\partial T_{\theta}(z)^\top\partial_{\theta_k}\partial_{\theta} T_{\theta}(z)\dot\theta d\lambda(z)
    \right]_{k=1}^m\\
    &=\Big[
    \int \dot\theta\partial_{\theta_k}\partial_{\theta} T_{\theta}(z)^\top\partial T_{\theta}(z)\dot\theta d\lambda(z)\Big]_{k=1}^m\\
    &=\Big[
    \int \partial_{\theta_k}\Big(\sum_{j=1}^m\dot\theta_j\partial_{\theta_j} T_{\theta}(z)\Big)^\top\partial T_{\theta}(z)\dot\theta d\lambda(z)\Big]_{k=1}^m\\
    &=\int \partial_{\theta}\Big(\sum_{j=1}^m\dot\theta_j\partial_{\theta_j} T_{\theta}(z)\Big)^\top\partial T_{\theta}(z)\dot\theta d\lambda(z)\\
    &=\int \sum_{j=1}^m\dot\theta_j\partial_{\theta_j} \partial_{\theta}T_{\theta}(z)^\top\partial T_{\theta}(z)\dot\theta d\lambda(z)
\end{align*}
Combining the results above and \eqref{eq:gPWHF-p}, we get
\begin{align*}
\frac{d}{dt}[\widehat{G}(\theta)\dot\theta-p]&=\frac{1}{2} [\dot{\theta}^{\top} (\partial_{\theta_k}\widehat{G})\dot{\theta}]_{k=1}^{m}+\int \partial_{\theta} T_{\theta}(z)^\top  \Gamma^{\theta}(z, t)~d\lambda(z)- \left(\frac{1}{2}[\dot{\theta}^{\top} (\partial_{\theta_k}\widehat{G})\dot{\theta}]_{k=1}^{m}-\nabla_{\theta}F(\theta)-S(\theta, p)\right)\\
&=\int \partial_{\theta} T_{\theta}(z)^\top  \left[\Gamma^{\theta}(z, t)+\nabla \frac{\delta}{\delta\rho}\mathcal{F}( T_{\theta\sharp}\lambda(\cdot), \cdot)\circ T_{\theta}(z)\right]~d\lambda(z)+S(\theta, p)
\end{align*}
which yields \eqref{eq:dr}.
\end{proof}

We are ready to establish the estimation on the magnitude of $\widehat{G}(\theta(t))\dot\theta(t)-p(t)$.
\begin{theorem}
\label{theorem: residual estimation}
    Let $(\theta, p )$ be the solution of \eqref{general PWGF pseudo} and denote $r(t):=\widehat{G}(\theta(t))\dot{\theta}(t)-p(t)$ and  $R(t):=|r(t)|^2$, then there is
    \begin{align}
    \label{eq:R-bound}
        R(t)\leq R(0)e^{C_St},
    \end{align}
    where $C_S:=2\sqrt{2m(H_0-F_{\min})}C_{\Theta}\lambda_{\min,\Theta}^{-3/2}$.
\end{theorem}
\begin{proof}
Taking the time derivative of $R(t)$, we obtain
\begin{equation}
\label{res norm decomp}
\begin{aligned}
        \frac{d}{dt}R(t)&=2\langle r(t), \frac{d}{dt}r(t)\rangle\\
    &=2\Big\langle \widehat{G}\dot{\theta}-p, \int \partial_{\theta} T_{\theta}(z)^\top  \left[\Gamma^{\theta}(z, t)+\nabla \frac{\delta}{\delta\rho}\mathcal{F}( T_{\theta\sharp}\lambda(\cdot), \cdot)\circ T_{\theta}(z)\right]~d\lambda(z)+S(\theta, p)\Big\rangle\ \\
    &=2\Big\langle \widehat{G}\dot{\theta}-p, \int \partial_{\theta} T_{\theta}(z)^\top  \left[\Gamma^{\theta}(z, t)+\nabla \frac{\delta}{\delta\rho}\mathcal{F}( T_{\theta\sharp}\lambda(\cdot), \cdot)\circ T_{\theta}(z)\right]~d\lambda(z)\Big\rangle+2\langle \widehat{G}\dot{\theta}-p, S(\theta, p)\rangle.
\end{aligned}
\end{equation}
By Proposition \ref{prop: map in range}, we know that $\int \partial_{\theta} T_{\theta}(z)^\top  [\Gamma^{\theta}(z, t)+ \nabla \frac{\delta}{\delta\rho}\mathcal{F}( T_{\theta\sharp}\lambda(\cdot), \cdot)\circ T_{\theta}(z)]~d\lambda(z) \in \mathcal{R}(\widehat{G}(\theta))$. Due to the property of pseudo inverse operator, $\widehat{G}(\theta(t))\dot{\theta}(t)-p(t)$ is orthogonal to the subspace $\mathcal{R}(\widehat{G}(\theta))$, hence the first term in \eqref{res norm decomp} vanishes. As a result, we can simply write:
\begin{align}
\label{res norm derivative}
    \frac{d}{dt}R(t)=\langle \widehat{G}\dot{\theta}-p, S(\theta, p)\rangle .
\end{align}
By Lemma \ref{lemma: bounded dot theta}, we know that $\dot{\theta}$ is uniformly bounded, so we can show that
\begin{align*}
    \left|[p^{\top}\widehat{G}^{\dagger}\widehat{G}^{\dagger}(\partial_{\theta_k}\widehat{G})(I-\widehat{G}\widehat{G}^{\dagger})p]_{k=1}^m\right|&=\left|[\dot{\theta}\widehat{G}^{\dagger}(\partial_{\theta_k}\widehat{G})(p-\widehat{G}\dot{\theta})]_{k=1}^m\right|\\
    &\leq \frac{\sqrt{m}|\dot{\theta}|}{\lambda_{ \min}(\widehat{G})}\max_{k}\|\partial_{\theta_k}\widehat{G}\|\cdot |\widehat{G}\dot{\theta}-p|\\
    &\leq \frac{1}{2}C_S\cdot|\widehat{G}\dot{\theta}-p|.
\end{align*}
Similarly, we can also bound $\left|[p^{\top}(I-\widehat{G}^{\dagger}\widehat{G})(\partial_{\theta_k}\widehat{G})\widehat{G}^{\dagger}\widehat{G}^{\dagger}p]_{k=1}^{m}\right|$ by $\frac{1}{2}C_S\cdot|\widehat{G}\dot{\theta}-p|$. These two bounds and the definition of $S(\theta,p)$ in \eqref{eq:S-def} yield
\begin{align}
\label{ineq: S bound}
    |S(\theta, p)|\leq \frac{1}{2}C_S\cdot|\widehat{G}\dot{\theta}-p|.
\end{align}
Combining \eqref{res norm derivative} and \eqref{ineq: S bound}, we get:
\begin{align*}
    \frac{d}{dt}R(t)&\leq 2|\widehat{G}\dot{\theta}-p|\cdot |S(\theta, p)|\leq C_S\cdot|\widehat{G}\dot{\theta}-p|^2 =C_S R(t).
\end{align*}
Applying Gronwall's inequality yields \eqref{eq:R-bound}.
\end{proof}

An immediate consequence of Theorem \ref{theorem: residual estimation} is that $\widehat{G}(\theta(t))\dot\theta(t)-p(t)=0$ for all $t$ as long as $p(0)=\widehat{G}(\theta(0))\dot{\theta}(0)$ which implies $R(0)=0$. This result is summarized in the following proposition, whose proof is omitted.
\begin{proposition}
    Under the same assumption as in Theorem \ref{theorem: residual estimation}, and set $p(0)=\widehat{G}(\theta(0))\dot{\theta}(0)$ for any $\dot{\theta}(0)\in \mathbb{R}^m$. Then system \eqref{general PWGF pseudo} is equivalent to:
        \begin{subequations}
\label{general PWGF pseudo simplified}
    \begin{align}
        \dot\theta &= \widehat{G}^{\dagger}p,\\
        \dot p &= \frac{1}{2}[(\widehat{G}^{\dagger}p)^{\top} (\partial_{\theta_k}\widehat{G})\widehat{G}^{\dagger}p]_{k=1}^{m}-\nabla_{\theta}F(\theta).
    \end{align}
\end{subequations}
\end{proposition}
We denote the ODE system \eqref{general PWGF pseudo simplified} as \textit{Parameterized Wasserstein Hamiltonian Flow with simplified metric}, or for short as PWHF.

\subsubsection{Error Analysis on the PWHF}
In Section \ref{Background WHF}, we discussed the particle WHF. Here we consider its counterpart for the PWHF, i.e., the particle level dynamics induced by our parameterized dynamics in parameter space:
\begin{equation}
  \ddot{\boldsymbol{Y}}^\Theta = \frac{d}{dt}\left(\partial_{\theta} T_{\theta}(z_0)\dot\theta\right) =  \sum_{k=1}^m\dot\theta_k\partial_{\theta_k}\partial_{\theta} T_{\theta}(z)\dot\theta+\partial_{\theta} T_{\theta}(z)\ddot\theta= \Gamma^{\theta}(T_{\theta}^{-1}(\boldsymbol{Y}^\Theta), t), \label{particle ode 2}
\end{equation}
where $\theta$ satisfies the ODE \eqref{general PWGF pseudo simplified}. This Lagrangian perspective enables us to carry out the error analysis results of PWHF. To be more specific, for a fixed initial position $z_0$, we first estimate the difference between the vector fields that drive $\boldsymbol{X}$ and $\boldsymbol{Y}^\Theta$. The result leads to an estimation on the $l^2$ distance between $\boldsymbol{X}$ and $\boldsymbol{Y}^\Theta$ for $t>0$. To achieve this goal, we introduce the particle level dynamics of the parameterized system, and provide several useful lemmas.
Here we assume the function $\frac{\delta}{\delta\rho} \mathcal{F}$ to be Lipschitz continuous in the $L^2(\mathbb{R}^d;\mathbb{R}^d, \lambda)$ sense as follows.
\begin{assumption}
\label{assumption: lip constant of delfta F}
There exists a constant $C_{\mathcal{F}}$ such that for any two push-forward maps $T$ and $ \tilde{T}$ there is
\begin{align}
        \int \Big|\nabla \frac{\delta}{\delta\rho}\mathcal{F}( T_{\sharp}\lambda(\cdot), \cdot)\circ T(z)-\nabla \frac{\delta}{\delta\rho}\mathcal{F}( \Tilde{T}_{\sharp}\lambda(\cdot), \cdot)\circ \tilde{T}(z) \Big|^2~d\lambda(z)\leq C_{\mathcal{F}}\int|T(z)-\tilde{T}(z)|^2~d\lambda(z).
\end{align}
\end{assumption}


\begin{lemma}[2nd order dynamic of $\boldsymbol{Y}^\Theta$]\label{dym of Y Theta t}
Let $\{\theta\}$ be a $C^2$ curve on $\Theta$. Assume
\begin{equation}
  \ddot{\boldsymbol{Y} }^\Theta(t) = \Gamma^{\theta}(T_{\theta(t)}^{-1}(\boldsymbol{Y}^\Theta(t)), t), \quad \boldsymbol{Y}^\Theta(0) = T_{\theta(0)}(z_0), \quad \dot{\boldsymbol{Y}}^\Theta(0) = \partial T_{\theta(0)}(z_0)\dot\theta(0) ,  \label{particle auto ode}
\end{equation}
admits a unique solution for any $z_0\in\mathbb{R}^d$. Then $\boldsymbol{Y}^{\Theta}(t)=T_{\theta(t)}(z_0)$ for $t\geq 0$.
\end{lemma}

\begin{proof}[Proof of Lemma \ref{dym of Y Theta t}]
We denote $\boldsymbol{X}^\Theta(t) = T_{\theta(t)}(z_0)$. Taking time derivative of $\boldsymbol{X}^\Theta$ twice, we obtain
\begin{equation}
  \ddot{\boldsymbol{X}}^\Theta = \frac{d}{dt}\left(\partial_{\theta} T_{\theta(t)}(z_0)\dot\theta(t)\right) =  \Gamma^{\theta}(T_{\theta(t)}^{-1}(\boldsymbol{X}^\Theta), t). \label{particle ode}
\end{equation}
We can verify that $\boldsymbol{X}^\Theta(0) = T_{\theta(0)}(z_0)$ and $\dot{\boldsymbol{X}}^\Theta(0)=\partial_{\theta} T_{\theta(0)}(z_0)\dot\theta(0)$. According to the uniqueness of ODE solution, we know $\boldsymbol{Y}^\Theta(t)=\boldsymbol{X}^\Theta(t) = T_{\theta(t)}(z_0)$ for any $t\geq 0$.
\end{proof}

The following lemma decomposes the dynamics \eqref{particle auto ode} into three parts, which provides us a way to estimate the difference between the particle level dynamics \eqref{particle dynamics} and \eqref{particle auto ode}. 
\begin{lemma}
\label{lemma: parameterized 2nd order dynamics decomp}
    Under the same assumptions as in Theorem \ref{theorem: residual estimation} with initial $p(0)=\widehat{G}(\theta(0))\dot{\theta}(0)$, we can decompose the second-order particle level dynamics \eqref{particle auto ode} as
\begin{equation}
\label{Y dyn decompose}
\begin{aligned}
    \ddot{\boldsymbol{Y}}^\Theta &= \left(\Gamma^{\theta}(T_{\theta}^{-1}(\boldsymbol{Y}^\Theta), t)-\mathcal{K}_{\theta}[\Gamma^{\theta}(\cdot, t)](T_{\theta}^{-1}(\boldsymbol{Y}^\Theta))\right)\\
    &\quad +\left(\nabla \frac{\delta}{\delta\rho}\mathcal{F}( T_{\theta\sharp}\lambda(\boldsymbol{Y}^\Theta), \boldsymbol{Y}^\Theta) - \mathcal{K}_\theta \Big[ \frac{\delta}{\delta\rho}\mathcal{F}( T_{\theta\sharp}\lambda(\cdot), \cdot)\circ  T_{\theta} \Big] (T_{\theta}^{-1}(\boldsymbol{Y}^\Theta))\right) - \nabla \frac{\delta}{\delta\rho}\mathcal{F}( T_{\theta\sharp}\lambda(\boldsymbol{Y}^\Theta), \boldsymbol{Y}^\Theta),
\end{aligned}
\end{equation}
which can also be written as a second-order system
\begin{equation}
    \begin{aligned}
       & \dot{\boldsymbol{Y}}^\Theta = \boldsymbol{P}^\Theta, \\
   & \dot{\boldsymbol{P}}^\Theta = (\textrm{Id} - \mathcal{K}_{\theta})[\Gamma^{\theta}(\cdot, t)+\nabla \frac{\delta}{\delta\rho}\mathcal{F}( T_{\theta\sharp}\lambda(\cdot), \cdot)\circ T_{\theta}(\cdot)](T_{\theta}^{-1}(\boldsymbol{Y}^\Theta)) - \nabla \frac{\delta}{\delta\rho}\mathcal{F}( T_{\theta\sharp}\lambda(\boldsymbol{Y}^\Theta), \boldsymbol{Y}^\Theta).
    \end{aligned}
\end{equation}
\end{lemma}
\begin{proof}
    We can rewrite $\ddot{\boldsymbol{Y}}^\Theta$ as:
\begin{equation}
\label{Y dyn decompose four}
\begin{aligned}
    \ddot{\boldsymbol{Y}}^\Theta =& \left(\Gamma^{\theta}(T_{\theta}^{-1}(\boldsymbol{Y}^\Theta), t)-\mathcal{K}_{\theta}[\Gamma^{\theta}(\cdot, t)](T_{\theta}^{-1}(\boldsymbol{Y}^\Theta))\right)\\
    &+\left(\nabla \frac{\delta}{\delta\rho}\mathcal{F}( T_{\theta\sharp}\lambda(\boldsymbol{Y}^\Theta), \boldsymbol{Y}^\Theta) - \mathcal{K}_\theta[\nabla \frac{\delta}{\delta\rho}\mathcal{F}( T_{\theta\sharp}\lambda(\cdot), \cdot)\circ  T_{\theta}](T_{\theta}^{-1}(\boldsymbol{Y}^\Theta))\right) - \nabla \frac{\delta}{\delta\rho}\mathcal{F}( T_{\theta\sharp}\lambda(\boldsymbol{Y}^\Theta), \boldsymbol{Y}^\Theta)\\
    &+\left(\mathcal{K}_\theta[\Gamma^{\theta}(\cdot, t)+\nabla \frac{\delta}{\delta\rho}\mathcal{F}( T_{\theta\sharp}\lambda(\cdot), \cdot)\circ  T_{\theta}](T_{\theta}^{-1}(\boldsymbol{Y}^\Theta))\right).
\end{aligned}
\end{equation}
By Lemma \ref{res derivative}, we know that
\begin{align*}
        \int \partial_{\theta} T_{\theta}(z)^\top  \left[\Gamma^{\theta}(z, t)+\nabla \frac{\delta}{\delta\rho}\mathcal{F}( T_{\theta\sharp}\lambda(\cdot), \cdot)\circ T_{\theta}(z)\right]~d\lambda(z)=\frac{d}{dt}[\widehat{G}(\theta)\dot\theta-p]-S(\theta, p).
\end{align*}
By Theorem \ref{theorem: residual estimation}, we have that both $\frac{d}{dt}[\widehat{G}(\theta)\dot\theta-p]$ and $S(\theta, p)$ equal to zero. Thus,
\begin{align*}
    \int \partial_{\theta} T_{\theta}(z)^\top  \left[\Gamma^{\theta}(z, t)+\nabla \frac{\delta}{\delta\rho}\mathcal{F}( T_{\theta\sharp}\lambda(\cdot), \cdot)\circ T_{\theta}(z)\right]~d\lambda(z)=0.
\end{align*}
Then we can compute:
\begin{align*}
    &\mathcal{K}_\theta[\Gamma^{\theta}(\cdot, t)+\nabla \frac{\delta}{\delta\rho}\mathcal{F}( T_{\theta\sharp}\lambda(\cdot), \cdot)\circ  T_{\theta}](T_{\theta}^{-1}(\boldsymbol{Y}^\Theta))\\
    =\ &\partial_{\theta} T_\theta(T_{\theta}^{-1}(\boldsymbol{Y}^\Theta))\widehat{G}^{\dagger} \int \partial_{\theta} T_\theta(z)^\top \left[\Gamma^{\theta}(z, t)+\nabla \frac{\delta}{\delta\rho}\mathcal{F}( T_{\theta\sharp}\lambda(\cdot), \cdot)\circ T_{\theta}(z)\right]~d\lambda(z)\\
    =\ &0.
\end{align*}
Plugging the above identity into \eqref{Y dyn decompose four}, we obtain \eqref{Y dyn decompose}.
\end{proof}

To measure the first two terms in \eqref{Y dyn decompose}, we introduce two quantities that characterize the approximation power of the push-forward map $T_{\theta}$:

\begin{align}
  \delta_0 &=\sup_{\theta\in \Theta}\min_{\zeta\in\mathcal{T}_{\theta}^*\Theta}\left\{ \int|\nabla \frac{\delta}{\delta\rho}\mathcal{F}( T_{\theta\sharp}\lambda(\cdot), \cdot) \circ T_{\theta}(z) - \partial_{\theta}T_{\theta}(z)\zeta|^2~d\lambda(z) \right\}\nonumber\\
  &=\sup_{\theta\in\Theta} \left\{ \int|\nabla \frac{\delta}{\delta\rho}\mathcal{F}( T_{\theta\sharp}\lambda(\cdot), \cdot) \circ T_{\theta}(z) - \mathcal{K}_\theta[\nabla \frac{\delta}{\delta\rho}\mathcal{F}( T_{\theta\sharp}\lambda(\cdot), \cdot)\circ T_{\theta} ](z)|^2~d\lambda(z) \right\}, \label{pseudo delta 0 def}
\end{align}
and
\begin{align}
\label{def: delta 1 def}
    \delta_1
    &=\sum_{1\leq i,j\leq m}  \sup_{\theta\in \Theta}\min_{\zeta\in\mathcal{T}_{\theta}^*\Theta}\left\{\int| \partial_{\theta_i}\partial_{\theta_j}T_{\theta}(z) - \partial_{\theta}T_{\theta}(z)\zeta|^2~d\lambda(z)\right\}\nonumber\\
    &=\sum_{1\leq i,j\leq m} \sup_{\theta\in\Theta}\left\{\int| \partial_{\theta_i}\partial_{\theta_j}T_{\theta}(z) - \mathcal{K}_\theta[\partial_{\theta_i}\partial_{\theta_j}T_{\theta}](z)|^2~d\lambda(z)\right\}.
\end{align}
Another quantity $\delta_2$ measures how well the initial tangent space $\mathcal{Q}^{\theta}$ approximates the initial velocity:
\begin{align}
    \delta_2&=\min_{\zeta\in\mathcal{T}_{\theta}^*\Theta}\int \left|\nabla\Phi(0, T_{\theta(0)}(z)) - \partial_{\theta}T_{\theta(0)}(z)\zeta\right|^2~d\lambda(z)\nonumber\\
    &=\int \left|\nabla\Phi(0, T_{\theta(0)}(z)) - \mathcal{K}_{\theta(0)}[\nabla\Phi(0, T_{\theta(0)}(\cdot))](z)\right|^2~d\lambda(z).
\end{align}

Now we are ready to provide an upper bound on the difference between $\rho_{\theta}$ of the PWHF and $\rho$ of the original WHF in the $W_2$ sense based on the values of $\delta_0$, $\delta_1$, and $\delta_2$. 

\begin{theorem}[Error estimation on $\rho$]\label{theorem: W2 error estimation}
Let $(\rho, \Phi)$ be the solution of WHF \eqref{WHF in dual coordinates} with given initial value $(\rho_0,\Phi_0)$ on time interval $[0, t_0)$. Suppose $p(0)=\int  \partial_{\theta}T_{\theta(0)}^{\top}(z)\nabla \Phi(0, T_{\theta(0)}(z))d\lambda (z)$ and $(\theta, p)$ is the solution of PWHF \eqref{general PWGF pseudo} with initial value $(\theta(0), p(0))$
Assume $C_{\nabla \Phi_0}:=\mathrm{Lip}(\nabla\Phi_0)<\infty$ and denote $\epsilon_\rho=W_2^2(\rho_{\theta(0)}, \rho_0)$ as the initial approximation error.  
  Then under Assumption \ref{assumtion: pseudo inverse sv} and \ref{assumption: lip constant of delfta F}, there is
  \begin{equation}
    W_2^2(\rho_{\theta(t)}, \rho_t )  
    \leq e^{Ct}\left((1+2C_{\nabla\Phi_0}^2)\epsilon_{\rho} + 2\delta_2\right) + \left(\frac{3\delta_0}{C}+\frac{12 \delta_1(H_0-F_{\min})^2}{C\lambda_{\min,\Theta}^2}\right)(e^{Ct}-1).  \quad \textrm{for } ~  0\leq t <  t_0,\label{err est}
  \end{equation}
  where $C := 2+3C_{\mathcal{F}}^2$, $\delta_0$ and $\delta_1$ are defined in \eqref{pseudo delta 0 def} and \eqref{def: delta 1 def} respectively, and $H_0$ and $F_{\min}$ are defined in Lemma \ref{lemma: bounded dot theta}.
\end{theorem}
\begin{proof}
Let $\boldsymbol{X}$ be the process satisfying the system \eqref{particle dynamics}, i.e., $\boldsymbol{X}$ solves the ODE
\begin{equation}
\label{def: x dynamics}
  \ddot{\boldsymbol{X}} = -\nabla_{X}\frac{\delta}{\delta \rho(t, \boldsymbol{X})}\mathcal{F}(\rho), \quad \boldsymbol{X}(0)\sim\rho_0, \quad \dot{\boldsymbol{X}}(0) = \nabla_p H(\boldsymbol{X}(0), \nabla\Phi(0, \boldsymbol{X}(0))) = \nabla\Phi(0, \boldsymbol{X}(0)),
\end{equation}
where $\rho_0(\cdot)$ and $\Phi_0(\cdot) = \Phi(0, \cdot)$ are the initial conditions of the WHF. 
We can verify that $\textrm{Law}(\boldsymbol{X}(t))=\rho(t, \cdot)$ for all $t$.

On the other hand, for $\{\theta\}_{t\in [0, t_0)}$, we consider the vector field $\Gamma^{\theta}(\cdot, t)$ defined in \eqref{def: Gamma func} and another dynamic
\begin{equation}
  \ddot{\boldsymbol{Y}}(t) = \Gamma^{\theta}(T_{\theta(t)}^{-1}(\boldsymbol{Y}(t)), t), \quad \boldsymbol{Y}(0) = T_{\theta(0)}(x_0), \quad \dot{\boldsymbol{Y}}(0) = \partial_{\theta} T_{\theta(0)(x_0)}\dot{\theta}(0),  \label{ODE xi}
\end{equation}
where $x_0 = \boldsymbol{X}(0)$. From the definition of $p(0)$ and proposition \ref{prop: map in range}, we know that $p(0)\in\mathcal{R}(\widehat{G}(\theta(0)))$, hence we can apply the decomposition \eqref{Y dyn decompose} for $\boldsymbol{Y}$.

Suppose the Monge map from $\rho_{\theta(0)}=T_{\theta(0)\sharp}\lambda$ to $\rho_0$, which exists and is unique under the 2-Wasserstein metric, is given by $\omega$ and we assume the random variables $\boldsymbol{X}, \boldsymbol{Y}$ are coupled via $\boldsymbol{X}(0) = \omega(\boldsymbol{Y}(0))$.

Now consider the expected $l^2$ distance between $(\boldsymbol{X},\dot{\boldsymbol{X}})$ and $(\boldsymbol{Y},\dot{\boldsymbol{Y}})$
\begin{equation}
  E(t) := \mathbb{E}|(\boldsymbol{X},\dot{\boldsymbol{X}}) -(\boldsymbol{Y},\dot{\boldsymbol{Y}}) |^2.  \label{def E(t) }
\end{equation}
Taking time derivative gives us
\begin{align}
    \frac{d}{dt} E(t) & = 2\mathbb{E} ((\boldsymbol{X},\dot{\boldsymbol{X}}) -(\boldsymbol{Y},\dot{\boldsymbol{Y}}))\cdot ((\dot{\boldsymbol{X}}, \ddot{\boldsymbol{X}})-(\dot{\boldsymbol{Y}}, \ddot{\boldsymbol{Y}}))  \nonumber\\
    & \leq 2\sqrt{\mathbb{E}|(\boldsymbol{X},\dot{\boldsymbol{X}})-(\boldsymbol{Y},\dot{\boldsymbol{Y}})|^2} \sqrt{\mathbb{E} |\dot{\boldsymbol{X}}-\dot{\boldsymbol{Y}}|^2+\mathbb{E}|\ddot{\boldsymbol{X}} - \ddot{\boldsymbol{Y}}|^2 }  \nonumber \\
    & \leq 2 \sqrt{E(t)} \sqrt{E(t)+\mathbb{E}|\ddot{\boldsymbol{X}} - \ddot{\boldsymbol{Y}}|^2} \nonumber\\
    &\leq 2E(t)+\mathbb{E}|\ddot{\boldsymbol{X}} - \ddot{\boldsymbol{Y}}|^2. \label{ineq: Et graonwal inequa}
\end{align}
From the fact that $\mathcal{K}_{\theta}[\partial_{\theta} T_{\theta}\dot\theta](z)=\partial_{\theta} T_{\theta}(z)\dot\theta$, we can check
\begin{align*}
    &\ \mathbb{E}|\Gamma^{\theta}(T_{\theta}^{-1}(\boldsymbol{Y}), t)-\mathcal{K}_{\theta}[\Gamma^{\theta}(\cdot, t)](T_{\theta}^{-1}(\boldsymbol{Y}))|^2\\
    =&\ \int \Big|\sum_{k=1}^m\dot\theta_k\partial_{\theta_k}\partial_{\theta} T_{\theta}(z)\dot\theta+\partial_{\theta} T_{\theta}(z)\ddot\theta-\mathcal{K}_{\theta}\Big[\sum_{k=1}^m\dot\theta_k\partial_{\theta_k}\partial_{\theta} T_{\theta}(z)\dot\theta+\partial_{\theta} T_{\theta}(z)\ddot\theta\Big](z) \Big|^2 \, d\lambda(z)\\
    =&\ \int \Big|\sum_{k=1}^m\dot\theta_k\partial_{\theta_k}\partial_{\theta} T_{\theta}(z)\dot\theta-\mathcal{K}_{\theta}\Big[\sum_{k=1}^m\dot\theta_k\partial_{\theta_k}\partial_{\theta} T_{\theta}(z)\dot\theta\Big](z) \Big|^2 \, d\lambda(z)\\
    \leq &\ |\dot{\theta}|^4 \sum_{1\leq i, j\leq m}\int |\partial_{\theta_i}\partial_{\theta_j} T_{\theta}(z)-\mathcal{K}_{\theta}[\partial_i\partial_{\theta_j} T_{\theta}](z)|^2~d\lambda(z)\\
    =&\delta_1 |\dot{\theta}|^4,
\end{align*}

By the definition of $\delta_0$, we have
\begin{align*}
    &\ \mathbb{E}\left|\nabla \frac{\delta}{\delta\rho}\mathcal{F}( T_{\theta\sharp}\lambda(\boldsymbol{Y}), \boldsymbol{Y}) - \mathcal{K}_\theta[\nabla \frac{\delta}{\delta\rho}\mathcal{F}( T_{\theta\sharp}\lambda(\cdot), \cdot)\circ  T_{\theta}](T_{\theta}^{-1}(\boldsymbol{Y}))\right|^2\\
    =& \ \int|\nabla \frac{\delta}{\delta\rho}\mathcal{F}( T_{\theta\sharp}\lambda(\cdot), \cdot)\circ T_{\theta}(z) - \mathcal{K}_\theta[\nabla \frac{\delta}{\delta\rho}\mathcal{F}( T_{\theta\sharp}\lambda(\cdot), \cdot)\circ  T_{\theta}](z)|^2~d\lambda(z) \\
    \leq &\ \delta_0.
\end{align*}
By Assumption \ref{assumption: lip constant of delfta F} on the potential $\mathcal{F}(\rho)$, and notice that $\boldsymbol{X}$ is a push-forward of $\boldsymbol{X}(0)$ through the dynamics \eqref{def: x dynamics}, we have
\begin{align}
    \mathbb{E} \Big|\nabla \frac{\delta}{\delta\rho}\mathcal{F}( T_{\theta\sharp}\lambda(\boldsymbol{Y}), \boldsymbol{Y})-\nabla \frac{\delta}{\delta\rho}\mathcal{F}( \rho(\boldsymbol{X}), \boldsymbol{X})\Big|^2\leq C_{\mathcal{F}}\mathbb{E}|\boldsymbol{Y}-\boldsymbol{X}|^2.
\end{align}
Combining all last three inequalities and applying the decomposition in Lemma \ref{lemma: parameterized 2nd order dynamics decomp}, we obtain
\begin{align*}
    \mathbb{E}|\ddot{\boldsymbol{X}} - \ddot{\boldsymbol{Y}}|^2&=\mathbb{E} \Big|\left(\Gamma^{\theta}(T_{\theta}^{-1}(\boldsymbol{Y}), t)-\mathcal{K}_{\theta}[\Gamma^{\theta}(\cdot, t)](T_{\theta}^{-1}(\boldsymbol{Y}))\right)\\
    &\qquad+\left(\nabla \frac{\delta}{\delta\rho}\mathcal{F}( T_{\theta\sharp}\lambda(\boldsymbol{Y}), \boldsymbol{Y}) - \mathcal{K}_\theta[\nabla \frac{\delta}{\delta\rho}\mathcal{F}( T_{\theta\sharp}\lambda(\cdot), \cdot)\circ  T_{\theta}](T_{\theta}^{-1}(\boldsymbol{Y}))\right) \\
    &\qquad-\left(\nabla \frac{\delta}{\delta\rho}\mathcal{F}( T_{\theta\sharp}\lambda(\boldsymbol{Y}), \boldsymbol{Y})-\nabla \frac{\delta}{\delta\rho}\mathcal{F}( \rho(\boldsymbol{X}), \boldsymbol{X})\right) \Big|^2\\
    &\leq 3\left(\delta_0+\delta_1 |\dot{\theta}|^4+C_{\mathcal{F}}\mathbb{E}|\boldsymbol{Y}-\boldsymbol{X}|^2\right)\\
    &\leq 3\left(\delta_0+4\delta_1 \frac{(H_0-F_{\min})^2}{\lambda_{ \min,\Theta}^2}+C_{\mathcal{F}}\mathbb{E}|\boldsymbol{Y}-\boldsymbol{X}|^2\right).
\end{align*}
Continuing from \eqref{ineq: Et graonwal inequa}, we compute
\begin{align}
    \frac{d}{dt}E(t)&\leq 2E(t)+3\left(\delta_0+4\delta_1 \frac{(H_0-F_{ \min})^2}{\lambda_{\min,\Theta}^2}+C_{\mathcal{F}}\mathbb{E}|\boldsymbol{Y}-\boldsymbol{X}|^2\right)\\
    &\leq 3\delta_0+12\delta_1 \frac{(H_0-F_{\min})^2}{\lambda_{\min,\Theta}^2}+(2+3C_{\mathcal{F}})E(t).
\end{align}
Recalling $C=2+3C_{\mathcal{F}}$ and applying Gronwall's inequality, we arrive at
\begin{equation}
   E(t)\leq e^{Ct}E(0) + \left(\frac{3\delta_0}{C}+\frac{12 \delta_1(H_0-F_{\min})^2}{C\lambda_{\min,\Theta}^2}\right)(e^{Ct}-1). \label{inequa by granwal}
\end{equation}
Now we estimate the initial error
\begin{align}
    E(0) = &\  \mathbb{E}|\boldsymbol{X}(0)-\boldsymbol{Y}(0)|^2 + \mathbb{E}|\dot{\boldsymbol{X}}(0)-\dot{\boldsymbol{Y}}(0)|^2 \\
    = &\ \mathbb{E}_{z_0\sim \lambda}|\omega(T_{\theta(0)}(z_0))-T_{\theta(0)}(z_0)|^2 + \mathbb{E}_{z_0\sim \lambda}\left|\nabla\Phi(0, \omega(T_{\theta(0)}(z_0))) - \partial_{\theta} T_{\theta(0)}(z_0)\dot\theta(0)\right|^2.
\end{align}
Since $\omega$ is the Monge map from $\rho_{\theta(0)}$ to $\rho_0$, the term $\mathbb{E}|\omega(T_{\theta(0)}(z_0))-T_{\theta(0)}(z_0)|^2=W_2^2(\rho_{\theta(0)}, \rho_0)=\epsilon_{\rho}$. For the second term above, we have
\begin{align*}
    &\ \mathbb{E}_{z_0\sim \lambda}\left|\nabla\Phi(0, \omega(T_{\theta(0)}(z_0))) - \partial_{\theta} T_{\theta(0)}(z_0)\dot\theta(0)\right|^2\\
    =&\ \mathbb{E}_{z_0\sim \lambda}\left|\nabla\Phi(0, \omega(T_{\theta(0)}(z_0))) - \nabla\Phi(0, T_{\theta(0)}(z_0)) + \nabla\Phi(0, T_{\theta(0)}(z_0)) - \partial_{\theta} T_{\theta(0)}(z_0)\dot\theta(0)\right|^2\\
    \leq &\ 2~\mathbb{E}_{z_0\sim \lambda} \left|\nabla\Phi(0, \omega(T_{\theta(0)}(z_0))) - \nabla\Phi(0, T_{\theta(0)}(z_0))\right|^2 + 2~\mathbb{E}_{z_0\sim \lambda} \left|\nabla\Phi(0, T_{\theta(0)}(z_0)) - \partial_{\theta} T_{\theta(0)}(z_0)\dot\theta(0)\right|^2 \\
    \leq &\ 2 C_{\nabla\Phi_0}^2 W_2^2(\rho_{\theta(0)}, \rho_0) + 2~\mathbb{E}_{z_0\sim \lambda} \left|\nabla\Phi(0, T_{\theta(0)}(z_0)) - \partial_{\theta} T_{\theta(0)}(z_0)\dot\theta(0)\right|^2 ,
\end{align*}
and 
\begin{align*}
    &\ \mathbb{E}_{z_0\sim \lambda} \left|\nabla\Phi(0, T_{\theta(0)}(z_0)) - \partial_{\theta} T_{\theta(0)}(z_0)\dot\theta(0)\right|^2\\
    =&\ \mathbb{E} \left|\nabla\Phi(0, T_{\theta(0)}(z_0)) - \partial_{\theta} T_{\theta(0)}(z_0)\widehat{G}(\theta(0))^{\dagger}\int  \partial_{\theta}T_{\theta(0)}^{\top}(z)\nabla \Phi(0, T_{\theta(0)}(z))d\lambda (z)\right|^2\\
    =&\ \mathbb{E} \left|\nabla\Phi(0, T_{\theta(0)}(z_0)) - \mathcal{K}_{\theta(0)}[\nabla\Phi(0, T_{\theta(0)}(\cdot))](z_0)\right|^2\\
    =&\ \delta_2.
\end{align*}
Thus the initial error
\begin{equation}
  E(0) \leq (1+2C_{\nabla\Phi_0}^2)\epsilon_{\rho} + 2\delta_2.  \label{initial err}
\end{equation}
Combining \eqref{inequa by granwal} and \eqref{initial err}, we get
\begin{equation}
E(t)\leq e^{Ct}\left((1+2C_{\nabla\Phi_0}^2)\epsilon_{\rho} + 2\delta_2\right) + \left(\frac{3\delta_0}{C}+\frac{12 \delta_1(H_0-F_{\min})^2}{C\lambda_{\min,\Theta}^2}\right)(e^{Ct}-1). \label{estimate E(t) }
\end{equation}
Since
\begin{align*}
      W_2^2(\textrm{Law}(\boldsymbol{Y}), \textrm{Law}(\boldsymbol{X})) \leq \mathbb{E}|\boldsymbol{Y}-\boldsymbol{X}|^2 \leq  E(t) ,
\end{align*}
and $\textrm{Law}(\boldsymbol{Y}(t))=\rho_{\theta(t)}=T_{\theta(t)  \sharp}\lambda$ and $\textrm{Law}(\boldsymbol{X}(t)) = \rho(t, \cdot)$, we obtain \eqref{err est}.
\end{proof}
\begin{remark}
    Assumption \ref{assumption: lip constant of delfta F} can be verified if $\nabla\frac{\delta}{\delta \rho}\mathcal{F}$ is Lipschitz. As a special case, for the linear potential $\mathcal{F}(\rho)=\int V(x)~d\rho(x),$ Assumption \ref{assumption: lip constant of delfta F} holds true if $\nabla V$ is Lipschitz continuous. 
\end{remark}
\begin{remark} 
    From the proof of Theorem \ref{theorem: W2 error estimation}, we can see that there is a tradeoff between $\{\delta_i:i=0, 1, 2\}$ and $\lambda_{ \min, \Theta}$. In fact, if we choose $\Tilde{G}$ as the inner product matrix of a subspace $\Tilde{\mathcal{Q}}\subset \mathcal{Q}^{\theta}$, the arguments in this section still hold true. The smallest positive eigenvalue $\lambda_{\min}(\Tilde{G})$ is no less than $\lambda_{\min}(\widehat{G})$, while the corresponding approximation errors $\{\tilde{\delta}_i:i=0, 1, 2\}$ are generally larger than the original ones. A more detailed analysis of the relationship among these quantities may serve as a future research direction.
\end{remark}


\begin{theorem}[Error estimation on $\Phi$] \label{theorem: error estimation on Phi}
 Denote $\vec{\textrm{u}}_\Theta(t, \cdot):= \partial_{\theta} T_{\theta(t)}\circ T_{\theta(t)}^{-1}(\cdot)\dot{\theta}(t):\mathbb{R}^d\rightarrow \mathbb{R}^d$. 
 If the Hamilton-Jacobi equation 
 \begin{equation}
 \label{HJ eq Phi}
   \frac{\partial \Phi(t, x)}{\partial t} + \frac{1}{2}|\nabla\Phi(t, x)|^2 = -\nabla\frac{\delta}{\delta \rho}\mathcal{F}(\rho(x), x),\quad \Phi(0, \cdot)=\Phi_0(\cdot),  
 \end{equation}
 admits a $C^1([0, t_1)\times \mathbb{R}^d)$ solution on a time interval $[0,t_1)$, then 
 there is
 \begin{equation}
 \label{ineq: phi error estimation}
    \int_{\mathbb{R}^d} |\vec{\textrm{u}}_\Theta(t, x)-\nabla\Phi(t, x)|^2\rho_{\theta(t)}(x)\,dx \leq 2(1+\mathrm{Lip}(\nabla\Phi(t, \cdot)))~\mathcal{C}(\mathcal{F},\rho_0,\Phi_0,\theta_0,p_0,T_\theta).
 \end{equation}
 for all $0\leq t <  \min\{t_0, t_1\}$, where $\rho_{\theta(t)} = T_{\theta(t)\sharp}\lambda$,
 $\mathcal{C}(\mathcal{F},\rho_0,\Phi_0,\theta_0,\lambda,T_\theta, t)$ is the bound on the right hand side of \eqref{err est} and depends on potential $F$, the initial values $\rho_0,\Phi_0,\theta(0),p(0)$, the push-forward map $T_\theta$, and time $t$.
\end{theorem}

\begin{proof}
We can upper bound the average velocity discrepancy $\mathbb{E}|\dot{\boldsymbol{X}} - \dot{\boldsymbol{Y}}|^2$ by 
  \begin{equation}
     \mathbb{E}|\dot{\boldsymbol{X}} - \dot{\boldsymbol{Y}}|^2 \leq E(t), \label{upper bdd of vel err}
  \end{equation}
  where $E(t)$ is defined in \eqref{def E(t) }. 
  
  On the given time interval $[0, t_1)$ in which the Hamilton-Jacobi equation for $\Phi(t, \cdot)$ possesses a regular solution, we can verify $\dot{\boldsymbol{X}} = \nabla \Phi(t, \boldsymbol{X})$ for $t\in [0, t_1)$. On the other hand, we know $\dot{\boldsymbol{Y}} = \partial_{\theta} T_{\theta}\circ T_{\theta}^{-1}(\boldsymbol{Y})\dot{\theta}$. Thus we have $\dot{\boldsymbol{Y}} = \vec{\textrm{u}}_{\Theta}(t, \boldsymbol{Y})$. From the inequality that $|a+b|^2\geq \frac{1}{2}|b|^2 - |a|^2$ for any two vectors $a,b\in\mathbb{R}^d$, we can estimate $\mathbb{E}|\dot{\boldsymbol{X}} - \dot{\boldsymbol{Y}}|^2$ as
  \begin{align}
     \mathbb{E}|\dot{\boldsymbol{X}} - \dot{\boldsymbol{Y}}|^2 
     = &\ \mathbb{E}|\nabla\Phi(t, \boldsymbol{X}) - \vec{\textrm{u}}_\Theta(t, \boldsymbol{Y})|^2\nonumber\\ 
     = &\ \mathbb{E} |\nabla\Phi(t, \boldsymbol{X}) - \nabla\Phi(t, \boldsymbol{Y}) + \nabla\Phi(t, \boldsymbol{Y}) - \vec{\textrm{u}}_\Theta(t, \boldsymbol{Y})|^2\nonumber  \\
     \geq &\ \frac{1}{2} \mathbb{E} |\nabla\Phi(t, \boldsymbol{Y}) - \vec{\textrm{u}}_\Theta(t, \boldsymbol{Y})|^2 - \mathbb{E} |\nabla\Phi(t, \boldsymbol{X}) - \nabla\Phi(t, \boldsymbol{Y})|^2. \label{two parts estimation of momemtum error}
  \end{align}
  
  Using the definition of $E(t)$ given in \eqref{def E(t) }, the second term in \eqref{two parts estimation of momemtum error} can be bounded by
  \begin{equation}
     \mathbb{E}|\nabla\Phi(t, \boldsymbol{X}) - \nabla\Phi(t, \boldsymbol{Y})|^2\leq \textrm{Lip}(\nabla\Phi(t, \cdot))\mathbb{E}|\boldsymbol{X} - \boldsymbol{Y}|^2\leq \textrm{Lip}(\nabla\Phi(t, \cdot))E(t).\label{estimation of second part}
  \end{equation}
  Combining \eqref{upper bdd of vel err}, \eqref{two parts estimation of momemtum error}, and \eqref{estimation of second part}, we obtain
  \begin{equation*}
     \mathbb{E}|\nabla\Phi(t, \boldsymbol{Y}) - \vec{\textrm{u}}_\Theta(t, \boldsymbol{Y})|^2 \leq 2(1+\textrm{Lip}(\nabla\Phi(t, \cdot)))E(t).
  \end{equation*}
  Recalling \eqref{estimate E(t) }, we obtain the estimate \eqref{ineq: phi error estimation}.
\end{proof}
\begin{remark}
     Theorem \ref{theorem: error estimation on Phi} reveals that the approximation quality of the momentum depends on the current distribution $\rho_{\theta}$. In regions where $\rho_{\theta}$ has higher density, a better approximation of $\vec{\textrm{u}}_\Theta(t, \cdot)$ of $\nabla\Phi(t, \cdot)$ is anticipated.
\end{remark}
\begin{remark}
   The time intervals $[0,t_0)$ and $[0, t_1)$ used in Theorems \ref{theorem: W2 error estimation} and \ref{theorem: error estimation on Phi} are determined by the singularity development of $\rho$ and $\Phi$ in the WHF \eqref{WHF in dual coordinates} respectively. However, we would like to highlight that the solutions of PWHF \eqref{general PWGF pseudo} and \eqref{general PWGF pseudo simplified} may exist beyond these singularities. The same is true for the solution of the particle WHF \eqref{particle dynamics} due to the solution existence and uniqueness of ODEs. When this happens, we may use $\boldsymbol{X}$ and $v$ to define $\rho$ and $\Phi$ beyond the singularity of \eqref{WHF in dual coordinates}. The error estimates obtained in both theorems still hold as long as both the solutions of \eqref{particle dynamics} and \eqref{general PWGF pseudo simplified} exist. The examples on Wasserstein geodesic and harmonic oscillators given in Section \ref{sec:experiments} can illustrate this situation. In both examples, the solutions of PWHF exist on $[0,\infty)$ while finite time singularities are developed in $\rho$ and $\Phi$ in \eqref{WHF in dual coordinates}. 
\end{remark}

\subsection{Two examples of the PWHF} To better convey our idea on how PWHF is proposed and formulated, we present two illustrative examples that have exact solutions for (\ref{general PWGF pseudo simplified}). 

\subsubsection{Harmonic oscillator with affine transform as the push-forward map}
\label{exact para HO} Let us use an affine transform $T_\theta(z)=\Gamma z+b$, $\theta = (\Gamma, b),\ z\in \mathbb{R}^d$ as the parameterized push-forward map. Here $\Gamma$ is a $d \times d$ invertible matrix and $b$ is a $d$ dimensional vector. We consider a Hamiltonian system with Hamiltonian $H(x, v) = \frac{1}{2}|v|^2+\frac{1}{2}x^\top U x$, where $U$ is $d\times d$ self-adjoint and positive definite matrix. The corresponding WHF is
\begin{align*}
    & \frac{\partial \rho(t, x)}{\partial t}+\nabla\cdot(\rho(t, x)\nabla\Phi(t, x)) = 0,\\
   & \frac{\partial \Phi(t, x)}{\partial t} + \frac{1}{2}|\nabla\Phi(t, x)|^2 = -\frac{1}{2}x^\top Ux, 
\end{align*}
with $\mathcal{F}=\int \frac{1}{2}x^\top Ux\rho(x)dx$.

We take the initial values of this WHF as Gaussian distribution and quadratic function respectively, i.e., $\rho_0=\mathcal{N}(\mu, \Sigma)$, and $\Phi(0, x) = \frac{1}{2}x^\top M x$, where $\Sigma$ is the covariance and $M$ is a symmetric positive definite matrix. We choose the reference distribution as the standard normal, i.e. $\lambda = \mathcal{N}(0,I)$. 

If writing $\theta=(\Gamma_{11},...\Gamma_{1d},\Gamma_{21},...,\Gamma_{2d},...,\Gamma_{d1},...,\Gamma_{dd},b_1,...,b_d)$, it can be verified that the metric tensor $\widehat{G}(\theta) = I_{d(d+1)\times d(d+1)}$, which is a constant matrix. When projected on the parameter space, the potential becomes
\begin{equation*}
  F(\theta) =\mathcal{F}(\rho_{\theta})= \int_{\mathbb{R}^d} \frac{1}{2}x^\top U x \rho_{\theta}(x)~dx = \int_{\mathbb{R}^d} \frac{1}{2} (\Gamma z + b)^\top U (\Gamma z + b)~d\lambda = \frac{1}{2}\textrm{Tr}(\Gamma^\top U \Gamma) +\frac{1}{2}b^\top U b.
\end{equation*}
Thus the proposed PWHF is formulated as
\begin{align}
  & \dot{\theta} = \widehat{G}(\theta)^{-1}p = p, \label{exact example1 1}\\
  & \dot{p} = \frac{1}{2}\dot{\theta}^\top \nabla_{\theta}\widehat{G}(\theta)\dot{\theta} - \nabla_\theta F(\theta)=- \nabla_\theta F(\theta). \label{exact example1 2}
\end{align}
We set the initial value as 
\begin{equation*}
  \theta(0) = (\sqrt{\Sigma}, \mu), \quad p(0) = (M\sqrt{\Sigma}, M\mu).
\end{equation*}
Then one can verify that $\rho_{\theta(0)} = T_{\theta(0)\sharp}\lambda= \textrm{Law}(\sqrt{\Sigma} z + \mu) = \mathcal{N}(\mu,\Sigma) = \rho_0$, here $z\sim \lambda$. Thus $\epsilon_\rho$ stated in Theorem \ref{theorem: W2 error estimation} equals $0$; On the other hand, we know $(\dot{\Gamma}(0), \dot{b}(0)) = \dot{\theta}(0)=p(0)=(M\sqrt{\Sigma}, M\mu)$, and $\partial_\theta T_{\theta(0)}(z)\dot{\theta}(0) = \dot{\Gamma}(0)z +\dot{b}(0) = M\sqrt{\Sigma}z + M\mu = M T_{\theta(0)}(z) = \nabla \Phi(0, T_{\theta(0)}(z))$. Then one verifies that $\delta_2 = 0$ whose formulation is stated in Theorem \ref{theorem: W2 error estimation}. 

Both equations \eqref{exact example1 1} and \eqref{exact example1 2} can be reduced to the following second-order differential equation
\begin{align}
  & \ddot{\Gamma}(t) = -\nabla_\Gamma \left(\frac{1}{2}\textrm{Tr}(\Gamma^\top U \Gamma)\right), \quad \ddot{b}(t) = -\nabla_b\left(\frac{1}{2}b^\top U b\right),  \label{2nd order eq exact example1} \\
  & \Gamma(0)=\sqrt{\Sigma},~ \dot{\Gamma}(0) = M\sqrt{\Sigma};~ b(0) = \mu,~ \dot{b}(0) = M\mu. \label{2nd order eq initial cond}
\end{align}
Since $\nabla_\Gamma \left(\frac{1}{2}\textrm{Tr}(\Gamma^\top U \Gamma)\right) = \frac{1}{2}(U^\top\Gamma+U\Gamma)=U\Gamma$.
Assume $U$ has a spectral decomposition $U = Q\Lambda Q^\top$, by substitution $\Xi(t) = Q^\top \Gamma(t)$, $\Xi(t)$ solves the equation $\ddot{\Xi}(t) = -\Lambda \Xi(t)$. Then $\Xi(t) = \cos(t\sqrt{\Lambda})\Pi_c + \sin(t\sqrt{\Lambda})\Pi_s$\footnote{Assume $f:\mathbb{R}^d\rightarrow \mathbb{R}$ is an analytical function with power expansion  $f(x) = \sum_{k=0}^{\infty}a_k x^k$, for any square matrix $A$, we define $f(A)=\sum_{k=0}^{\infty} a_kA^k$. Typically, if $A$ is self-adjoint and has spectral decomposition $A = Q\textrm{diag}(\lambda_1,...,\lambda_d)Q^\top$, then $f(A)=Q\textrm{diag}(f(\lambda_1),...,f(\lambda_d))Q^\top$.}, where $\Pi_c,\Pi_s$ are constant $d\times d$ matrix that need to be determined by using the initial condition. Thus, $\Gamma(t) = Q\Xi(t) = Q(\cos(t\sqrt{\Lambda})\Pi_c + \sin(t\sqrt{\Lambda})\Pi_s)$. Similarly, one can verify $b(t) = Q(\cos(t\sqrt{\Lambda})\textrm{v}_c + \sin(t\sqrt{\Lambda})\textrm{v}_s)$, where $\textrm{v}_c, \textrm{v}_s$ are two constant $d$ dimensional vectors. 

One can determine $\Pi_c=Q^\top \sqrt{\Sigma},\Pi_s=\sqrt{\Lambda}^{-1}Q^{\top}M\sqrt{\Sigma}, \textrm{v}_c=Q^\top \mu, \textrm{v}_s=\sqrt{\Lambda}^{-1}Q^\top M\mu$ from the initial condition \eqref{2nd order eq initial cond}, and obtain the solution to the PWHF as
\begin{align*}
  & \Gamma(t) = (\cos(t\sqrt{U}) +\sin(t\sqrt{U})\sqrt{U}^{-1}M)\sqrt{\Sigma}, \\
  & b(t) = (\cos(t\sqrt{U}) +\sin(t\sqrt{U})\sqrt{U}^{-1}M)\mu.
\end{align*}
One can tell that 
\begin{equation*}
 \textrm{span}\left\{\frac{\partial T_\theta(\cdot)}{\partial\theta_l}\right\}_{1\leq l\leq d(d+1)} = \textrm{span}\left\{...,\frac{\partial T_\theta(\cdot) }{\partial \Gamma_{ij}},...,\frac{\partial T_\theta(\cdot)}{\partial b_k},...\right\}.
\end{equation*}
where $\frac{\partial T_\theta(x) }{\partial \Gamma_{ij}} = x_j\mathbf{e}_i$, $\frac{\partial T_\theta(\cdot)}{\partial b_k} = \mathbf{e}_k$. Then 
\begin{equation*}
  \nabla \frac{\delta}{\delta\rho}\mathcal{F}(\rho, x) = U x = \sum_{1\leq i,j\leq d}U_{ij} x_j\mathbf{e}_i\in \textrm{span}\left\{\frac{\partial T_\theta}{\partial\theta_l}\right\}_{1\leq l\leq d(d+1)}.
\end{equation*}
Thus the quantity $\delta_0$ introduced in \eqref{pseudo delta 0 def} equals $0$. On the other hand, since $T_{\theta}$ is linear w.r.t. to $\theta$, $\frac{\partial^2 T_\theta(\cdot)}{\partial \theta^2}=0$, thus $\delta_1$ as defined in \eqref{def: delta 1 def} is also $0$. Hence we verify that $\epsilon_\rho,\delta_0,\delta_1, \delta_2=0$, according to the error estimation provided in \eqref{err est}, one can tell that the parameterized Hamiltonian flow $\{(\Gamma(t), b(t))\}$ recovers the exact flow $\{\rho(t, \cdot), \Phi(t, \cdot)\}$.

In addition, denote $(\boldsymbol{X}, \boldsymbol{P}_t)=(\Gamma(t)z+b(t), \dot{\Gamma}(t)z+\dot{b}(t))$, one can verify 
\begin{align*}
 & \dot{\boldsymbol{X}} = \boldsymbol{P}, \quad \boldsymbol{X}_0\sim \mathcal{N}(\mu,\Sigma);\\
 & \dot{\boldsymbol{P}} = -\nabla V(\boldsymbol{X}), \quad \boldsymbol{P}_0 = \nabla\Phi(0, \boldsymbol{X}_0),
\end{align*}
by direct calculation, this also leads to the aforementioned assertion.

\subsubsection{Entropic potential with diagonal matrix as the push-forward map}
\label{exact para Entropy}
Here we consider a WHF with 
$\mathcal{F}(\rho)$ taken as the entropic potential $\mathcal{E}(\rho)=\int \rho\log\rho~dx$, i.e.,
\begin{align}
    & \frac{\partial \rho(t, x)}{\partial t}+\nabla\cdot(\rho(t, x)\nabla\Phi(t, x)) = 0,\quad \rho(0, \cdot)=\mathcal{N}(0, I_d);\label{exact example2 1}\\
   & \frac{\partial \Phi(t, x)}{\partial t} + \frac{1}{2}|\nabla\Phi(t, x)|^2 = -\frac{\delta \mathcal{E}(\rho)}{\delta\rho}(t, x), \quad \Phi(0, x)=\frac{|x|^2}{2}. \label{exact example2 2}
\end{align}

Here we have $\frac{\delta \mathcal{E}(\rho)}{\delta\rho}(t, x) = 1+\log\rho(t, x)$. Again, let us take the reference distribution $\lambda=\mathcal{N}(0, I_d)$, and the push-forward map is a linear transform with diagonal matrix, i.e., $T_\theta(z) = D z$, where $D=\textrm{diag}(D_1,...,D_d)$ is $d\times d$ diagonal matrix with $D_k > 0$ for $1\leq k\leq d$. Then the parameter $\theta = (D_1,...,D_d)$, and $\widehat{G}(\theta) = I_{d\times d}$. In addition, we have 
\begin{equation*}
  \mathcal{E}(\rho_\theta)=\int \rho_\theta\log\rho_\theta~dx = \int_{\mathbb{R}^d} \left(\log\left(\frac{1}{(2\pi)^{\frac{d}{2}} \textrm{det}(D)}\right) - \frac{1}{2}x^\top D^{-2} x\right)\rho_\theta~dx = -\log(2\pi)^{\frac{d}{2}} - \sum_{k=1}^d \log D_k -\frac{1}{2}.
\end{equation*}
The PWHF is reduced to $\ddot{\theta} = -\nabla_\theta F(\theta)$, and it becomes the following equation
\begin{equation}
\label{eq: entropy D equation}
\ddot{D}_k(t)=\frac{1}{D_k(t)}
\end{equation}
We set the initial values as
\begin{equation*}
  D_k(0) = 1, \quad \dot{D_k}(0) = 1, \quad 1\leq k\leq d.
\end{equation*}
By a similar argument as given in the previous example, one can verify that $\rho_{\theta(0)}=\mathcal{N}(0, I_d)=\rho_0$, as well as $\partial_\theta T_{\theta(0)}(z)\dot{\theta}(0) = \nabla \Phi(0, T_{\theta(0)}(z))$. Thus we have $\epsilon_\rho=0$, $\delta_2 = 0$.

Since all the $D_k$ solve the same differential equation with common initial conditions, we simply drop the subscript and denote each $D_k$ as $D$. By multiplying $\dot{D}(t)$ on both sides of \eqref{eq: entropy D equation}, one can verify that $D$ solves $\dot{D}(t) = \sqrt{1+\log D(t)}$. And thus we can solve for
\begin{equation}
\label{entropy true para}
  D(t)=\exp((\chi^{-1}(\chi(1)+\frac{e}{2}t))^2-1),
\end{equation}
here $\chi(\cdot)$ denotes the primitive function of $e^{t^2}$. Let us denote $\widehat{\boldsymbol{X}} = T_{\theta}(\boldsymbol{Z}) = D(t)\boldsymbol{Z}$ with $\boldsymbol{Z}\sim \lambda$. Then it can be verified that $\{\widehat{\boldsymbol{X}}\}$ solves the following Vlasov-typed ordinary differential equation associated to the Wasserstein Hamiltonian flow \eqref{exact example2 1}, \eqref{exact example2 2},
\begin{equation}
  \frac{d^2}{dt^2}\widehat{\boldsymbol{X}} = -\nabla \frac{\delta\mathcal{E}(\rho)}{\delta\rho}(t, \widehat{\boldsymbol{X}})=-\nabla\log\rho(t, \widehat{\boldsymbol{X}}), ~\frac{d^2}{dt^2}\widehat{\boldsymbol{X}}_0\sim\rho_0, ~\frac{d}{dt}\widehat{\boldsymbol{X}}_0 = \boldsymbol{X}_0.  \label{exact example2 SDE}
\end{equation}
where $\rho(t, \cdot)$ denotes the density of $\textrm{Law}(\widehat{\boldsymbol{X}})$. 

As a result, we can tell that $\rho_{\theta}=T_{\theta\sharp}\lambda=\textrm{Law}(\widehat{\boldsymbol{X}}_t)=\mathcal{N}(0, D^2(t)I_d)$ exactly solves for $\rho_t$ of the Hamiltonian flow \eqref{exact example2 1},\eqref{exact example2 2}. At the same time, the momentum $\vec{\textrm{u}}_\Theta$ obtained from the parameterized Hamiltonian flow satisfies $\vec{\textrm{u}}_\Theta (t, \widehat{\boldsymbol{X}}) = \partial_\theta T_{\theta}\circ T_{\theta}^{-1}(\widehat{\boldsymbol{X}})\dot{\theta} = \dot{D}(t)T_{\theta}^{-1}D(t)\widehat{\boldsymbol{X}}_0=\dot{D}(t)\widehat{\boldsymbol{X}}_0=\frac{d}{dt}\widehat{\boldsymbol{X}}$. This verifies that $\vec{\textrm{u}}_{\Theta}(t, \cdot)$ exactly solves for $ \nabla\Phi(t, \cdot)$ from the Hamiltonian flow \eqref{exact example2 1},\eqref{exact example2 2}.

\section{Numerical scheme}
\label{sec: numerical method}
In this section, we develop a numerical scheme to solve the PWHF \eqref{general PWGF pseudo simplified}. Since \eqref{general PWGF pseudo simplified} is a Hamiltonian system, it is desirable that the scheme has a symplectic structure. For clarity, we denote the iteration number of variables as superscripts in this section.

{\bf Symplectic Scheme}
We start from the following sympletic Euler scheme \cite{hairer2006geometric},
\begin{subequations}
\label{SymEuler}
    \begin{align}
        \frac{\theta^{l+1}-\theta^l}{h}&=\nabla_pH(\theta^{l+1}, p^l)=\widehat{G}(\theta^{l+1})^{\dagger}p^l, \label{eq:SymEuler-theta}\\
        \frac{p^{l+1}-p^l}{h}&=-\nabla_{\theta}H(\theta^{l+1}, p^l),\label{eq:SymEuler-p}
    \end{align}
\end{subequations}
where $h>0$ is the time step size and $l \in \mathbb{N}$ is iteration number.
Note that \eqref{eq:SymEuler-theta} is implicit in $\theta^{l+1}$, which needs to be solved from this equation for fixed $(\theta^l,p^l)$. To solve \eqref{eq:SymEuler-theta}, we employ a fixed point iteration method. For convenience, we call the fixed point procedure the inner iteration for any fixed $(\theta^l,p^l)$, while the advancement in time, namely the iterations of \eqref{SymEuler} in $l$, the outer iteration. 
%
%
For each $(\theta^{l},p^{l})$, we can solve for $\theta^{l+1}$ in \eqref{eq:SymEuler-theta} by the following fixed point iterations:
\begin{subequations}
\label{FPI}
    \begin{align}
    \xi^{l,j+1} &= \underset{\xi}{\textrm{argmin}} ~ \Big\{\frac{1}{2} \xi^{\top} \widehat{G}(\alpha^{l,j}) \xi - \xi^{\top} p^l \Big\}, \label{eq:FPI-xi}\\
    \alpha^{l,j+1} & = \theta^{l} + h \xi^{l,j+1}. \label{eq:FPI-alpha}
    \end{align}
\end{subequations}
for $l=1,2,\dots$. 
We have two choices to initialize the above fixed point iterations. The first one is to set $\alpha^{l,0}=\theta^l$ and $\xi^{l,0}=(\theta^l-\theta^{l-1})/h$, and the second choice is to set $\alpha^{l,0}=\theta^l$ and $\xi^{l,0}=\widehat{G}(\theta^l)^{\dagger}p^l$. The first choice utilizes information from previous time step and is more computationally efficient, while the second choice requires to solve a linear system, but generally performs better than the first one. We adopt the second way to initialize $\alpha, \xi$ in all of our experiments.

If $(\alpha^{l,j},\xi^{l,j}) \to (\alpha^{l,*},\xi^{l,*})$ as $j\to\infty$, then we set $\theta^{l+1}=\alpha^{l,*}$ and $p^{l+1} = p^l + h \frac{1}{2} (\xi^{l,*})^{\top} \nabla_{\theta}\widehat{G}(\theta^{l+1}) \xi^{l,*} - h \nabla_\theta F(\theta^{l+1})$.
It can be shown that this fixed point iteration converges if $\lambda_{\min}(\widehat{G}(\theta))$ is bounded away from 0. 
In our experiments, the minimization subproblem of $\xi$ is approximated by a one-step gradient descent with step size $\gamma$:
\[
\xi^{l,j+1} = \xi^{l,j} - \gamma (\widehat{G}(\alpha^{l,j}) \xi^{l,j} - p^l),
\]
and the inner iteration \eqref{FPI} reduces to
\[
\xi^{l,j+1} = \xi^{l,j} - \gamma (\widehat{G}(\theta^l + h \xi^{l,j}) \xi^{l,j} - p^l).
\]
which is the fixed point iterations applied to solving $\xi$ from $J(\xi):=(\widehat{G}(\theta^l+ h \xi) \xi - p^l= 0$. Since $\nabla_{\xi} J(\xi) = h \nabla \widehat{G}(\theta^l+ h \xi) \xi + \widehat{G}(\theta^l+ h \xi)$, for certain choices of $T_{\theta}$ (e.g. $T_{\theta}(z) = \Gamma z + b$ and $z \sim \lambda = N(0,I)$, or $T_{\theta}(z) = z + \sigma(\Gamma z + b)$, or deeper ResNet, or normalizing flows with bounded $\sigma,\sigma',\sigma''$ where $\sigma $ is the activation function) we can show that there exists $M > 0$ such that $\|\nabla_{\xi} J(\xi)\| \le M$ for all $\xi$. Then $\tau \in (0, 1/M)$ guarantees that the fixed point iteration is linearly convergent since $I-\tau J$ is a contraction.

{\bf Useful tricks in implementation in PyTorch} We discuss a few tricks in implementing the symplectic Euler scheme \eqref{SymEuler} when using a machine learning package such as PyTorch. PyTorch leverages automatic differentiation and can quickly compute the gradient of a scalar valued function with input dimension $d$. However, it takes $O(kd)$ complexity to compute the Jacobian of a vector-valued function with $k$ output dimension. In our case, we need to repeatedly compute the matrix-vector product involving $\partial_{\theta}T_{\theta}(z)$ and some vector $\eta \in \mathbb{R}^{d}$:
\begin{align}
    \widehat{G}(\theta)\eta =\int  \partial_\theta T_\theta(z)^{\top} \partial_\theta T_\theta(z)\eta~d\lambda(z).\label{mvp}
\end{align}
We need to do $n\times d$ times differentiation to get $\{\partial_{\theta}T_{\theta}(z_i)\}_{i=1}^n$ where $n$ is the number of samples and $d$ is dimension. To avoid such computation, we design a duplication trick based on chain rule to evaluate this term without reducing the efficiency. We duplicate the push-forward map $T_{\theta}$ to get an identical copy $T_{\Tilde{\theta}}$, with exactly the same structure and value of parameters, but the parameters $\Tilde{\theta}$ are detached from the computational graph of the original parameters $\theta$. Then we evaluate the scalar-valued integral:
\begin{align}
    loss(\theta, \Tilde{\theta})=\int  T_\theta(z)^{\top}  T_{\Tilde{\theta}}(z)~d\lambda(z).
\end{align}
We auto differentiate it with respect to the parameters $\Tilde{\theta}$ and compute the inner product between this gradient and vector $\eta$: 
\begin{align}
    g_1(\theta, \Tilde{\theta}) &= \partial_{\Tilde{\theta}}loss(\theta, \Tilde{\theta})\cdot \eta =\int T_\theta(z)^{\top} \partial_{\Tilde{\theta}} T_{\Tilde{\theta}}(z)\eta~d\lambda(z).
\end{align}
Finally we auto differentiate $g_1$ w.r.t parameters $\theta$ and notice the fact that $\Tilde{\theta}$ and $\theta$ has identical value, we get:
\begin{align*}
    g_2(\theta, \Tilde{\theta})&=\partial_{\theta}g_1(\theta, \Tilde{\theta})  =\int  \partial_\theta T_\theta(z)^{\top} \partial_{\Tilde{\theta}} T_{\Tilde{\theta}}(z)\eta~d\lambda(z) =\widehat{G}\eta.
\end{align*}
In this way, we obtain the value $\widehat{G}\eta$ with high efficiency and accuracy. Similarly, we can also evaluate $[\eta^{\top} (\partial_{\theta_k}\widehat{G})\eta]_{k=1}^{m} $, which shows up in the second equation in (\ref{general PWGF pseudo simplified}):
\begin{align*}
[\eta^{\top} (\partial_{\theta_k}\widehat{G})\eta]_{k=1}^{m}&= \nabla _{\theta}[\eta^{\top}g_2(\theta, \Tilde{\theta})] +\nabla _{\Tilde{\theta}}[\eta^{\top}g_2(\theta, \Tilde{\theta})] =2 \cdot\nabla _{\theta}[\eta^{\top}g_2(\theta, \Tilde{\theta})] .
\end{align*}
Here we used chain rule and the fact that $\theta=\Tilde{\theta}$.

{\bf Evaluate $\widehat{G}^{\dagger}p$ by solving linear system} In the fixed-point iteration, we need to compute $\xi^0=\widehat{G}(\theta^k)^{\dagger}p^k$ in the beginning of each inner iteration.
The simplified form metric tensor $\widehat{G}$ is defined through the push-forward map via equation (\ref{relaxed metric tensor}), which can be evaluated through samples, but the computational cost grows fast when the number of parameters increases. Fortunately, the equations in (\ref{general PWGF pseudo simplified}) can be treated as linear system, hence we can apply iterative method to solve the system. In this sense, we don't require the full information of matrix $\widehat{G}$, but just $\widehat{G}$ as matrix-vector product operator. To find $\widehat{G}^{\dagger}p$, we consider the linear system $\widehat{G}(\theta)\eta=p$ and apply iterative method to solve for $\eta = \widehat{G}^{\dagger}p$. Generally the condition number of $\widehat{G}(\theta)$ can be very large, so we choose MINRES as the iterative solver \cite{saad2003iterative}.  

{\bf Initialization} In the experiments, the initialization for $\theta$ and $p$ are treated differently:

\begin{itemize}
    \item \textbf{Initialize $\theta$.} We initialize $\theta$ to minimize the difference between $\rho_{\theta^0}$ and $\rho_0$. It can be done via minimizing the KL divergence or Stein's discrepancy between $\rho_{\theta^0}$ and $\rho_0$ \cite{rezende2015variational}.
In the experiments, we choose the reference distribution the same as $\rho_0$ and initialization for $T_{\theta}$ as the identity map, so the initialization procedure for $\theta$ is generally omitted.
    \item \textbf{Initialize $p$.} From the particle interpretation of $\Phi$, we can derive
\begin{equation}
\begin{split}
p^0&=\widehat{G}\dot{\theta}^0=\int_M \partial_{\theta}T_{\theta}^{\top}(z)v(T_{\theta}(z))d\lambda (z) = \int_M \partial_{\theta}T_{\theta}^{\top}(z)\nabla \Phi(0, T_{\theta}(z))d\lambda (z) = \mathbb{E}_{\lambda}[\nabla_{\theta}\Phi(0, T_{\theta}(z))].
\end{split}
\end{equation}
The above identity povides a natural way to initialize $p$ from initial condition for $\Phi$. We sample $\{z_i\}_{i=1}^{K_{p}}$ from the reference distribution $\lambda$, and set $p^0$ to be the sample expectation $\frac{1}{K_{p}}\sum_{i=1}^{K_{p}}\nabla_{\theta}\Phi(0, T_{\theta}(z_i))$.
\end{itemize}

In summary, we suggest the following numerical algorithm to solve the PWHF.
\begin{algorithm}[H]
\caption{Parameterized Wasserstein Hamiltonian flow solver}
\label{alg:HFsolver}
\begin{algorithmic}
\STATE{Initialize the neural network $T_{\theta}$, and solve $\theta^0=\underset{\theta}{\textrm{argmin}} \{\mathcal{D}_{\mathrm{KL}}(\rho_0 \| \rho_{\theta})$\} }
\STATE{Initialize $p^0=\nabla_{\theta}\mathbb{E}_{z\sim\lambda}[\Phi(0, T_{\theta}(z))]$}
\FOR{$l=0, \cdots, K-1$}
\STATE{Sample $\{X_1, \cdots, X_{N_{\theta}}\}$ from $\rho_{\theta}$}
\STATE{Apply MINRES to solve $\xi^{l, 0}$ from equation $\widehat{G}(\theta^{l})\xi =p^l$, set $\alpha^{l, 0}=\theta^l$}
\FOR{$j=1, \cdots, n_{in}$}
\STATE{Update $\alpha^{l, j}= \theta^l+h\xi^{l, j}$}
\STATE{Update $\xi^{l,j+1} = \xi^{l,j} - \gamma (\widehat{G}(\alpha^{l,j}) \xi^{l,j} - p^l)$
}
\ENDFOR
\STATE{Set $\theta^{l+1}=\alpha^{l, n_{in}}, \eta^{l+1}=\xi^{l, n_{in}}$}
\STATE{Sample $\{X_1, \cdots, X_{N_p}\}$ from $\rho_{\theta^{l+1}}$, evaluate $\nabla_{\theta}F(\theta^{l+1})$}
\STATE{Set $p^{l+1}=p^l+\frac{h}{2}[(\eta^{l+1})^{\top} \partial_{\theta_k}\widehat{G}\eta^{l+1}]_{k=1}^{m}-h\nabla_{\theta}F(\theta^{l+1})$}
\ENDFOR
\end{algorithmic}
\end{algorithm}

\section{Numerical results}
\label{sec:experiments}
In this section, we demonstrate the performance of the proposed algorithm \ref{alg:HFsolver} by solving several examples with different potential energy $\mathcal{F}$. 


In our experiments, we use neural networks as the push-forward map $T_{\theta}$ unless stated otherwise. There are multiple choices for the neural networks structure to represent $T_{\theta}$. We can use the invertible neural networks (e.g., normalizing flow \cite{rezende2015variational}, Real NVP \cite{dinh2016density} and neural ODE \cite{chen2018neural}) or non-invertible neural networks (e.g., the multi-layer perceptron or ResNet \cite{he2016deep}), both has its own advantages. Normalizing flow simplifies the computation of log determinant of Jacobian matrix of the map, so we can easily compute the density function. In some cases, the potential function explicitly depends on the value of density $\rho_{\theta}$, invertible neural networks provides an effective way to evaluate it. In this work, we use the residual neural network as push-forward map if not specified:
\begin{align}
\label{residual nn}
T_{\theta} = \textrm{Id} + f_{\theta}    ,
\end{align}
 where $\textrm{Id}$ is the identity map and $f_{\theta}: \mathbb{R}^d\rightarrow \mathbb{R}^d$ is a standard multilayer perceptron with two hidden layers, and each hidden layer contains 50 and 80 neurons for the $2$D and $10$D examples, respectively. We take hyperbolic tangent function as the activation since we require the second-order derivative in the computation. The bias for the output layer in $f_{\theta}$ is set to be None. MINRES is applied to solve the linear system $\widehat{G}\dot{\theta}=p$ with tolerance $3\cdot 10^{-4}$.
To solve the inner system, we set $n_{in}=1$ and run the algorithm. Experiments show that $n_{in}=1$ is enough for most of the examples. 

For all the experiments, we choose the reference distribution as $\lambda=\mathcal{N}(0, I_d)$, and take $\rho_0=T_{\theta\sharp}\lambda$. Generally initial $\theta$ is small, hence $T_{\theta}$ is close to the identity map and $\rho_0$ is close to the standard Gaussian as a result. For the computation, 50000 samples are generated from the reference distribution to evaluate the matrix-vector product. For those cases with potential energy in the form $\mathcal{F}(\rho)=\mathbb{E}_{\rho}V(x)$, we use the same sample size to evaluate $F(\theta)$. For other cases, we only use 12000 samples for the computation of $F(\theta)$ due to the memory limitation.



It's possible that Hamiltonian-type PDEs develop singularity in finite time. This means, at some time $t_0$, 
the density may become a delta function and/or $\Phi$ is no longer well-defined. Our method can treat this singularity since the push-forward is still well-defined and smooth at time $t_0$. See the geodesic equation and harmonic oscillator as examples.


To examine our solutions, we use the particle level dynamics. Assume $X(0,x_0)=T(0,x_0)$ solves the system (\ref{pWHF-1}), i.e., $T(t)$ is the ground truth push-forward map for the WHF. For some special potential energy, $T(t)$ have closed form solutions. Hence we can compare our results with true solutions. For other examples where true solutions are not available, we generate 10000 random samples and run the particle level dynamics to get numerical approximation to $T(t)$, then we compare our results $T_{\theta(t)}$ with the particle level results.

\subsection{Geodesic equation as Wasserstein Hamiltonian flow}
\label{geodesic sec}
We first consider the geodesic equation on $\mathbb{R}^2$ as a WHF, which corresponds to the zero potential case, i.e., $\mathcal{F}(\rho)\equiv 0$. The equations are
\begin{align*}
    \frac{\partial}{\partial t}\rho+\nabla \cdot (\rho\nabla \Phi)&=0,\\
    \frac{\partial }{\partial t}\Phi +\frac{1}{2}|\nabla\Phi|^2&=0.
\end{align*}
The Wasserstein geodesic equation plays an important role in the optimal transport theory. The trajectory of the corresponding particle level dynamics is a straight line, i.e., $T_{\theta (t)}(x_0)=x_0 + v(x_0)t$ where $v(x_0)=\nabla \Phi(0, x_0)$ is the initial velocity and $x_0$ is an arbitrary initial position. 
With some choice of $\Phi$, the system may develop a singularity in finite time. The numerical experiments, which are presented next, show that our method can approximate the solution very well, even beyond the singularity. 
The PWHF corresponding to the Wasserstein geodesic equation is
\begin{equation}
\label{geodesic theta PWGF}
    \begin{aligned}
        \dot\theta &= \widehat{G}^{\dagger}p,\\
        \dot p &= \frac{1}{2}(\widehat{G}^{\dagger}p)^{\top} (\nabla_{\theta}\widehat{G})\widehat{G}^{\dagger}p.
    \end{aligned}
\end{equation}
We set $\lambda=\mathcal{N}(0, I)$ and choose $T_{\theta}$ as in (\ref{residual nn}). We initialize the neural network parameters $\theta_0$, take the initial density as $\rho_0=T_{\theta(0)\sharp}\lambda$ and $\Phi(0, x)=-\frac{1}{2}x_1^2$  with $x=(x_1, x_2)\in \mathbb{R}^2$. From the choice of $\Phi(0, x)$, we know that for each initial position its velocity is always in the $x_1$ direction with constant magnitude $|x_1|$. As a result, all points arrive at the $x_2$-axis when $t=1$ and the density function $\rho(1, x)$ becomes a Gaussian supported on the $x_2$-aixs. We solve the parameter dynamics (\ref{geodesic theta PWGF}) and compare the reuslts with true solution.
We choose $5$ points randomly and plot the projection of their trajectories onto $x_1, x_2$ axis, as shown in Figure \ref{geodesic traj plot}. The true trajectories are $T_{t}(x)=((1-t)x_1,x_2)$. To show $\rho_{\theta}$, we also draw $2000$ samples from reference distribution $\lambda$ and then apply the neural network $T_{\theta}$ solved from our algorithm for different time, results are shown in Figure \ref{geodesic sampleplot}. Notice that $T_1(x)=(0, x_2)$ for all $x=(x_1, x_2)\in \mathbb{R}^2$, which implies finite time singularity in $\rho$. Our solution solves the problem very well despite the existence of singularity.

\begin{figure*}[!htb]
    \centering
    \begin{subfigure}{0.45\textwidth}
        \centering
        \includegraphics[width=0.95\linewidth]{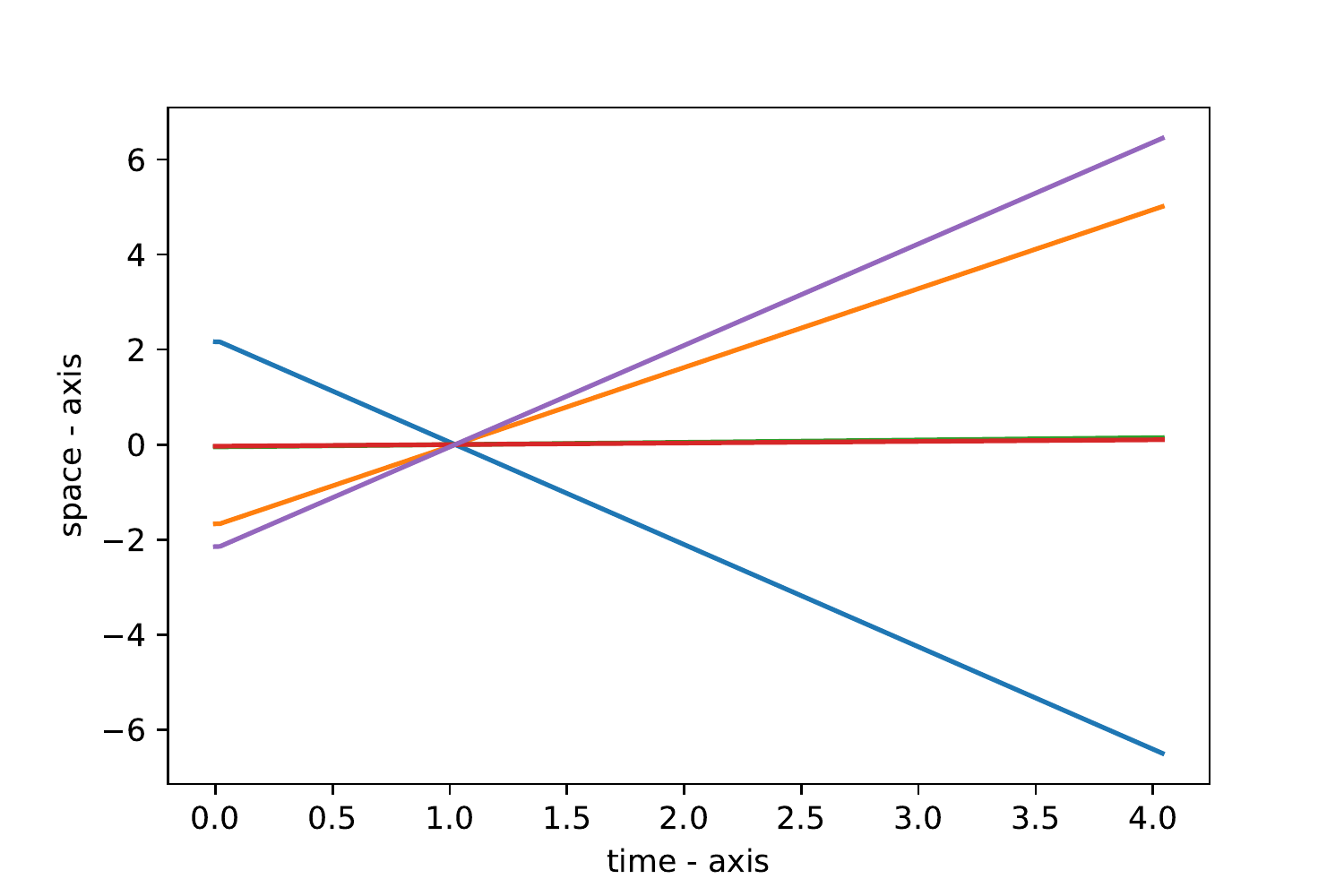}
        \caption{projection of trajectories in $x_1$ direction}
    \end{subfigure}%
    ~
    \begin{subfigure}{0.45\textwidth}
        \centering
        \includegraphics[width=0.95\linewidth]{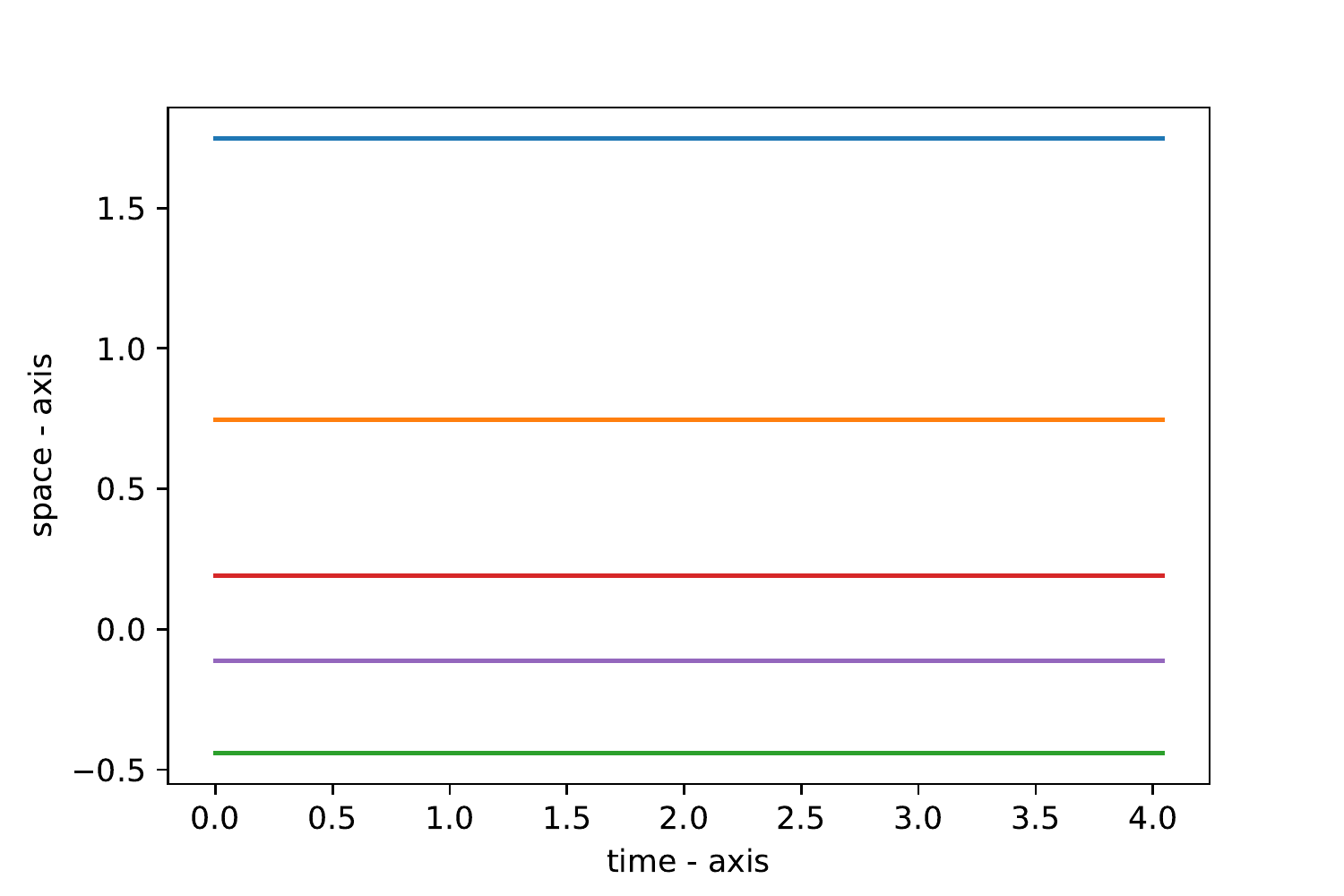}
        \caption{projection of trajectories in $x_2$ direction}
    \end{subfigure}
    \caption{Projection of trajectories with random initial positions. For each initial position $x=(x_1, x_2)$, its velocity $v=(-x_1, 0)$ is constant along time. As a result, the equation develops singularity at time $t=1$, and projection of $T_{\theta}(x)$ onto $x_1$-plane intersects for all points.}
    \label{geodesic traj plot}
\end{figure*}

\begin{figure*}[t!]
    \begin{subfigure}{0.32\textwidth}
        \centering
        \includegraphics[width=0.95\linewidth]{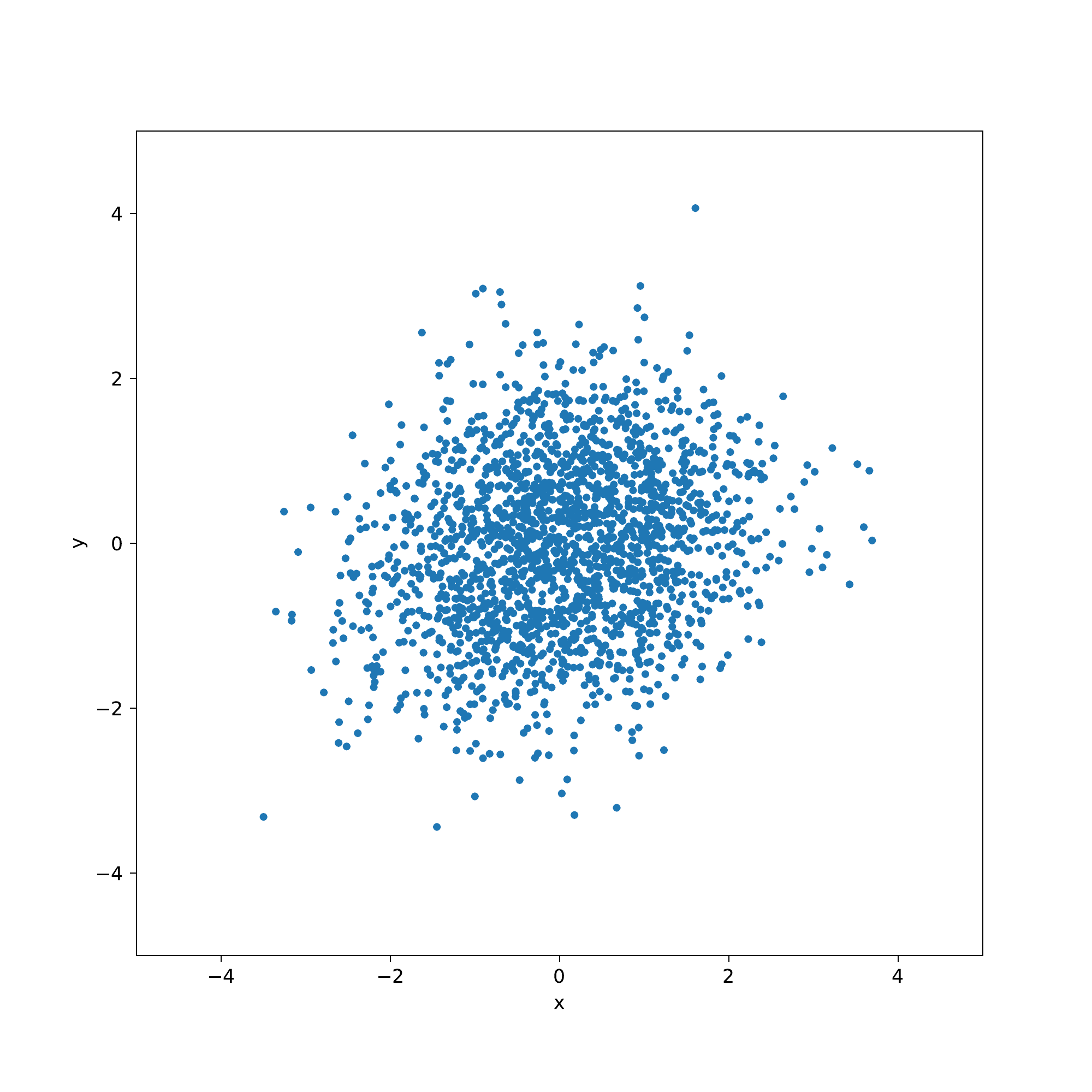}
        \caption{$t=0$}
    \end{subfigure}%
    ~
    \begin{subfigure}{0.32\textwidth}
        \centering
        \includegraphics[width=0.95\linewidth]{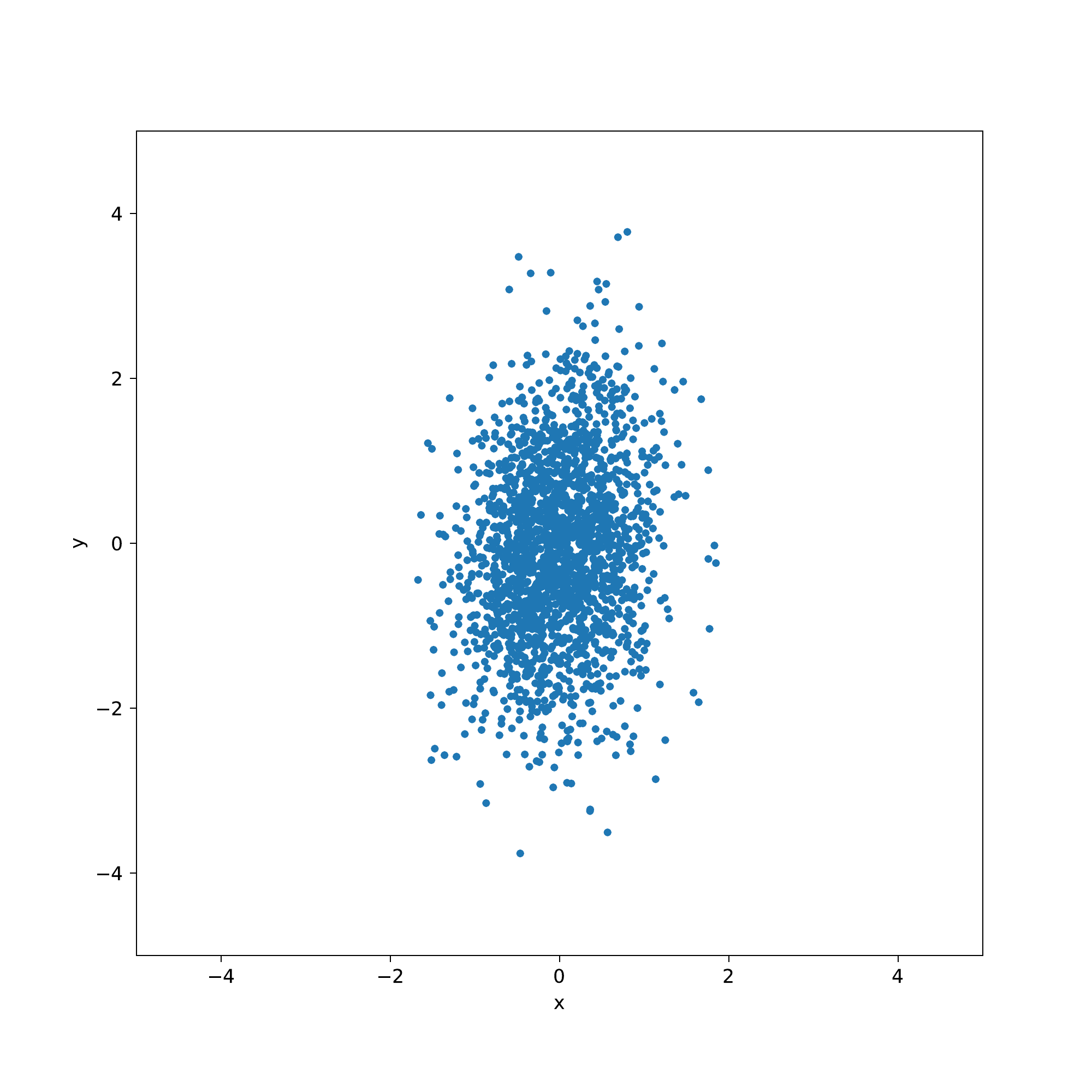}
        \caption{$t=0.5$}
    \end{subfigure}
    ~
    \begin{subfigure}{0.32\textwidth}
        \centering
        \includegraphics[width=0.95\linewidth]{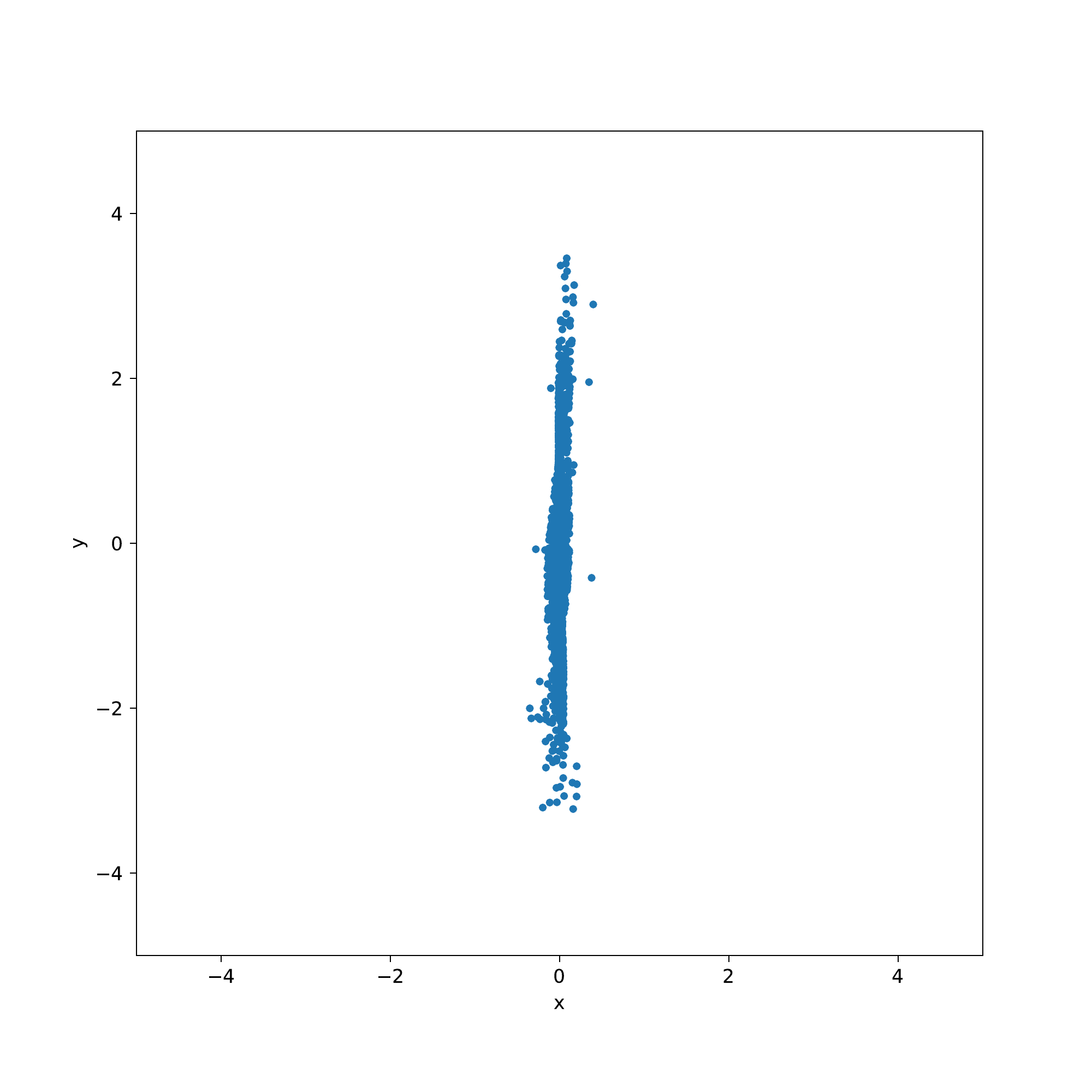}
        \caption{$t=1$}
    \end{subfigure}
    \\
    \begin{subfigure}{0.32\textwidth}
        \centering
        \includegraphics[width=0.95\linewidth]{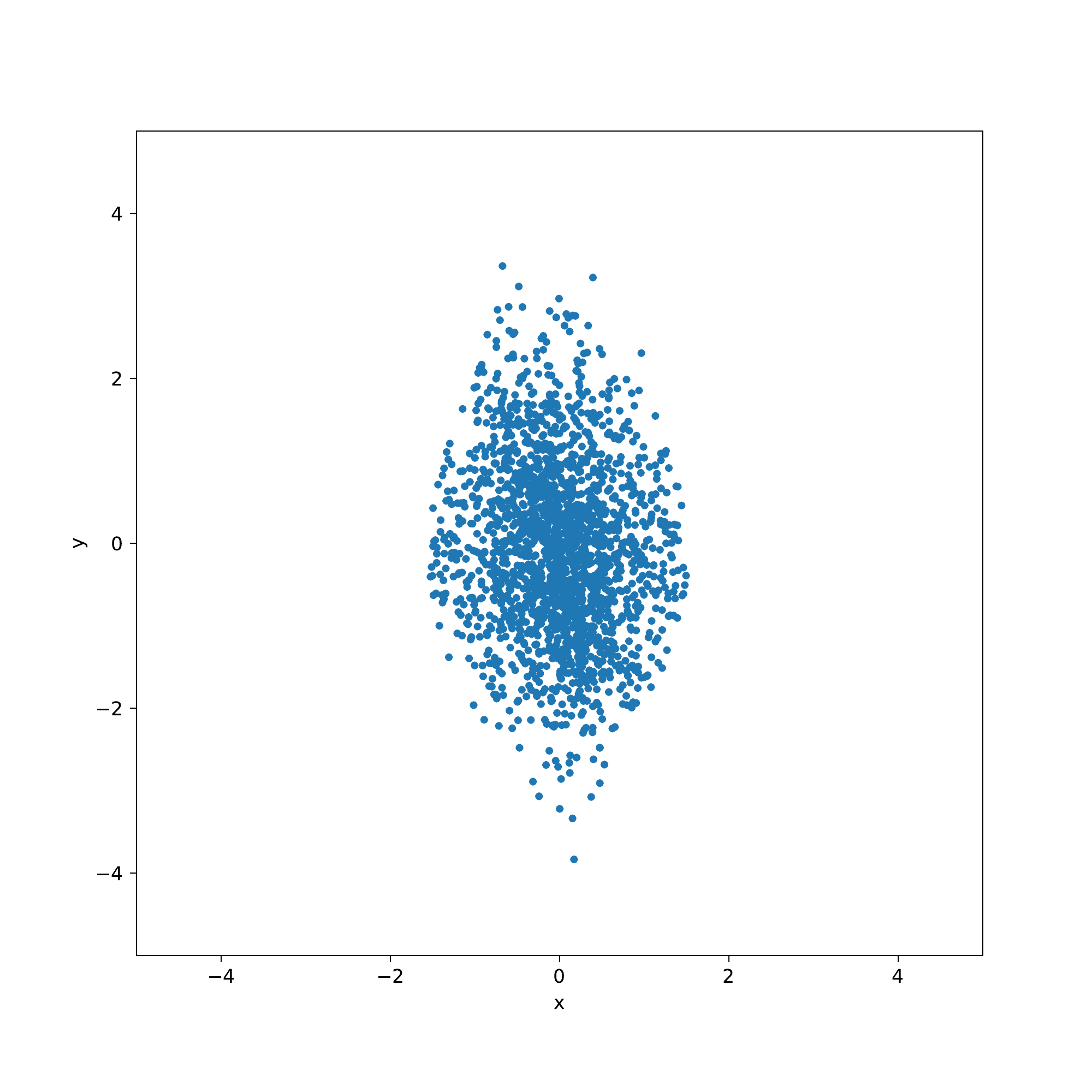}
        \caption{$t=1.5$}
    \end{subfigure}%
    ~
    \begin{subfigure}{0.32\textwidth}
        \centering
        \includegraphics[width=0.95\linewidth]{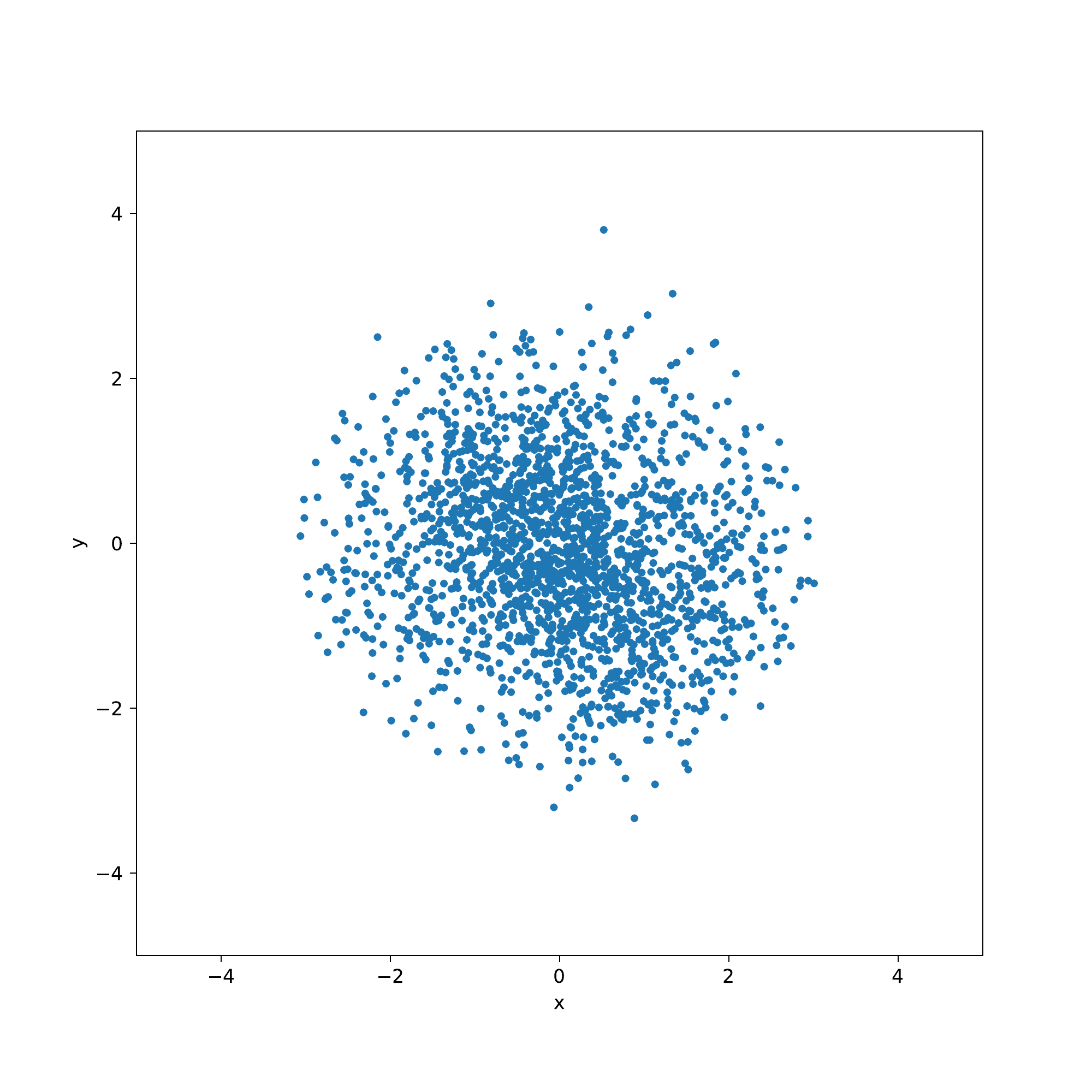}
        \caption{$t=2$}
    \end{subfigure}
    ~
    \begin{subfigure}{0.32\textwidth}
        \centering
        \includegraphics[width=0.95\linewidth]{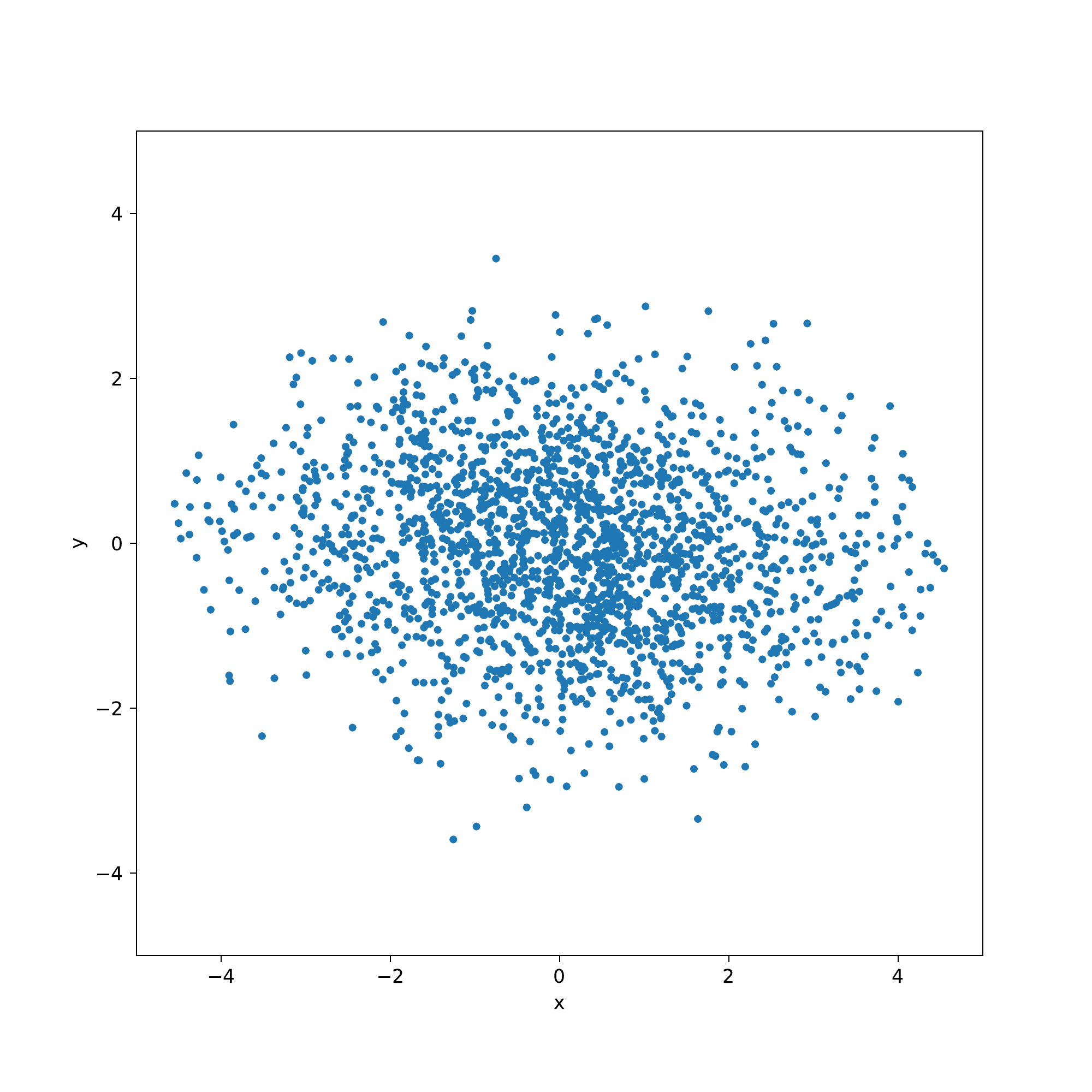}
        \caption{$t=2.5$}
    \end{subfigure}
    \caption{Sample plots of computed $\rho_{\theta}$ at different time $t$}
    \label{geodesic sampleplot}
\end{figure*}

\subsection{Quadratic potential} The second example is the WHF in which the potential is set to be quadratic function of the position variable, i.e., $\mathcal{F}(\rho)=\int_{\mathbb{R}^d} V(x)\rho(x)dx $ where $V$ is a quadratic function of $x$. The problem can be solved explicitly in this case as shown in Section \ref{exact para HO}, and the approximation error $\delta_0, \delta_1$ vanish if we choose affine transform as the push-forward map. We verify the symplectic property as well as the linear dependence between error and step size of our numerical scheme here. On the other side, the solution may display interesting structure for some special combination of $V(x)$ and $\Phi(0, x)$, which is the so called Lissajous curve. We demonstrate this case in section \ref{sec: quadratic2D affine map}.

In our experiments, the potential function $V$ and initial condition for dual variable $\Phi$ are taken as
\begin{equation}
    \label{eq:quad_potential_func}
    \begin{aligned}
    &V(x) = \sum_i\frac{1}{2}a_ix_i^2,\\
      &\Phi(0, x)= \sum_i\frac{1}{2}b_ix_i^2.
    \end{aligned}
\end{equation}
Here $a_i$ is a positive real number while $b_i$ can be negative. The index $i$ runs from $1$ to $d$. We can write the WHF as

\begin{align}
\label{HO equation}
    &\frac{\partial}{\partial t}\rho +\nabla \cdot (\rho\nabla \Phi)=0,\\
    &\frac{\partial }{\partial t}\Phi +\frac{1}{2}|\nabla\Phi|^2=-\sum_i\frac{1}{2}a_ix_i^2,\\
    &\rho(0, \cdot)=T_{\theta_0\sharp}\mathcal{N}(0, I),\quad  \Phi(0, x)= \sum_i\frac{1}{2}b_ix_i^2.
\end{align}

From the particle version of equations, we know that both the position and velocity can be expressed explicitly. In fact, the $i$-th component of solution is given as
\begin{equation}
    \label{eq:quad_potential_sol}
    \begin{aligned}
    X_i(t, x)=\sqrt{1+b_i^2}\cdot x_i\cdot \cos(\sqrt{a_i}\cdot t-\arctan(b_i)),
    \end{aligned}
\end{equation}
where $x=(x_1, \cdots, x_d)\in \mathbb{R}^d$. The true push-forward map is
\begin{align*}
    T_t(x)=A(t)x, \quad A(t)=\mathrm{diag}([\sqrt{1+b_i^2} \cos(\sqrt{a_i}\cdot t-\arctan(b_i))]_{i=1}^{d}).
\end{align*}
We can see that $a_i$ determines the frequency, while $b_i$ determines the amplitude and initial phase.

\subsubsection{2D case with affine transform as the push-forward map}
\label{sec: quadratic2D affine map}

We first test the 2D harmonic oscillator example with potential energy $\mathcal{F}(\rho)=\int_{\mathbb{R}^2} \frac{1}{2}(2.25x_1^2+0.36x_2^2)\rho(x) dx$ and initial dual function $\Phi(0, x)=-\frac{1}{2}x_1^2$ as in section \ref{exact para HO}. We choose affine transform as the push-forward map, i.e., 
\begin{align*}
    T_{\theta}(x)=\Gamma z+b, \theta=(\Gamma, b), \Gamma \in \mathbb{R}^{2\times 2}, b\in \mathbb{R}^2.
\end{align*}

We pick up random initial position $x_0$, compare the trajectory $\{T_{\theta}(x):t\in [0, 20]\}$ from our solution with the true trajectory $\{T_{t}(x):t\in [0, 20]\}$ in Fig \ref{fig:2d HO demo }.
\begin{figure}[H]
    \centering
    \includegraphics[width=0.5\linewidth]{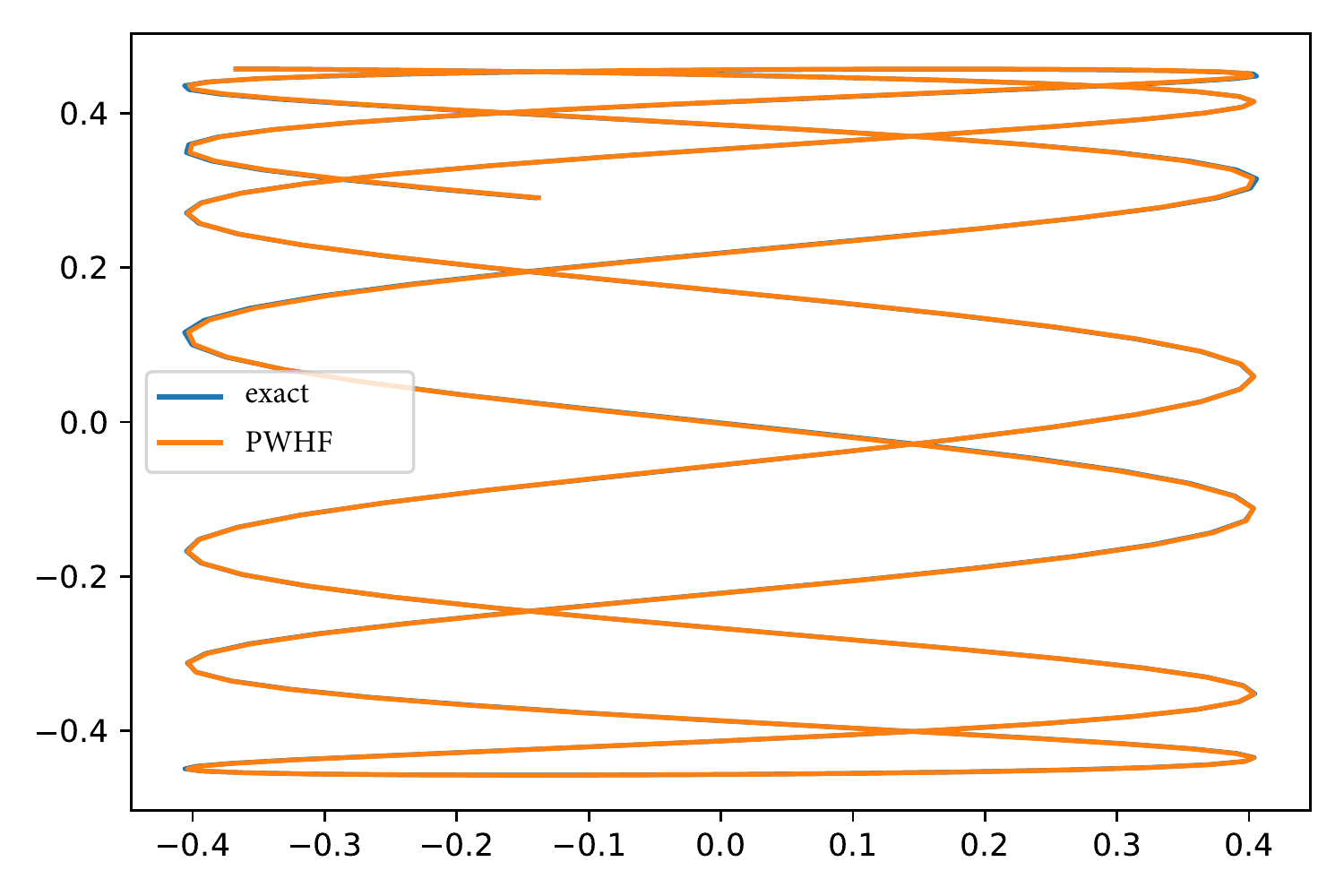}
    \caption{Trajectory of randomly picked initial point}
    \label{fig:2d HO demo }
\end{figure}

In this case, the true push-forward map lies in the space of parameterized functions $\{T_{\theta}\}_{\theta \in \Theta}$. In fact, $\theta=(A(t), \vec{0})$ gives $T_{\theta}\equiv T_t$, hence the approximation error $\delta_0, \delta_1$ equal $0$. Since we know the true solution, We define the error as
$$\widehat{\epsilon}= \max_{l\in [0, \frac{20}{h}]}\frac{1}{N_{\theta}}\sum_{i=1}^{N_{\theta}}|T_{\theta(lh)}(z_i)-T_{lh}(z_i)|. $$
We check the linear dependence between the error $\widehat{\epsilon}$ and step size $h$ in Figure \ref{err vs dt}, which confirms our theoretical estimate. 

\begin{figure}[!htb]
\centering
\minipage{0.45\textwidth}
\centering
    \includegraphics[width=0.95\linewidth]{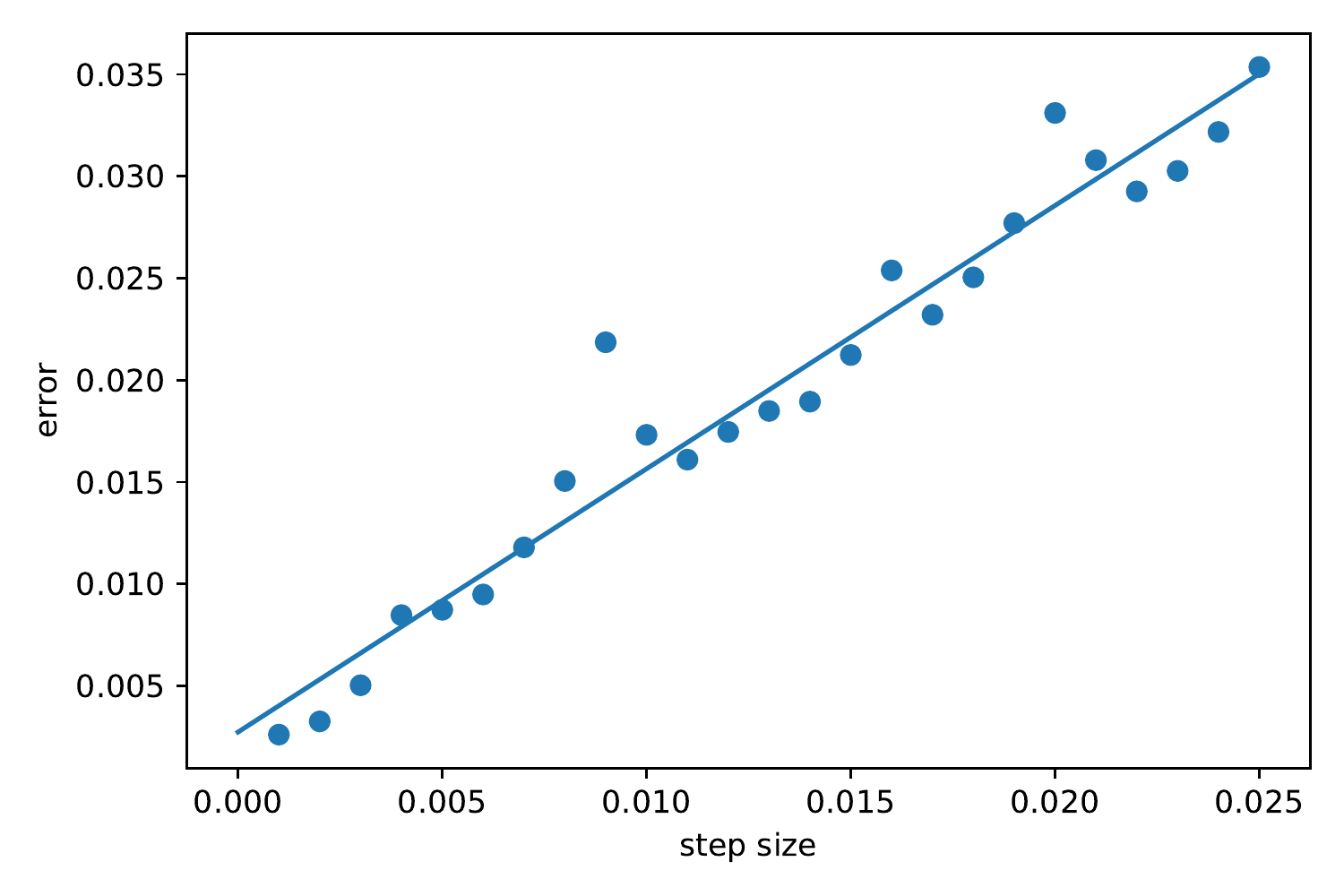}
    \caption{Error versus stepsize. We vary the stepsize from $0.001$ to $0.025$ and evaluate the error. The blue line is the line of best fit though linear regression on the data points.}
    \label{err vs dt}
\endminipage\hfill
\minipage{0.45\textwidth}
\centering
    \includegraphics[width=0.95\linewidth]{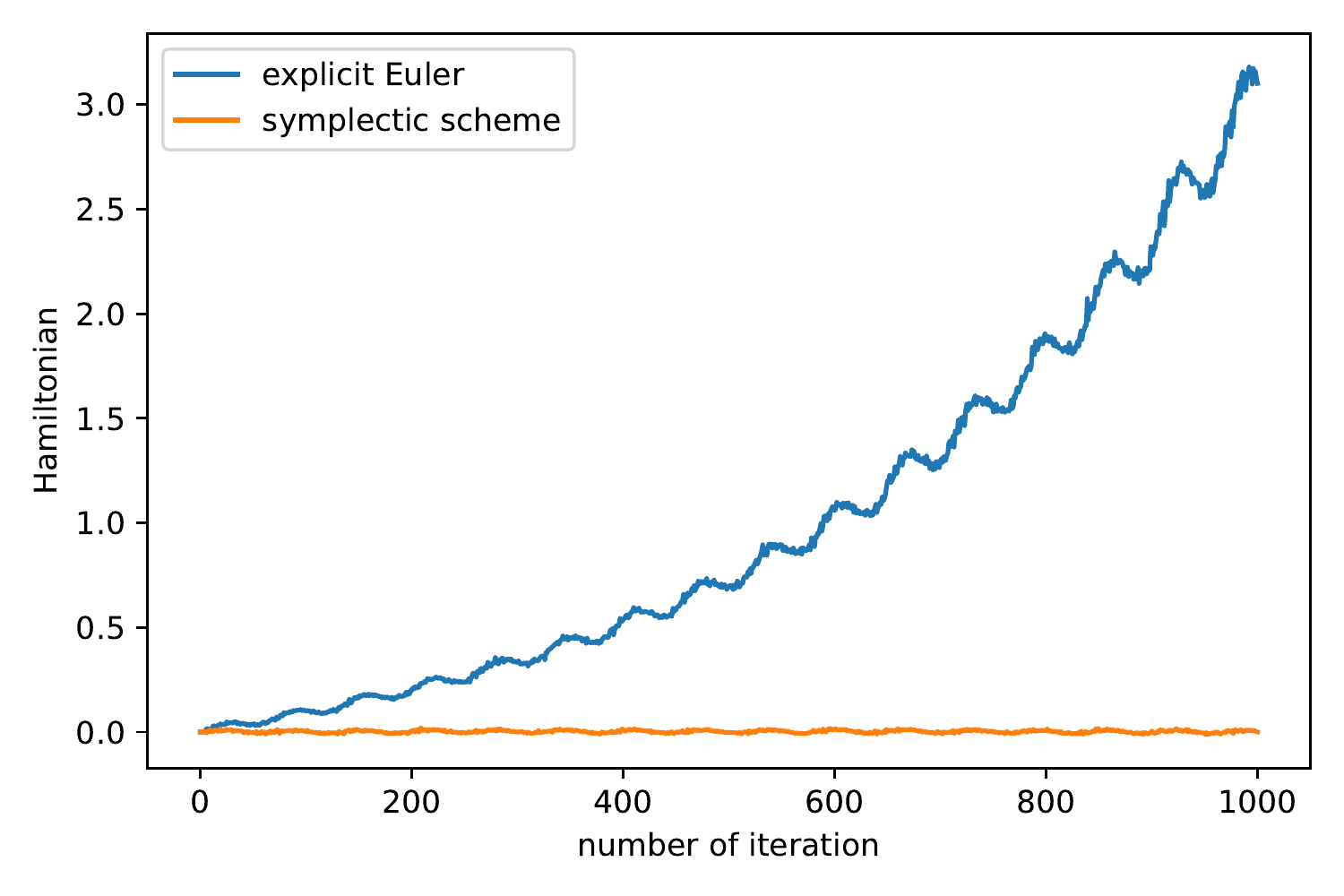}
    \caption{Conservation of Hamiltonian. The orange curve is obtained by using the proposed symplectic scheme. The blue curve is computed by forward Euler scheme.}
    \label{quadratic cons H}
\endminipage
\end{figure}

We also verify the symplectic preservation of our numerical scheme in Figure \ref{quadratic cons H}. We run Algorithm \ref{alg:HFsolver} and plot the change of Hamiltionian $\Delta H(\theta_k, p_k)$ as the orange curve. In comparison, we replace the symplectic step by a forward Euler step, and run the experiments, with corresponding Hamiltonian as the blue curve. It is clear that the proposed symplectic scheme preserves the Hamiltonian while the forward Euler scheme does not.

\subsubsection{2D case with Lissajous curve}
For a single point $x=(x_1, x_2)\in \mathbb{R}^2$, its motion under the Hamiltonian flow (\ref{HO equation}) is a $2$-dimensional harmonic oscillator. Denote $\delta \beta_0=\arctan(b_1) - \arctan(b_2)$ as the initial phase difference and $r=\sqrt{\frac{a_1}{a_2}}$ as the ratio of frequency for two components, it's well-known that we may have interesting periodic patterns known as Lissajous curve for some combination of $(\delta\beta_0, r)$. Numerical experiments show that our method can capture this structure as well. In fact, we run our algorithm for several different combination of $(\delta\beta_0, r)$, choose a suitable initial position, and plot its trajectory under the map $T_{\theta}$ in Figure \ref{2d lissajous}. To show the full structure of Lissajous curve, we run $20000$ iterations which corresponds to physical time $t_1=40$.

In this part, We use neural network as push-forward map where $T_{\theta}$ is same as in section \ref{geodesic sec} and set numerical step size $h=0.002$ in the computation.

\begin{figure*}[t!]
    \begin{subfigure}{0.3\textwidth}
        \centering
        \includegraphics[width=0.95\linewidth]{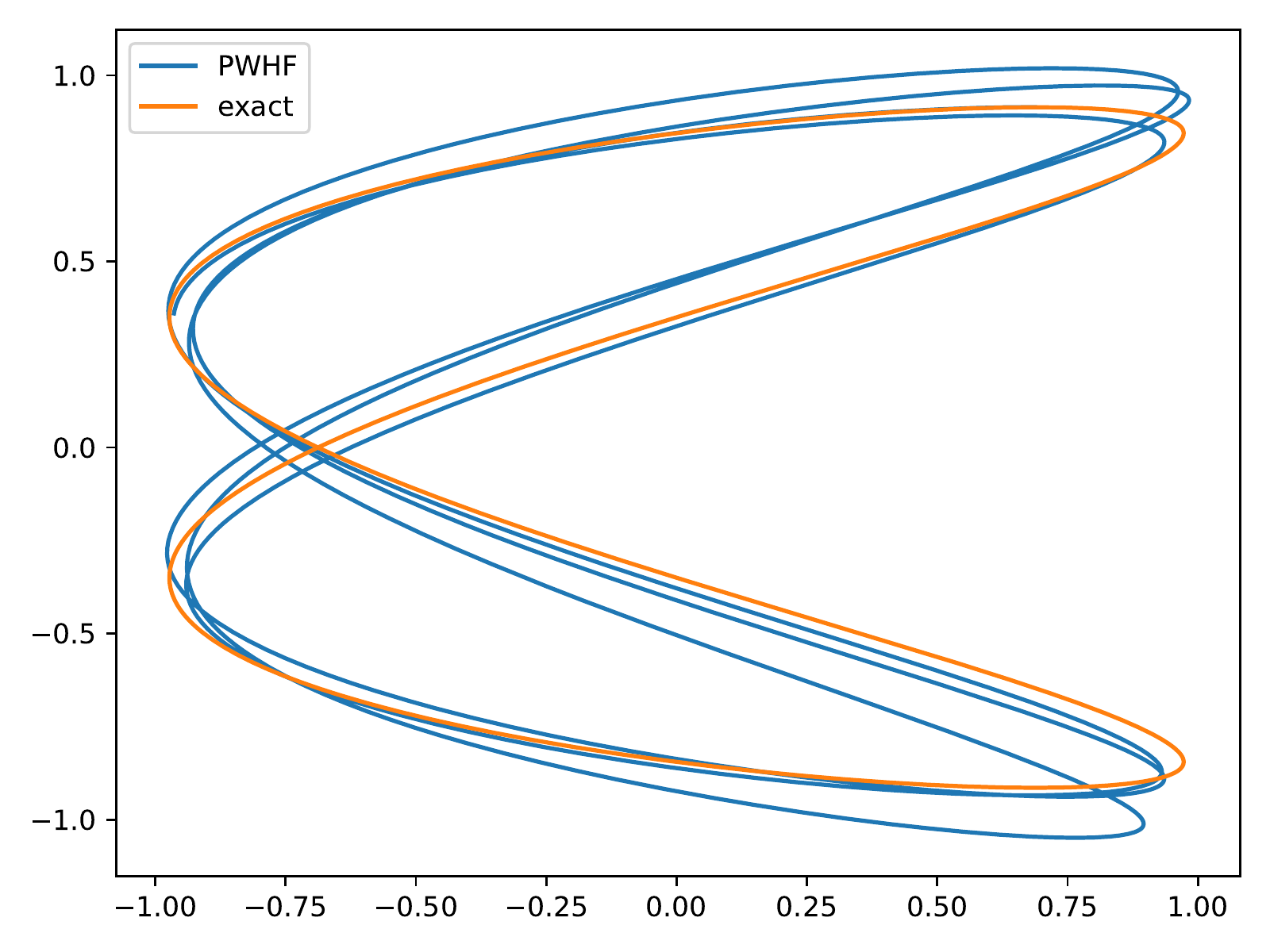}
        \caption{$\delta\beta_0=\frac{\pi}{2}, r = \frac{1}{2}$}
    \end{subfigure}%
    ~
    \begin{subfigure}{0.3\textwidth}
        \centering
        \includegraphics[width=0.95\linewidth]{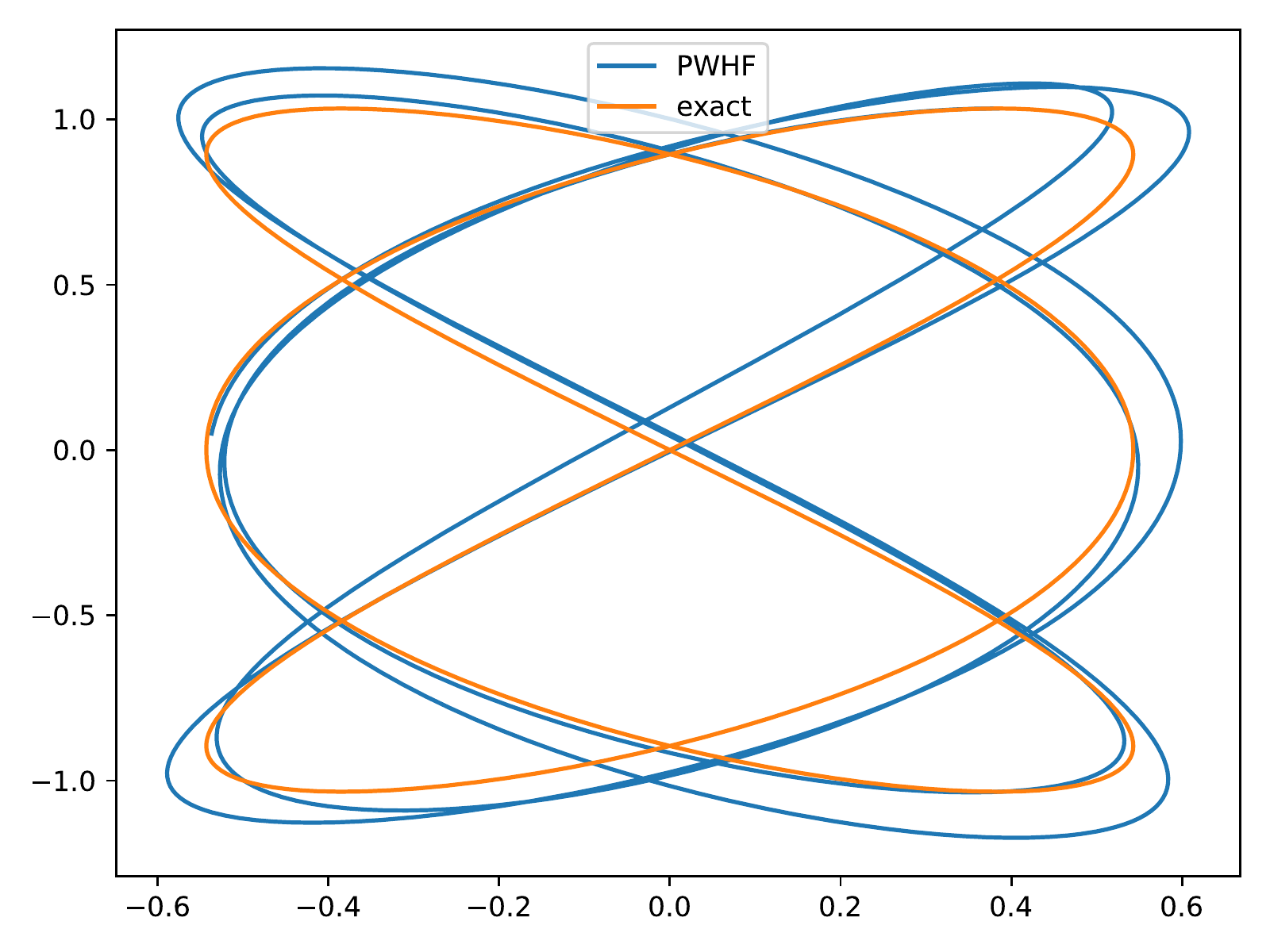}
        \caption{$\delta\beta_0=\frac{\pi}{2}, r = \frac{2}{3}$}
    \end{subfigure}
    ~
    \begin{subfigure}{0.3\textwidth}
        \centering
        \includegraphics[width=0.95\linewidth]{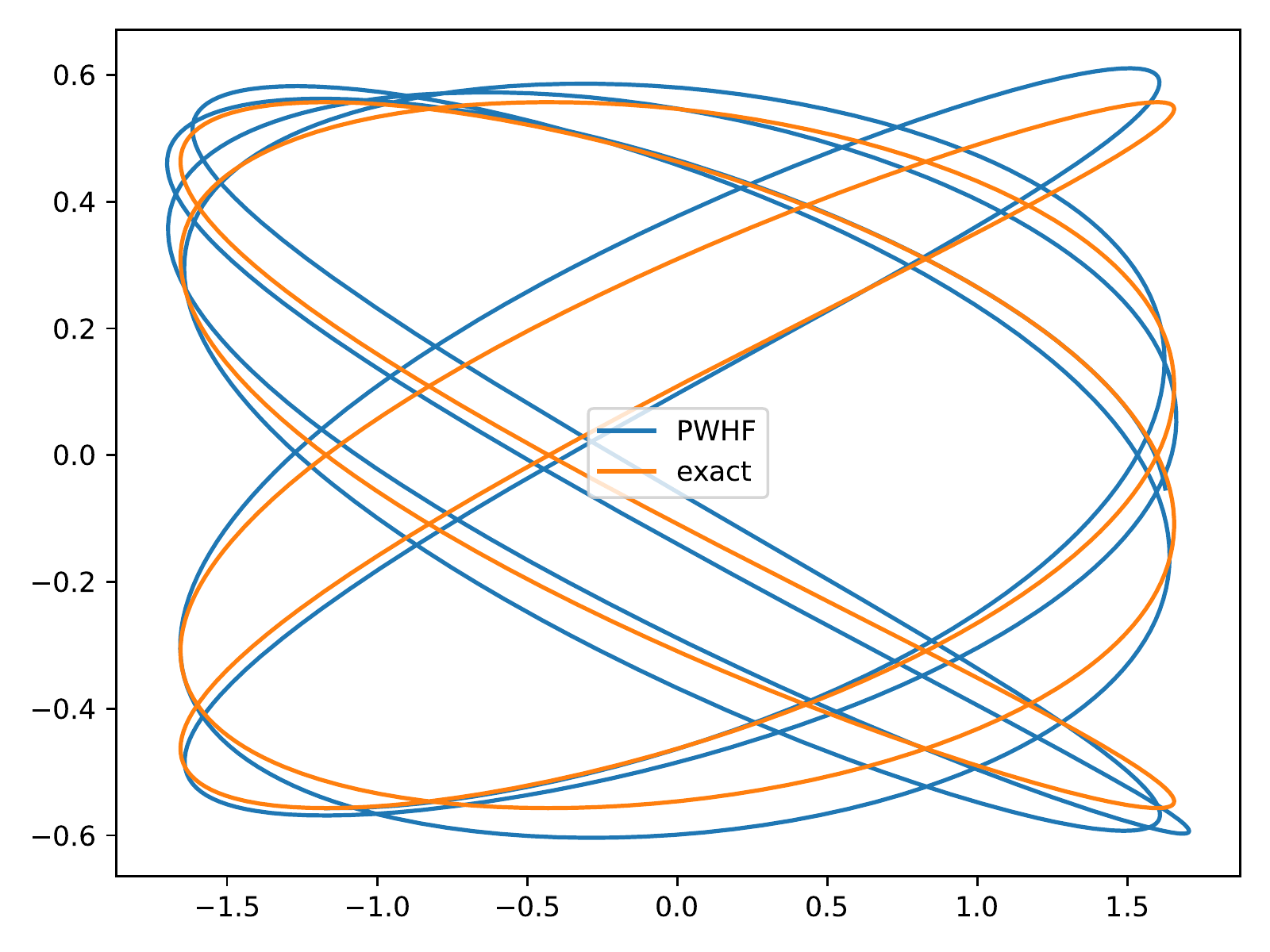}
        \caption{$\delta\beta_0=\frac{\pi}{2}, r = \frac{3}{4}$}
    \end{subfigure}
    \caption{Demonstration of Lissajous curve. (a)-(c) are trajectory plots which show the Lissajous pattern.}
    \label{2d lissajous}
\end{figure*}

We verify the velocity from our method in this subsection. We set 
\begin{align}
\mathcal{F}(\rho)=\int_{\mathbb{R}^{2}}\left( \frac{1}{2}x_1^2+\frac{1}{3}x_2^2 \right)\rho(x)~dx,\quad \Phi(0, x)=-\frac{1}{2}x_1^2.   
\end{align}

For a fixed initial point $x = T_{\theta_0}(z)$ where $z$ is from the reference distribution, its true velocity at time $t$ is $v(t, x)=\frac{d}{dt}X(t, x)$ where $X(t, x)$ is given in equation (\ref{eq:quad_potential_sol}), and the velocity from our model can be evaluated through $\Tilde{v}(t, x)=\frac{\partial}{\partial \theta}T_{\theta}(z)\cdot \dot{\theta}_t$. We choose $20$ initial points, plot their positions as well as the velocities on the plane as in Figure \ref{HO velocity plot}. The start point of each arrow represents the position and arrow itself represents its velocity. As shown in the figure, the true trajectories for different initial positions may intersect, which leads to the singularity of density function $\rho_t$. The velocity $\Tilde{v}$ from our model is still close to the true velocity despite the singularity.

\begin{figure*}[!htb]
    \begin{subfigure}{0.3\textwidth}
        \centering
        \includegraphics[width=0.95\linewidth]{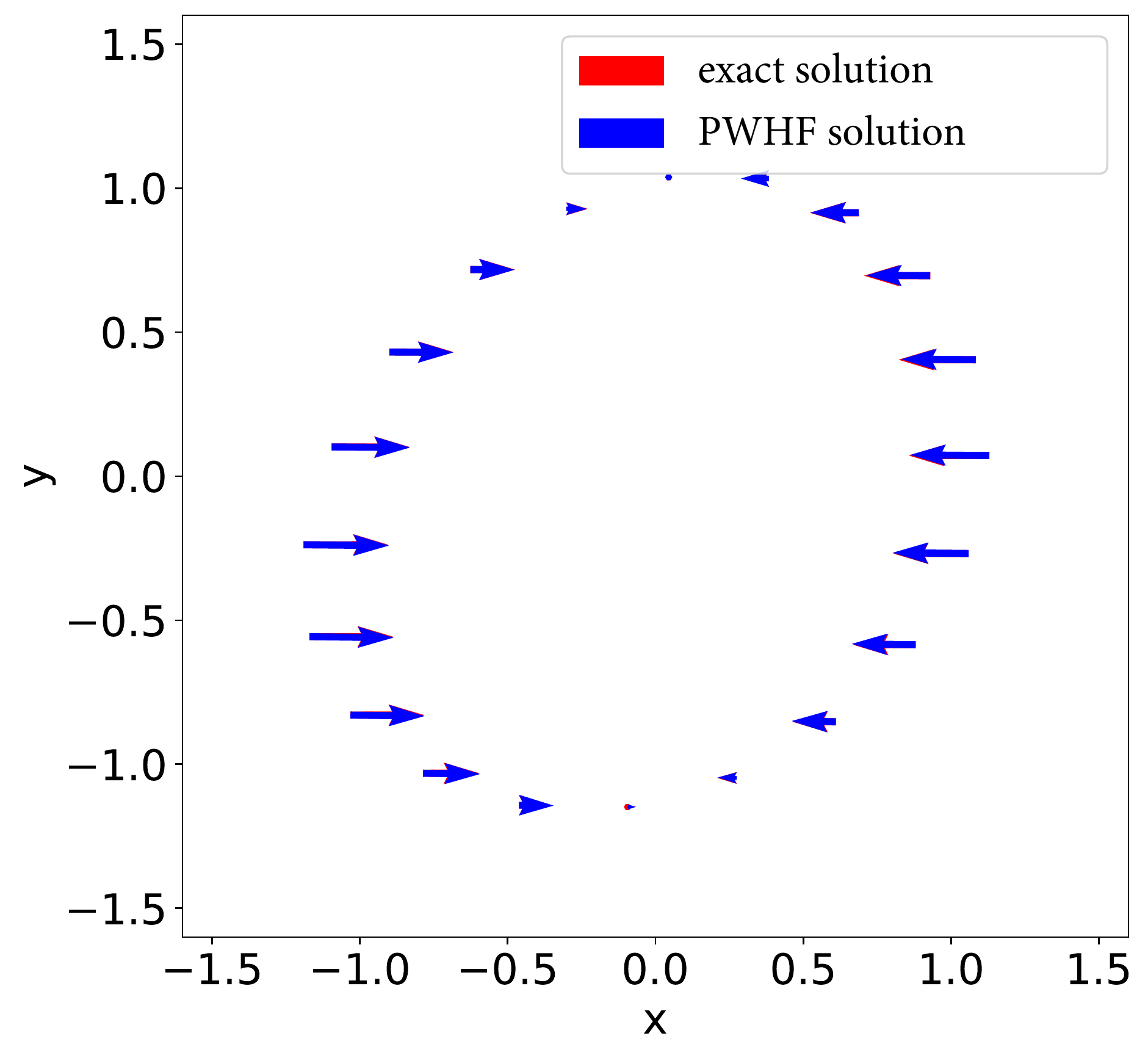}
        \caption{$t=0$}
    \end{subfigure}%
    ~
    \begin{subfigure}{0.3\textwidth}
        \centering
        \includegraphics[width=0.95\linewidth]{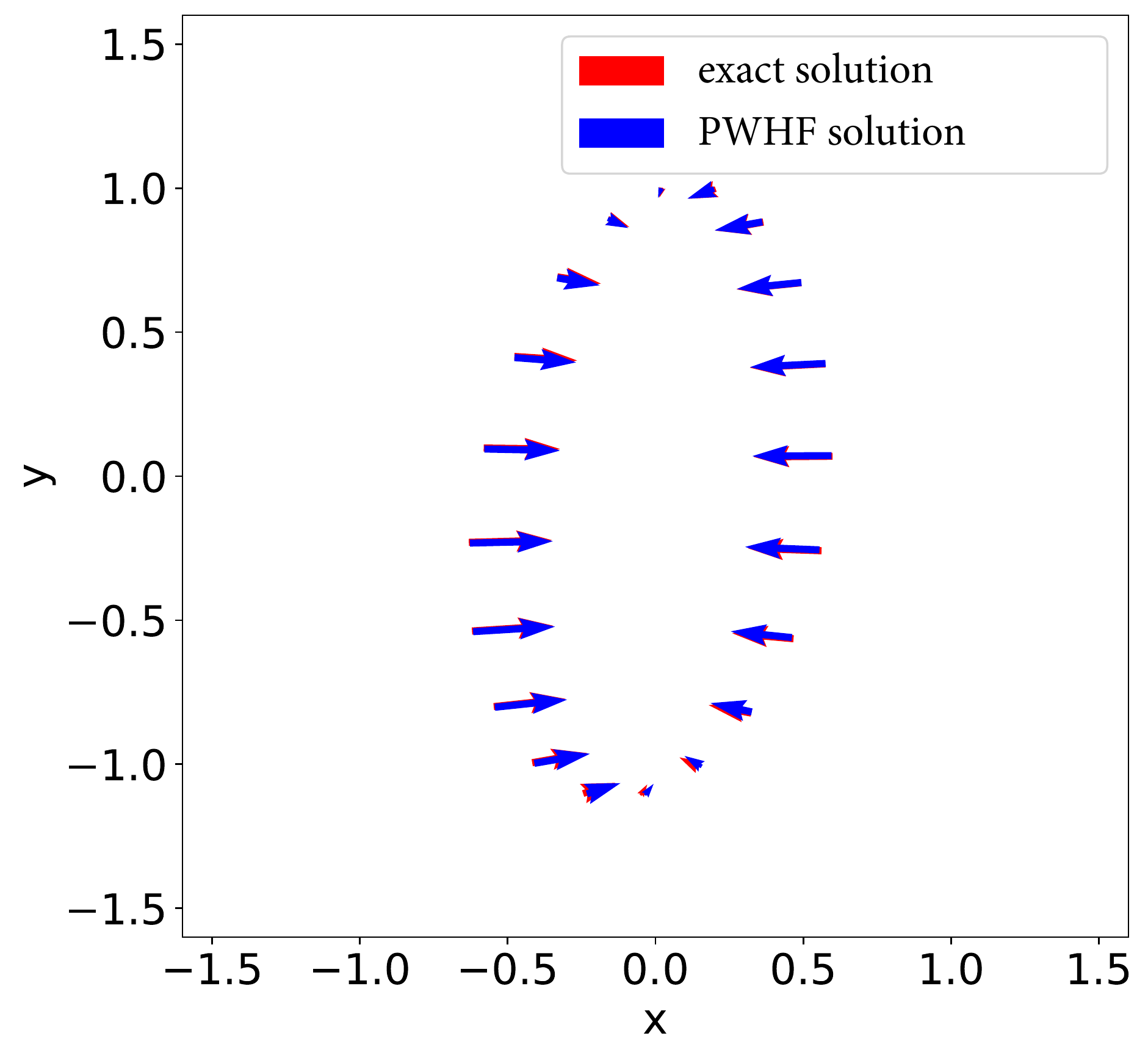}
        \caption{$t=0.4$}
    \end{subfigure}
    ~
    \begin{subfigure}{0.3\textwidth}
        \centering
        \includegraphics[width=0.95\linewidth]{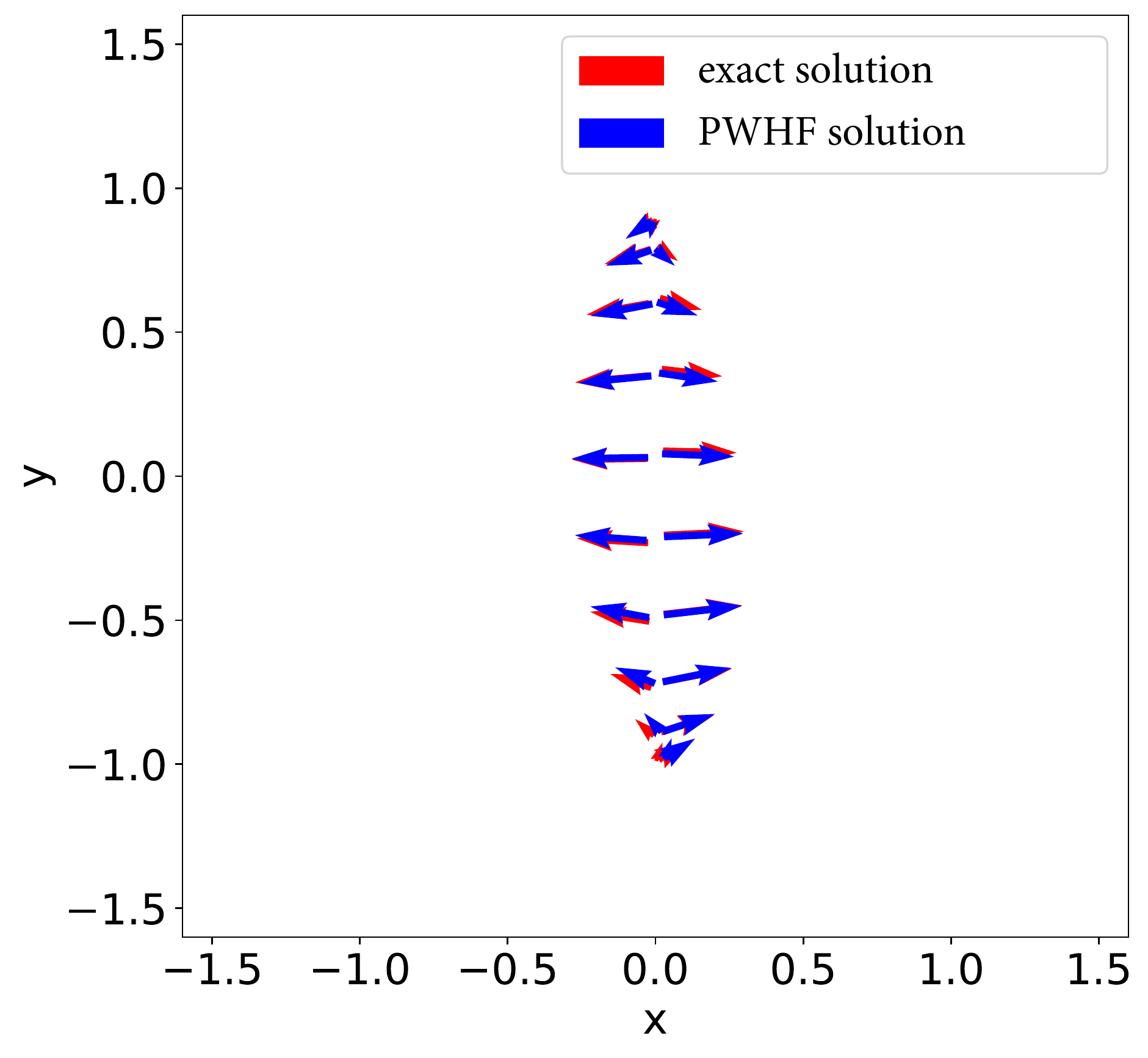}
        \caption{$t=0.8$}
    \end{subfigure}
    \\
    \begin{subfigure}{0.3\textwidth}
        \centering
        \includegraphics[width=0.95\linewidth]{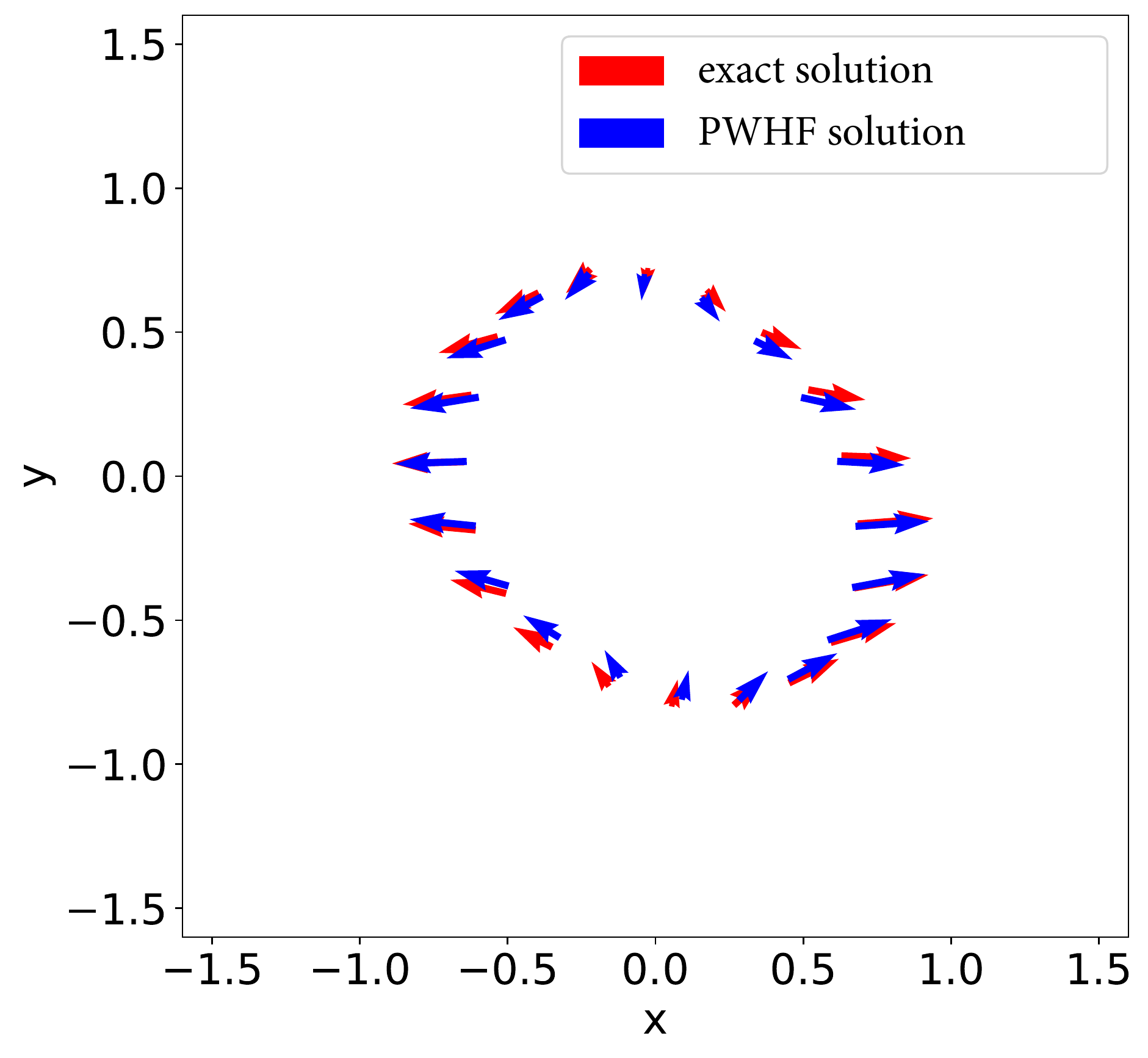}
        \caption{$t=1.2$}
    \end{subfigure}%
    ~
    \begin{subfigure}{0.3\textwidth}
        \centering
        \includegraphics[width=0.95\linewidth]{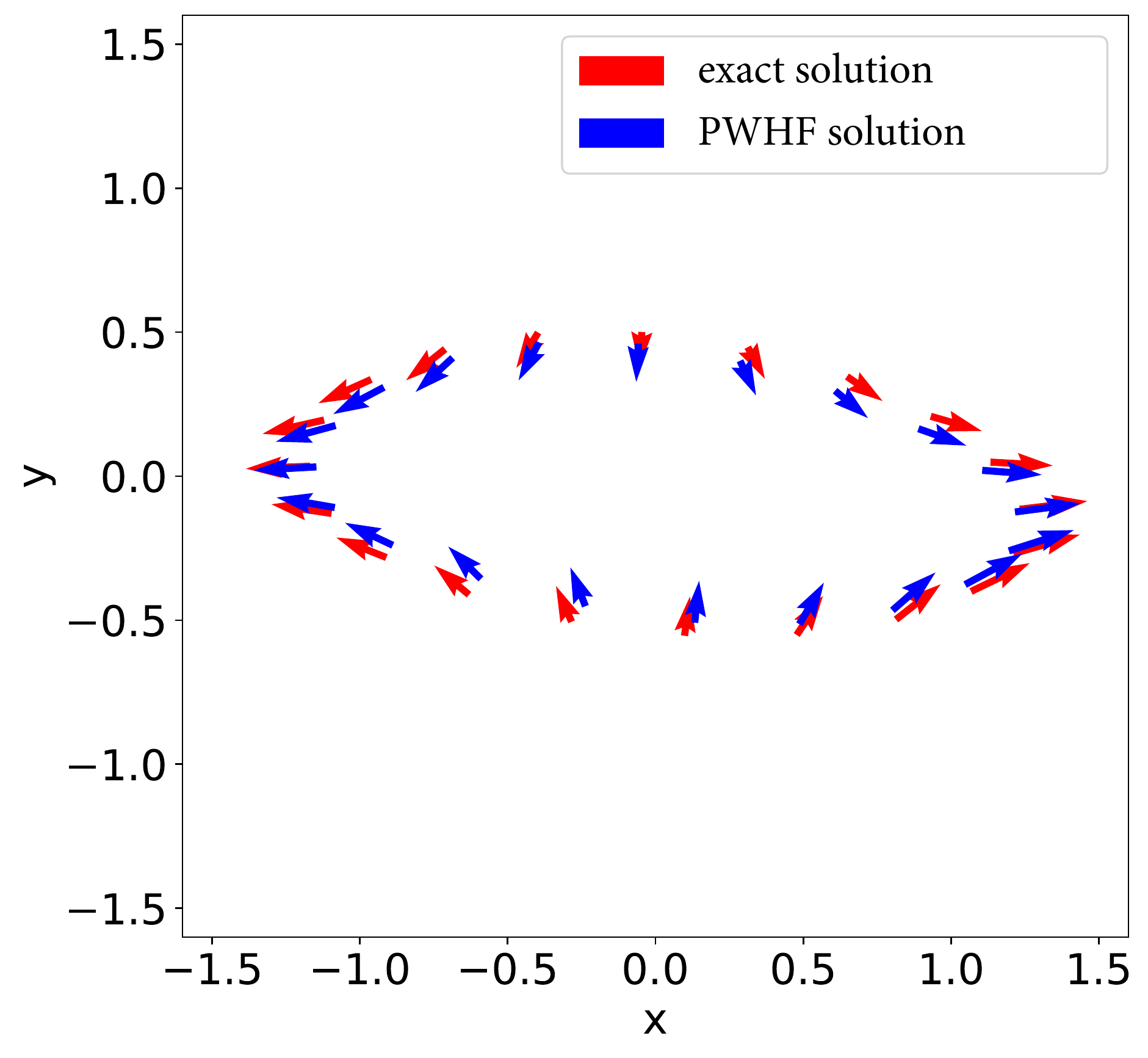}
        \caption{$t=1.6$}
    \end{subfigure}
    ~
    \begin{subfigure}{0.3\textwidth}
        \centering
        \includegraphics[width=0.95\linewidth]{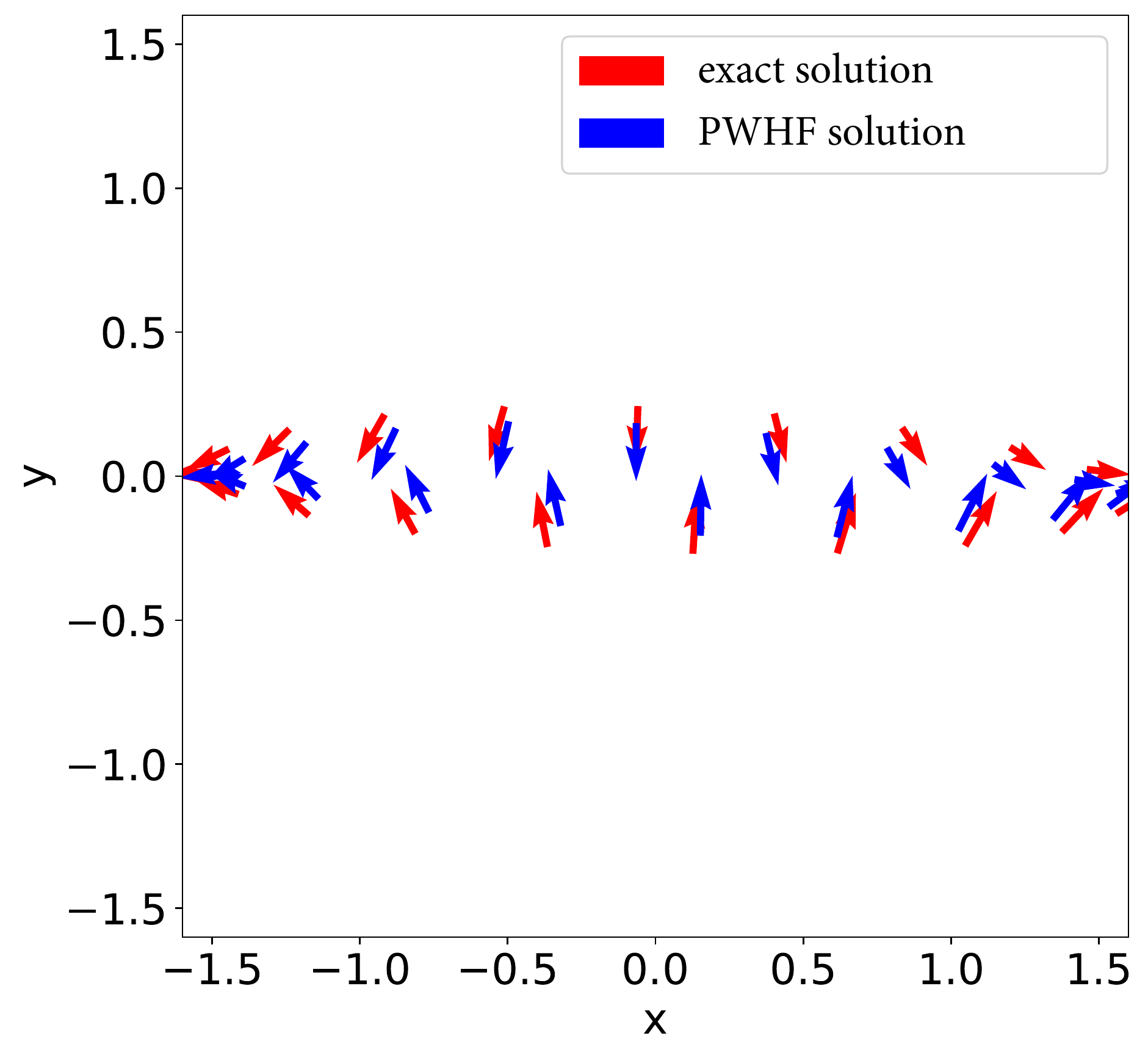}
        \caption{$t=2$}
    \end{subfigure}
    \caption{Evolution of positions and corresponding velocity}
    \label{HO velocity plot}
\end{figure*}

\subsubsection{10-D harmonic oscillator problem}
Traditional numerical schemes suffer from the curse of the dimensionality, namely the computational cost grows exponentially as the dimension increases, which makes solving high-dimensional PDEs numerically extremely expensive. On the contrary, our method is sampling based and can handle the challenges incurred from high dimensionality. To demonstrate it, we test our algorithm on  the $10$-dimensional harmonic oscillator problem with
\begin{align*}
    \mathcal{F}(\rho)=\int_{\mathbb{R}^{10}}\left(\frac{3}{8}x_1^2+\frac{1}{2}\sum_{i=2}^{10}x_i^2 \right)\rho(x)~dx,\quad \Phi(0, x)=\frac{1}{2}\sum_{i=2}^{10}x_i^2.
\end{align*}

We use the residual neural network as in (\ref{residual nn}) with $80$ neurons in each hidden layer. We solve the system on a time period $[0, 2\pi]$ with step size $h=0.001$.

We generate $10000$ samples from $\rho_{\theta}$, plot the samples histograms (orange) projected to the second dimension and compare it with true density function $\rho$ in Figure \ref{quad 10d histplot}. Here the projection of $\rho$ onto the second dimension is a Gaussian distribution whose variance is a $cosine$ function of time, depicted by the blue curve in the figure. Similar to the geodesic equation example, $\rho$ develops finite time singularity in this harmonic oscillator case, while numerical results show that our algorithm solves the problem well despite the singularity and dimension. We plot the empirical Hamiltonian from these samples in Figure \ref{10d Hamiltonian}. We emphasize that the scale used in Figure \ref{10d Hamiltonian} is concentrated around $9.35$ with small variation. This indicates that the Hamiltonian is essentially preserved while the kinetic and potential energy oscillate in the opposite phase in the computation as shown in Figures  \ref{10d KE} and \ref{10d PE} respectively.

\begin{figure*}[t!]
    \begin{subfigure}{0.24\textwidth}
        \centering
        \includegraphics[width=0.95\linewidth]{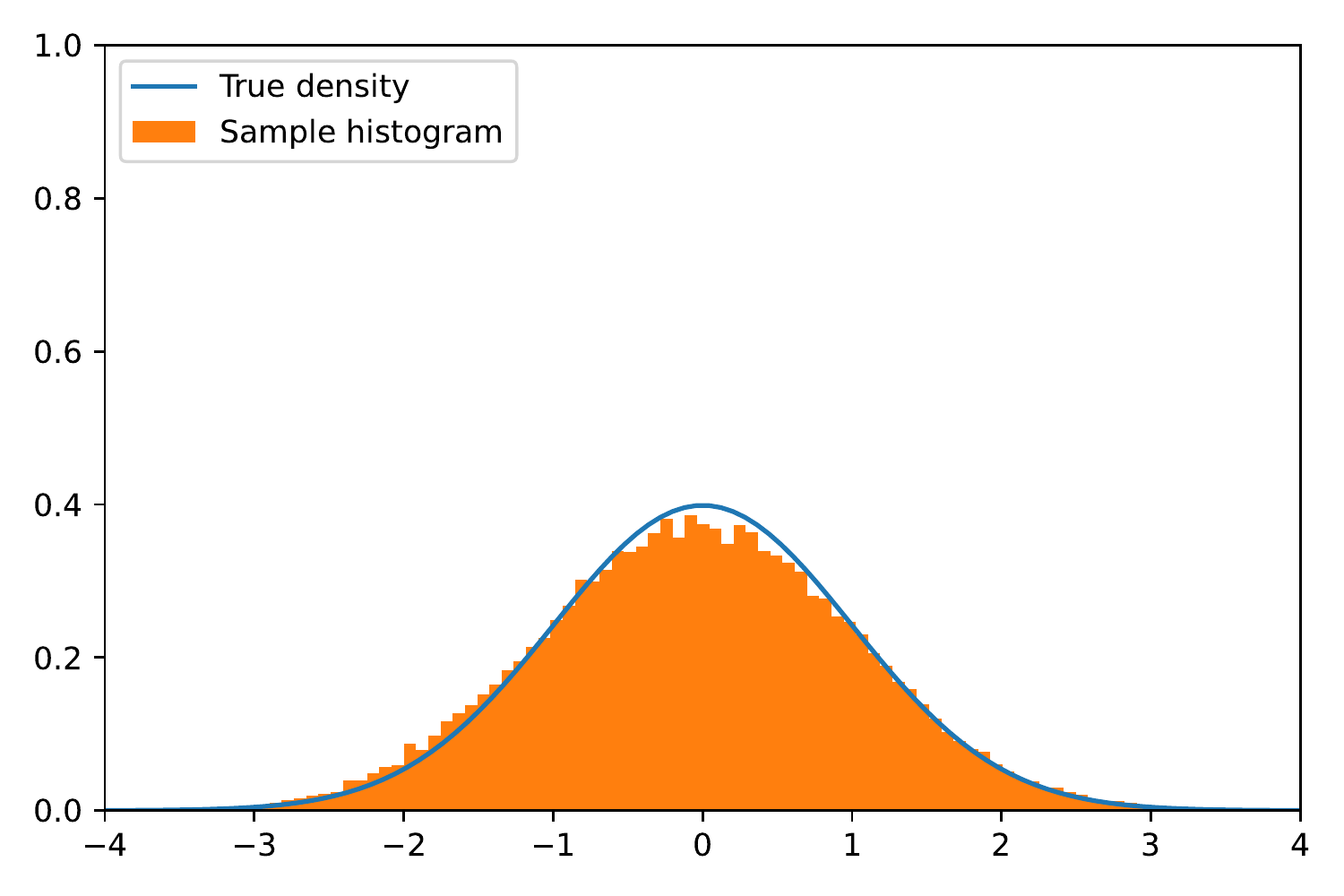}
        \caption{$t=0.$}
    \end{subfigure}%
    \begin{subfigure}{0.24\textwidth}
        \centering
        \includegraphics[width=0.95\linewidth]{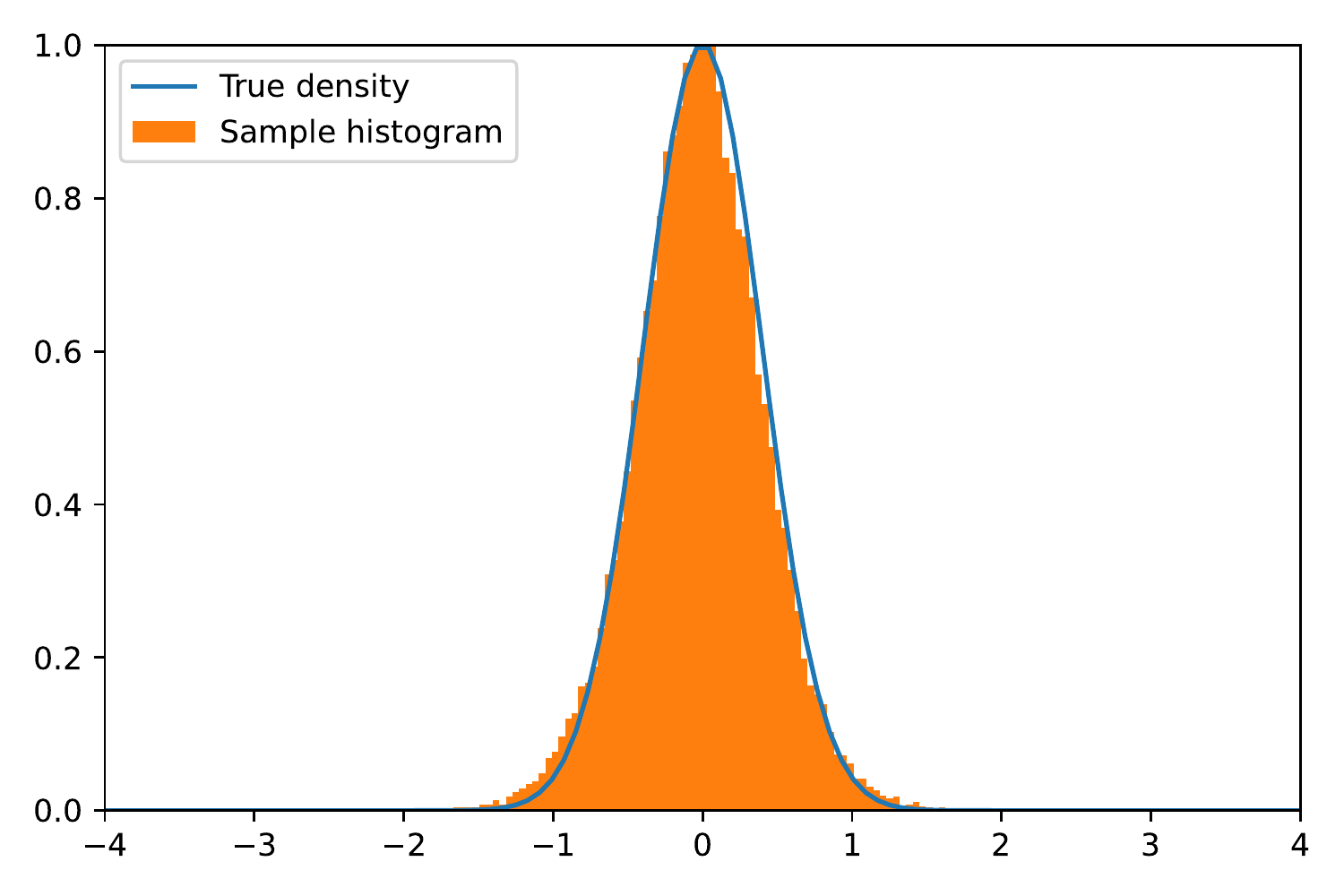}
        \caption{$t=0.5$}
    \end{subfigure}
    \begin{subfigure}{0.24\textwidth}
        \centering
        \includegraphics[width=0.95\linewidth]{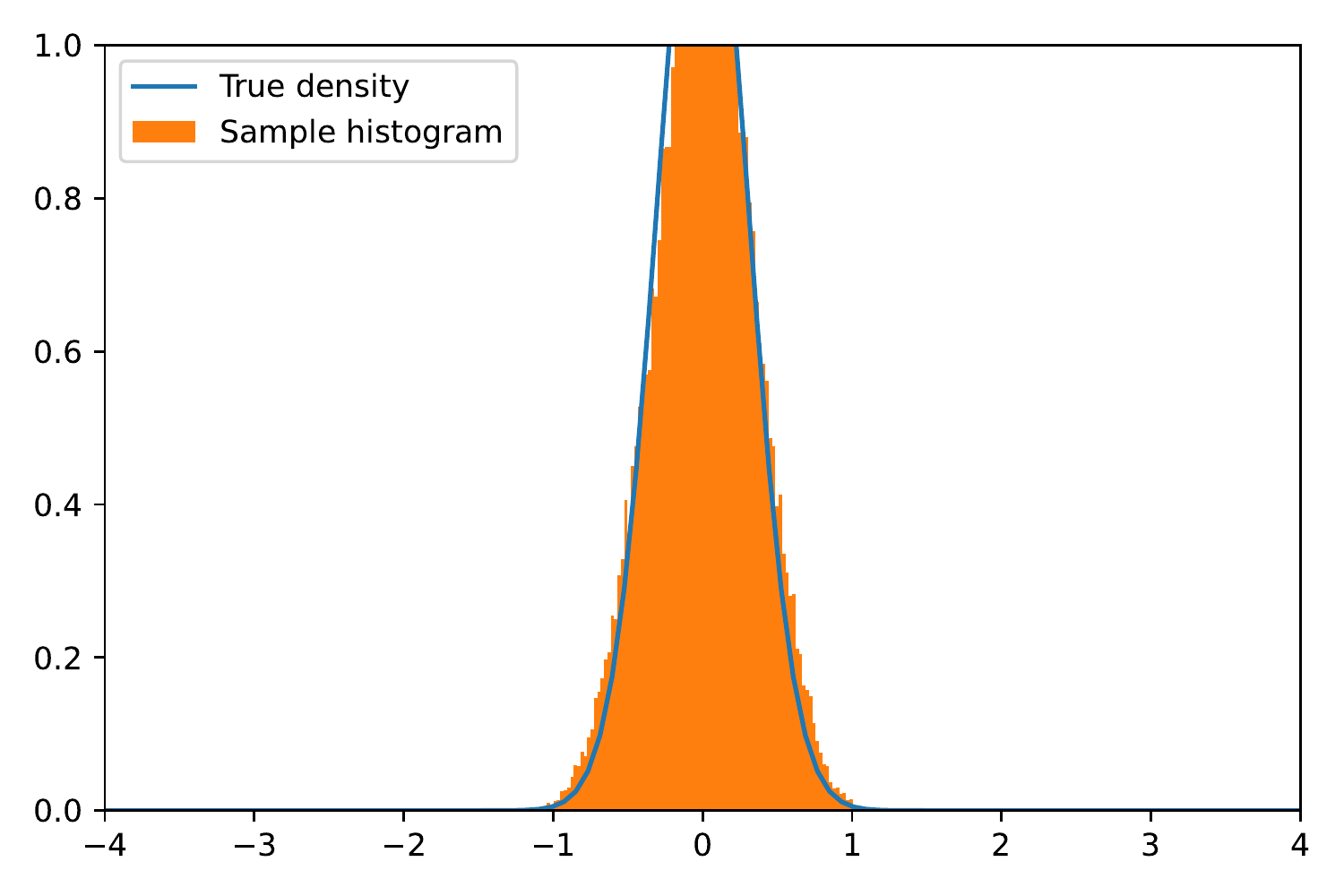}
        \caption{$t=1$}
    \end{subfigure}
    \begin{subfigure}{0.24\textwidth}
        \centering
        \includegraphics[width=0.95\linewidth]{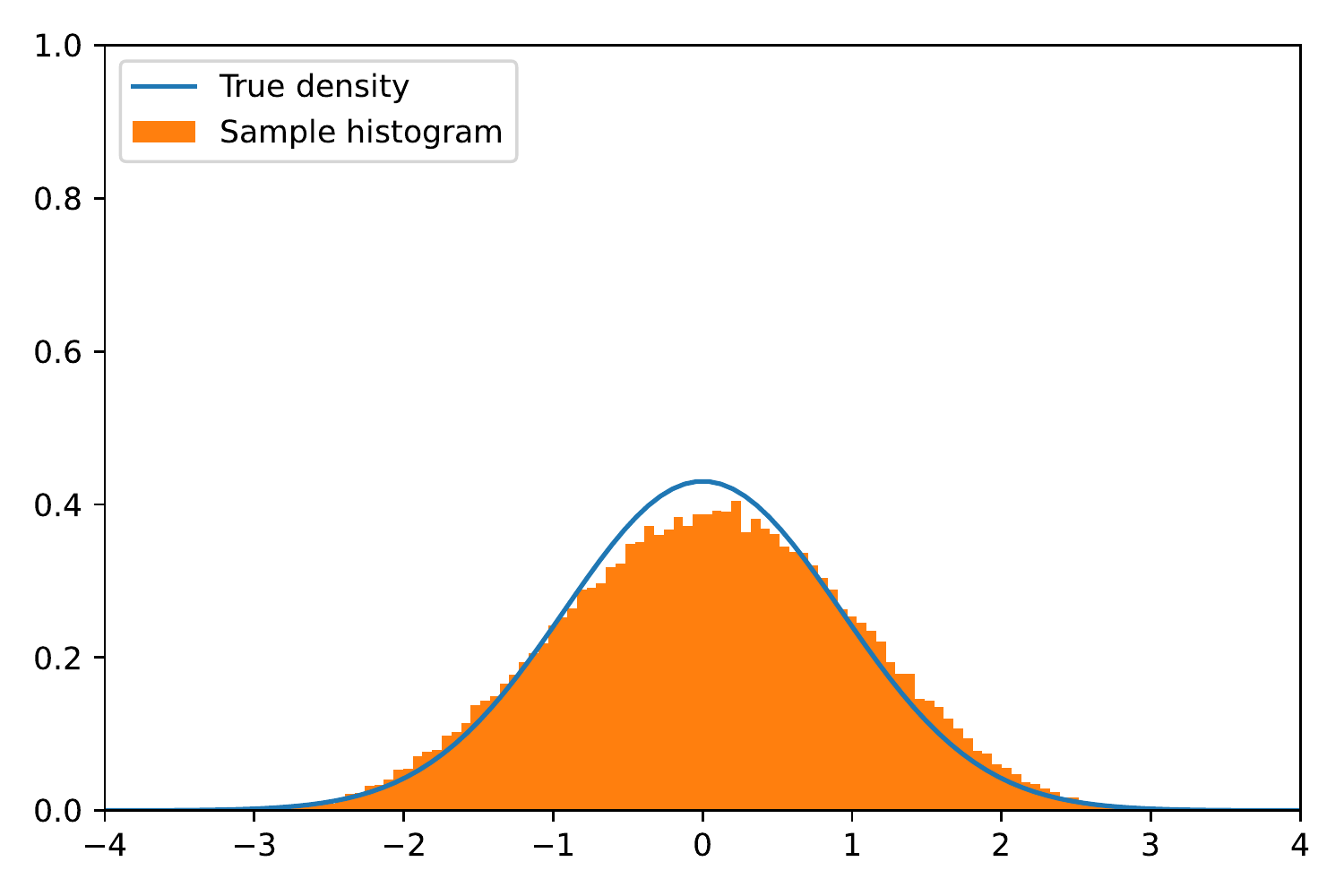}
        \caption{$t=1.5$}
    \end{subfigure}
    \\
       \begin{subfigure}{0.24\textwidth}
        \centering
        \includegraphics[width=0.95\linewidth]{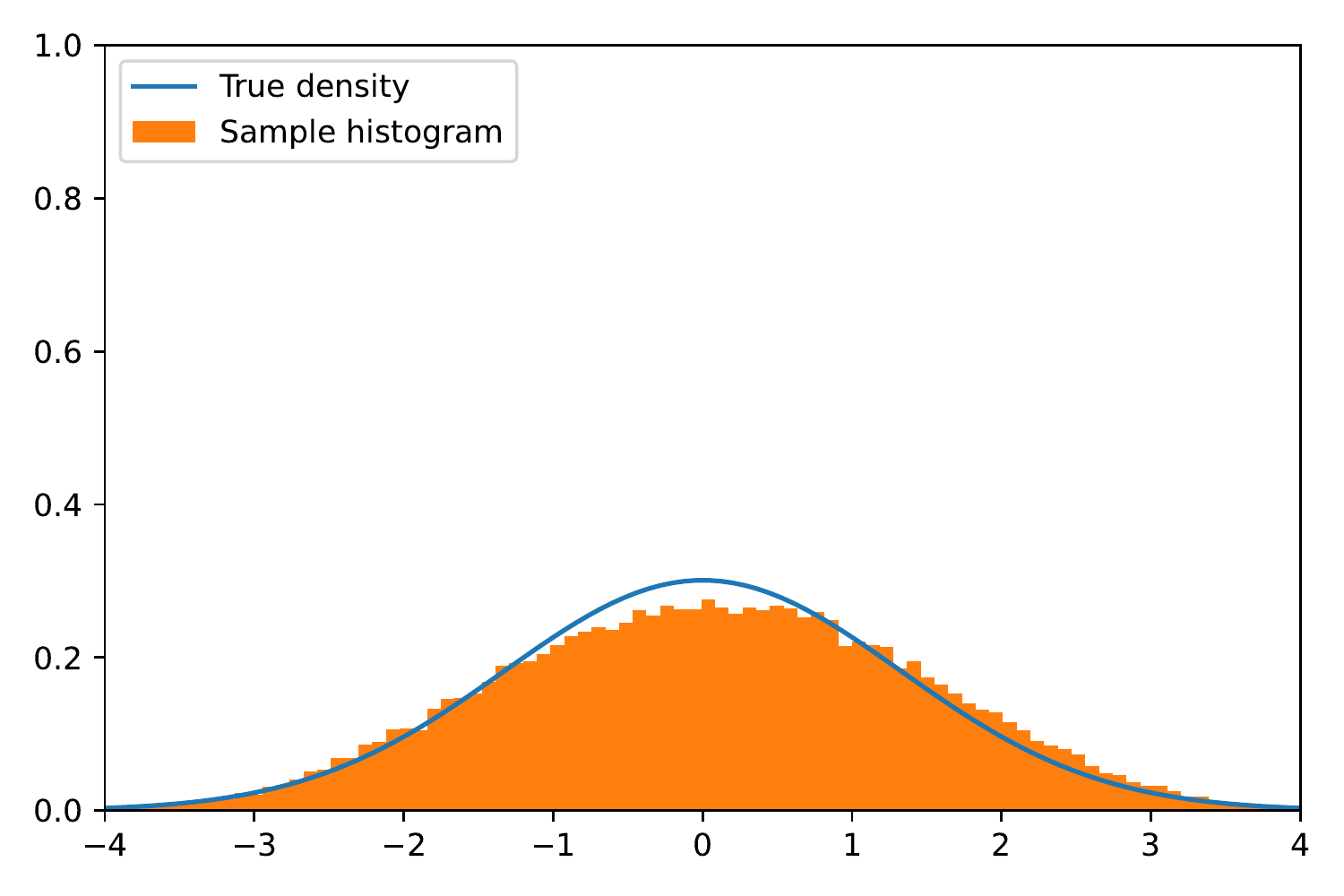}
        \caption{$t=2$}
    \end{subfigure}%
    \begin{subfigure}{0.24\textwidth}
        \centering
        \includegraphics[width=0.95\linewidth]{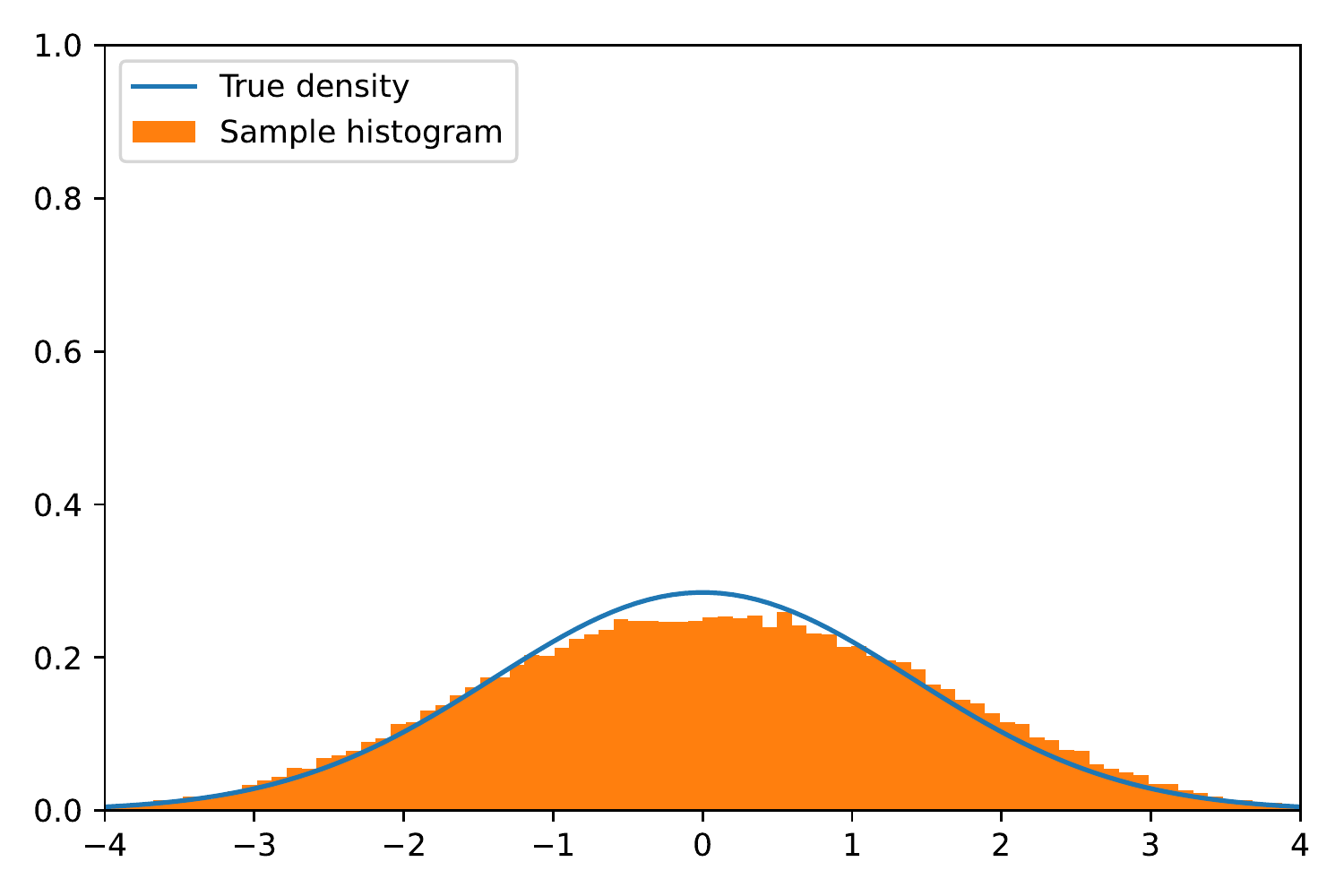}
        \caption{$t=2.5$}
    \end{subfigure}
    \begin{subfigure}{0.24\textwidth}
        \centering
        \includegraphics[width=0.95\linewidth]{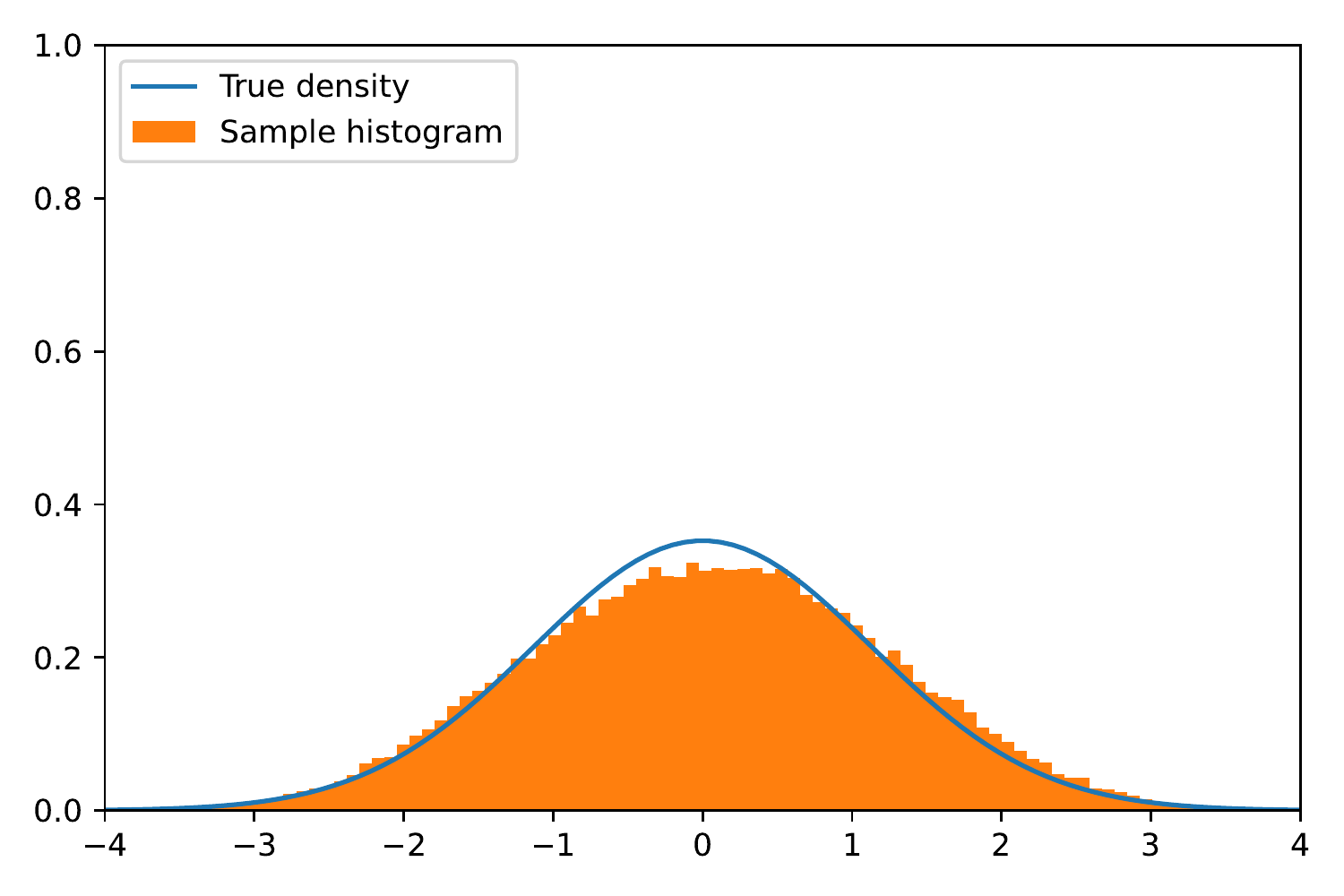}
        \caption{$t=3$}
    \end{subfigure}
    \begin{subfigure}{0.24\textwidth}
        \centering
        \includegraphics[width=0.95\linewidth]{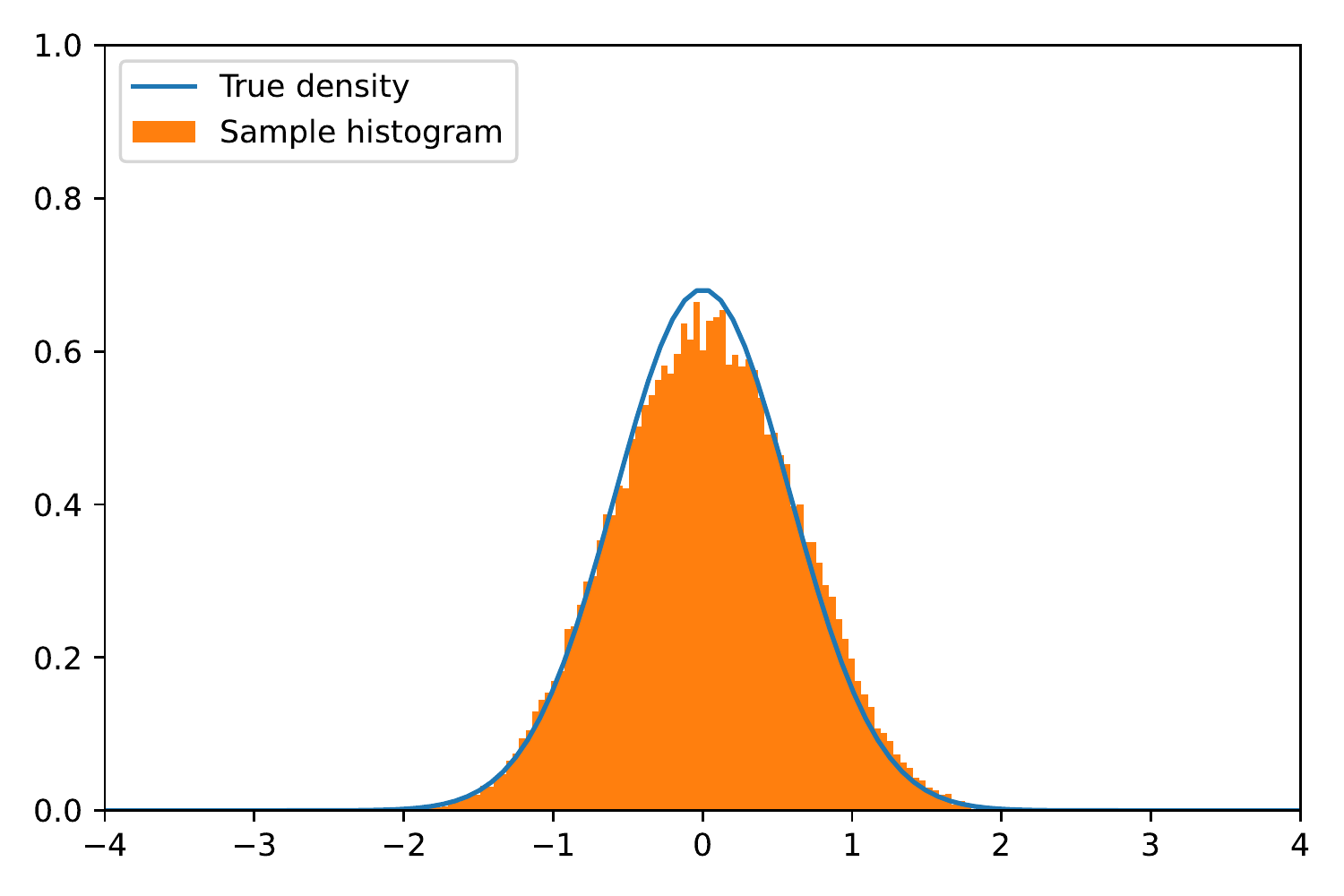}
        \caption{$t=3.5$}
    \end{subfigure}
    \caption{Time evolution of the projected histogram for $10$-d harmonic oscillator example}
    \label{quad 10d histplot}
\end{figure*}



\begin{figure}[!htb]
\centering
\minipage{0.3\textwidth}
\centering
  \includegraphics[width=0.95\linewidth]{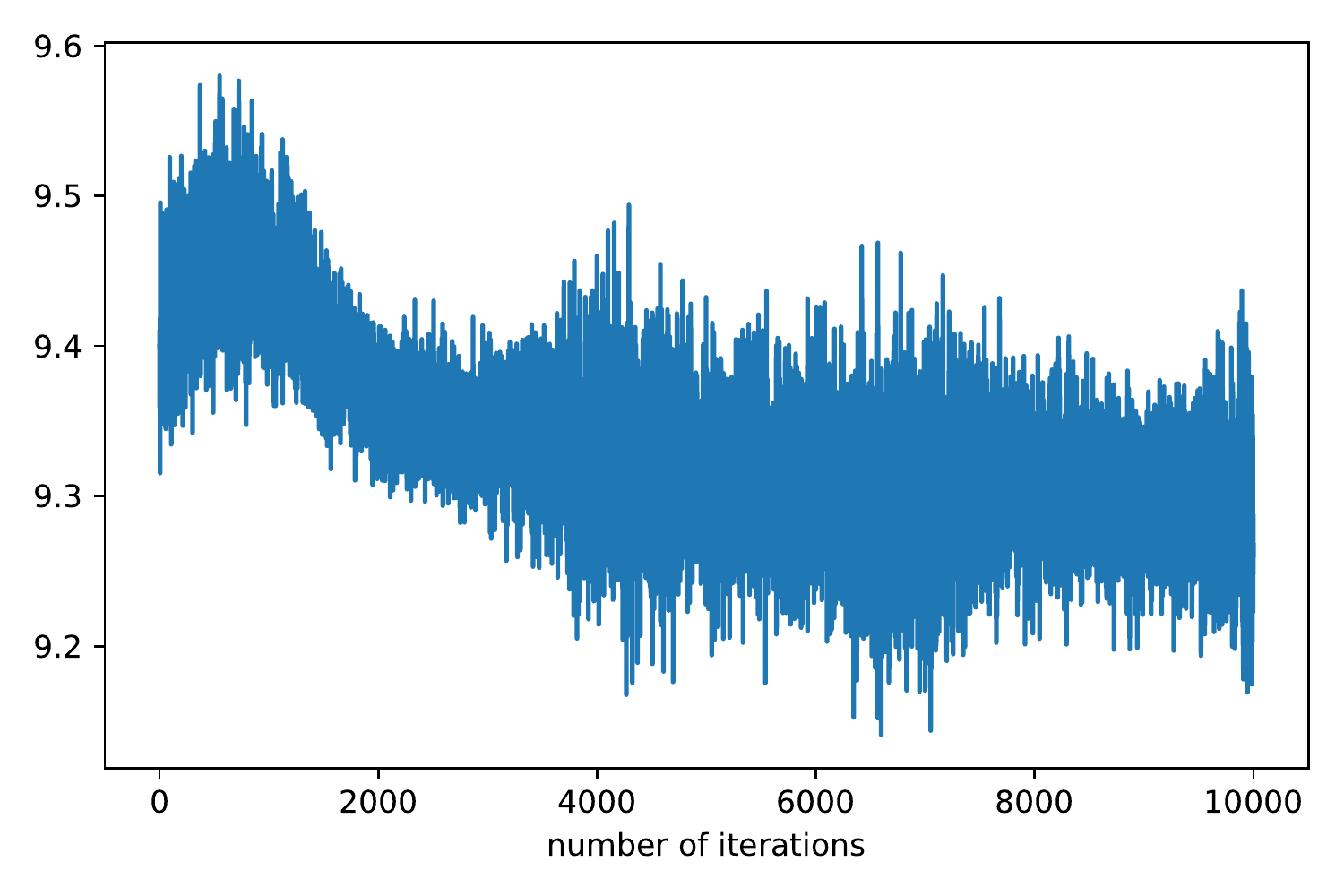}
  \caption{Hamiltonian}
  \label{10d Hamiltonian}
\endminipage\hfill
\minipage{0.3\textwidth}
\centering
  \includegraphics[width=0.95\linewidth]{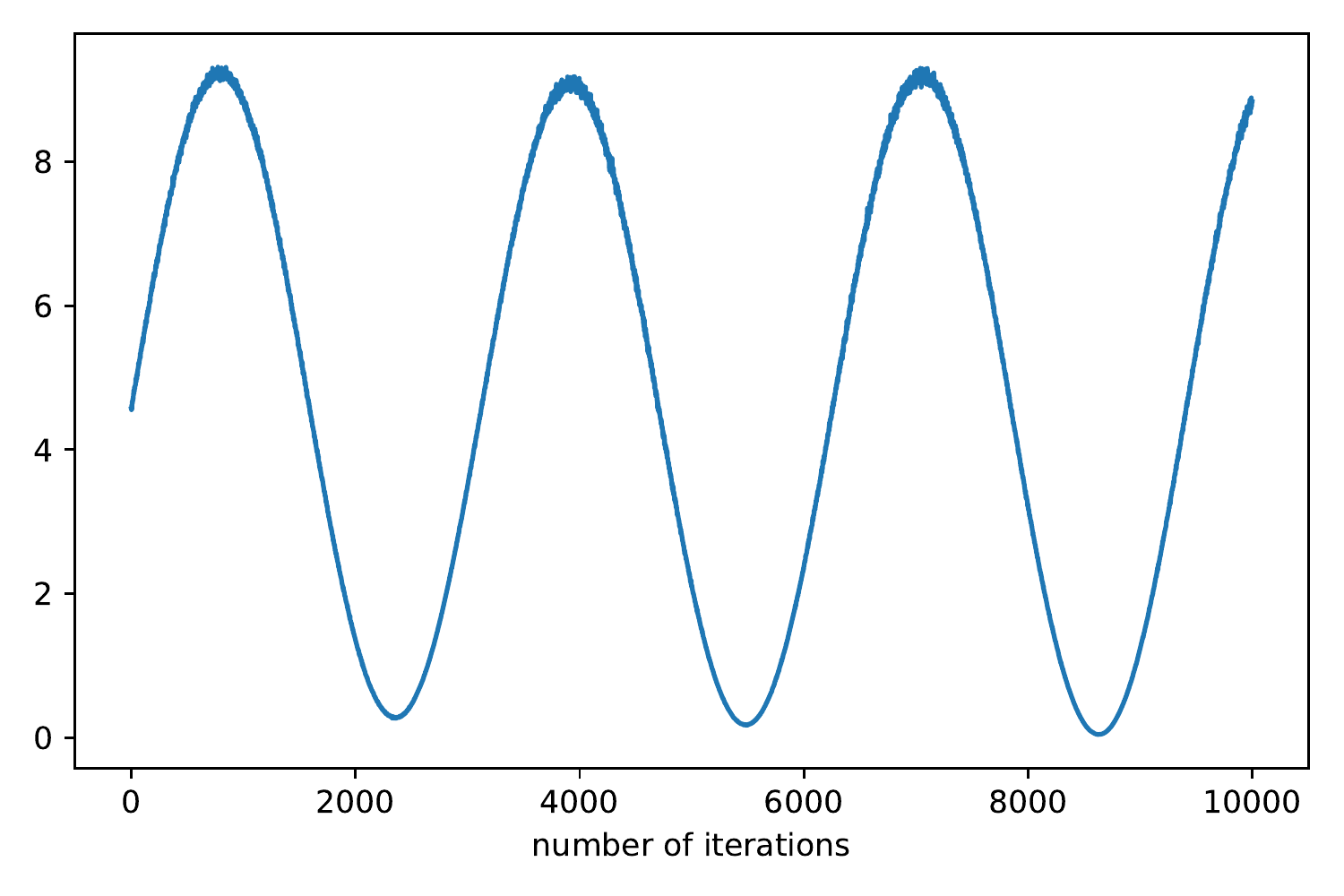}
  \caption{Kinetic energy}
  \label{10d KE}
\endminipage\hfill
\minipage{0.3\textwidth}
\centering
  \includegraphics[width=0.95\linewidth]{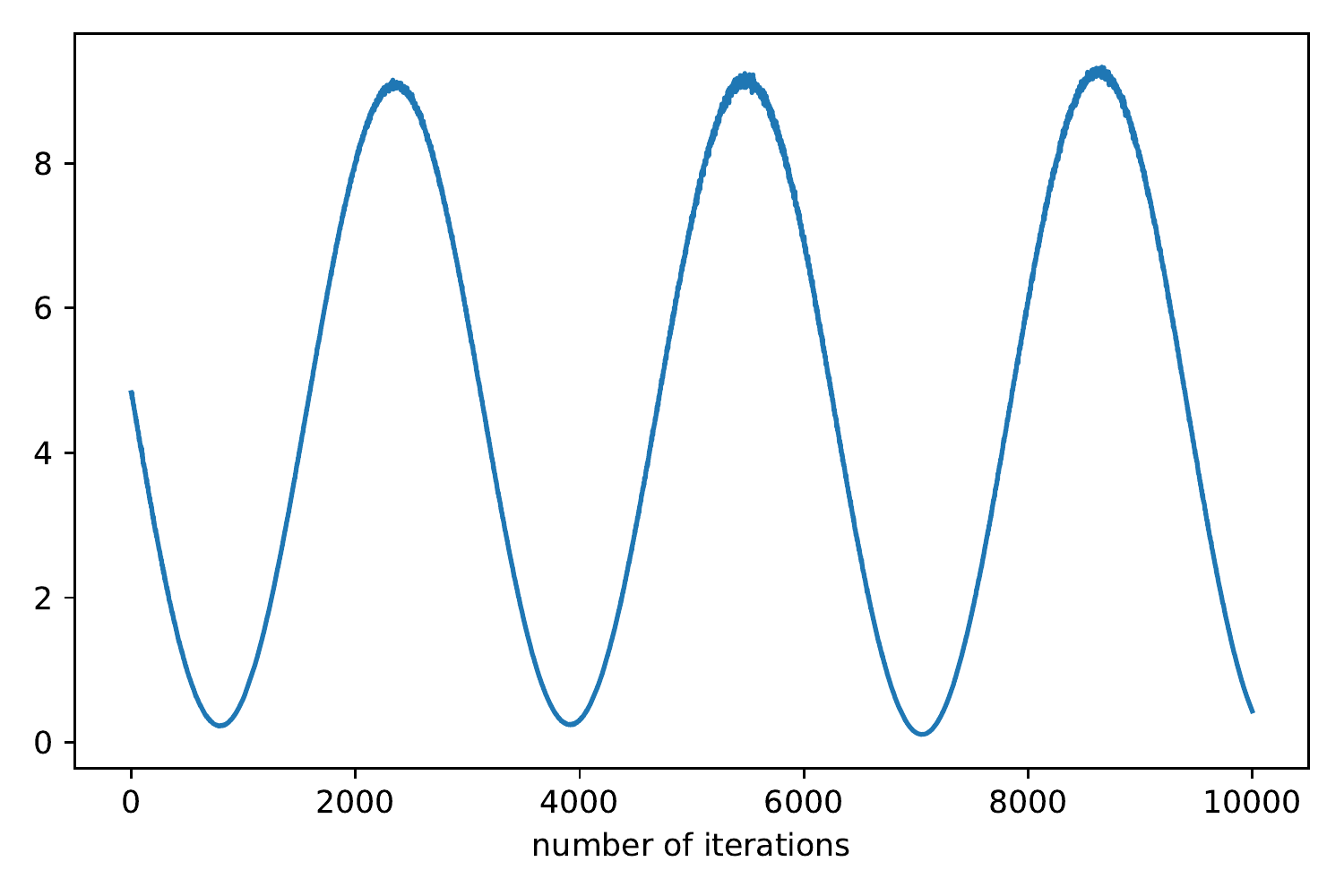}
  \caption{Potential energy}
  \label{10d PE}
\endminipage
\end{figure}


\subsection{Interaction potential}
We consider some nonlinear potentials, such as the interaction potential or entropy in the reminder of this section. The nonlinear terms are generally hard to evaluate directly through the traditional methods, although some algorithms such as kernel density estimation \cite{terrell1992variable, 10.1214/10-AOS799} can be used for it. On the contrary, push-forward maps may provide a good alternative for this purpose.

Interaction potential is described through the following potential energy $\mathcal{F}$: 
\begin{equation}
    \label{eq:coulomb-potential}
    \begin{aligned}
    \mathcal{F}(\rho)=\iint_{\mathbb{R}\times\mathbb{R}} C(x, y)\rho(x)\rho(y)dxdy.
    \end{aligned}
\end{equation}
$\mathcal{F}$ involves double integral of $\rho$ and introduces nonlinearity, which makes the problem much harder to solve in general. In physics, $C(x, y)=\frac{a}{|x-y|^2}$ is often used to model gravity or Coulomb force. To avoid the numerical instability we will use a modified $C_b$ with $b=0.1$ in the following experiment:
\begin{equation}
    \label{eq:interaction-func}
    \begin{aligned}
    C_b(x, y)=\frac{1}{b+|x-y|^2}.
    \end{aligned}
\end{equation}
We choose the same forward map $T_{\theta}$ as the neural network described in section \ref{geodesic sec}. For this problem we don't have a close form solution, so we compare the result with particle level numerical simulations. More precisely, we generate 10000 random samples and run the particle level dynamics to get numerical approximation to $T_t$, then compare our results $T_{\theta}(x)$ with the numerical results. We pick an initial point $x$, plot its trajectory from our model $T_{\theta}$ as well as the trajectory from numerical simulation in Figure \ref{fig: interaction potential demo}, which shows good agreement.
\begin{figure}[H]
    \centering
    \includegraphics[width=0.5\linewidth]{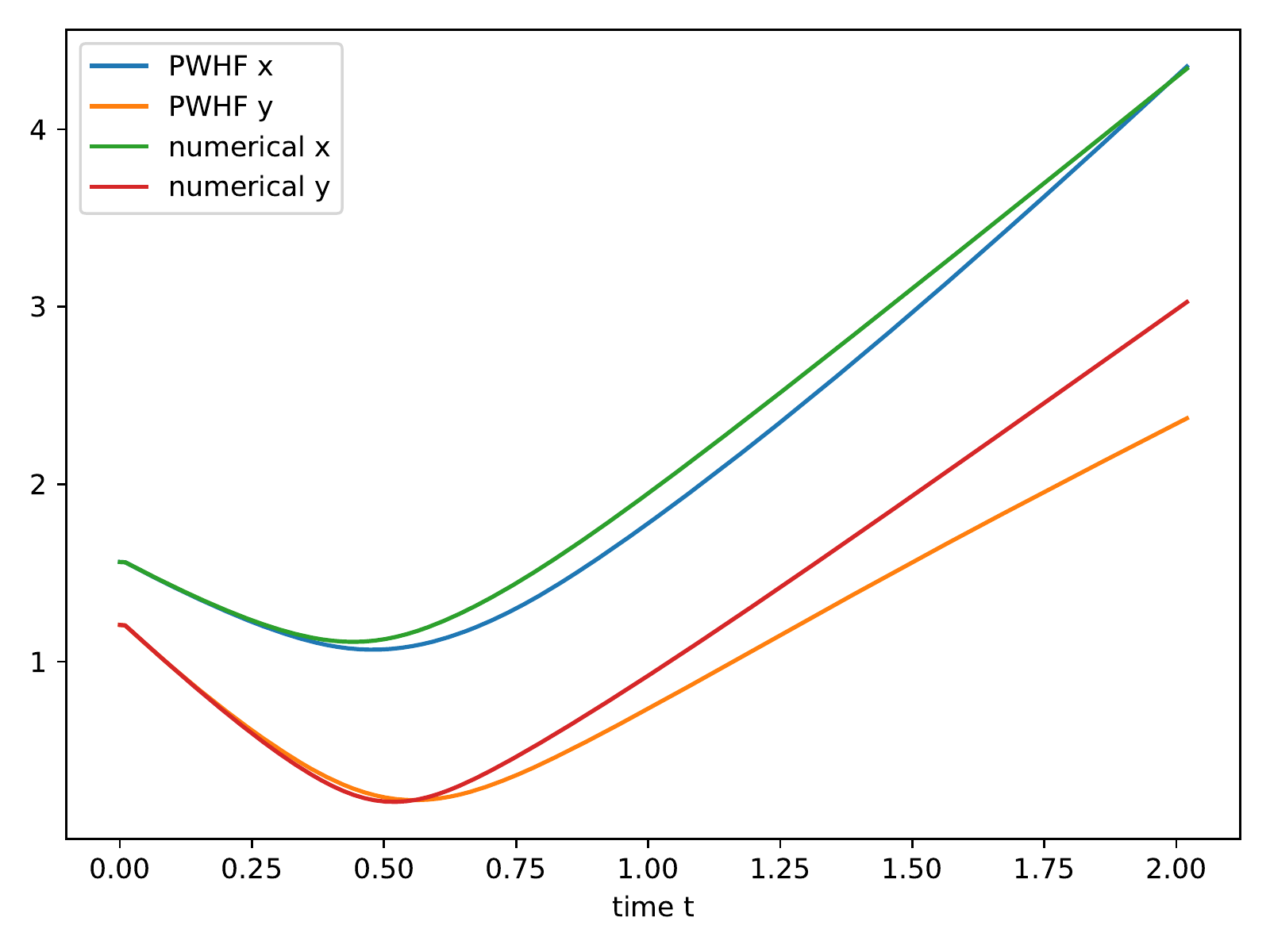}
    \caption{Projection of trajectory in interaction potential model}
    \label{fig: interaction potential demo}
\end{figure}

\subsection{Entropic potential}
Another example of nonlinear potential is
\begin{equation}
    \label{eq:entropy-potential}
    \begin{aligned}
    \mathcal{F}_{\textrm{entropy}}(\rho)=\int \rho(x)\log\rho(x)dx .
    \end{aligned}
\end{equation}
If we take diagonal map $T_{\theta}(x)=D_{\theta}x$ as the push-forward map where $D_{\theta}$ is a diagonal matrix and $\theta$ is its diagonal element, then the parameter dynamics can be solved exactly as shown in Section \ref{exact para Entropy}.
Again we solve the PWHF and compare our solution with true solution. The exact solution can be expressed as $T_t(x) = D(t)x$ with $D(t)$ defined in (\ref{entropy true para}). Figure \ref{fig: entropy params demo} shows the results from our PWHF (blue) against the exact solution (orange). They are nearly identical, which is expected because the error is close to zero according to the theoretical estimates.  

\begin{figure}[H]
    \centering
    \includegraphics[width=0.5\linewidth]{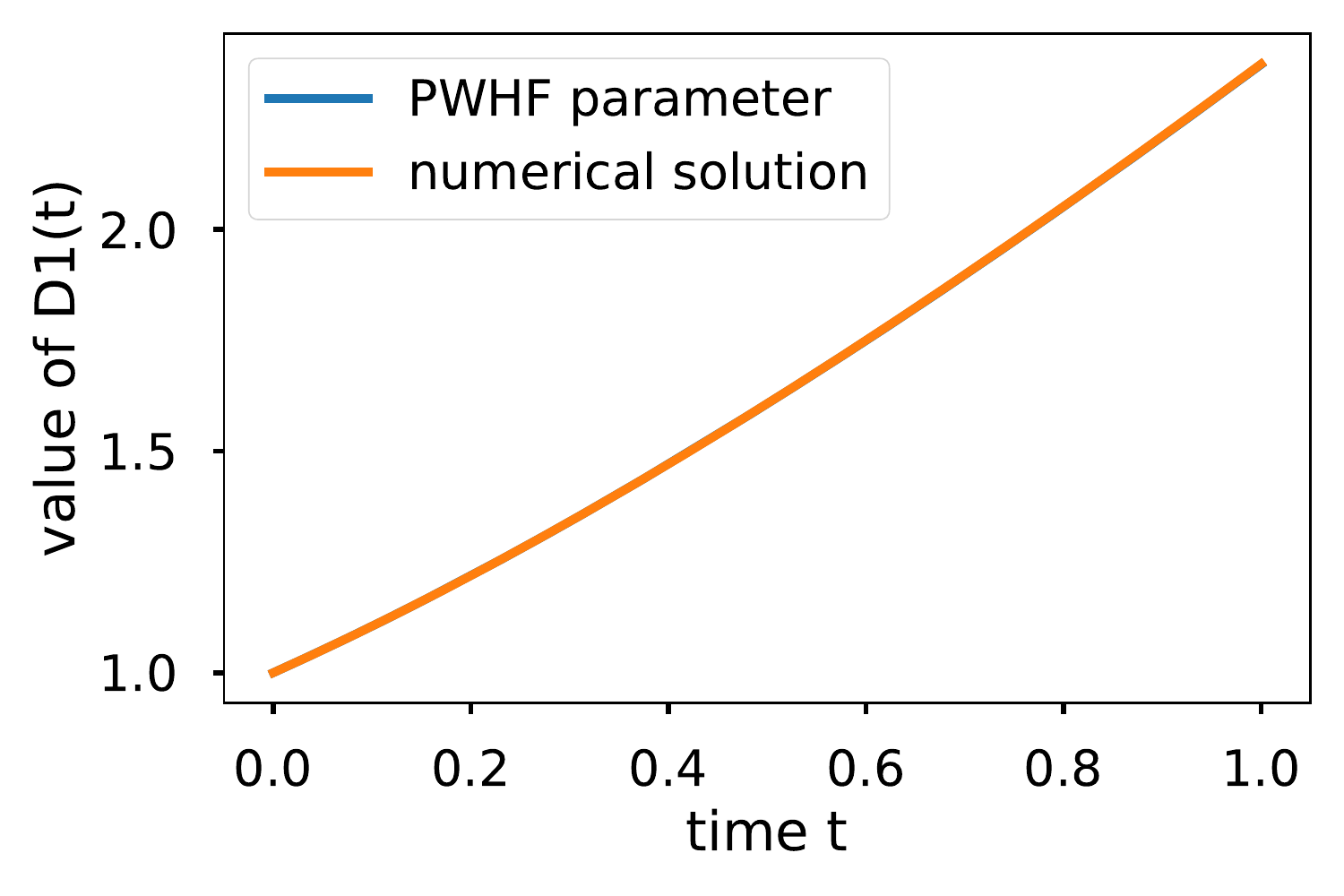}
    \caption{Value of $D, D_{\theta}$}
    \label{fig: entropy params demo}
\end{figure}
\subsection{Empirical bound on $\delta_0$} In section \ref{pseudo inverse error analysis}, we proved the convergence result for our $\theta(t)$. There are two crucial terms $\delta_0, \delta_1$ which show up in the error estimation (\ref{err est}) as well as (\ref{theorem: error estimation on Phi}). It's generally hard to give a sharp bound on these values, although one may resort to the well-known universal approximation theorem stating that they are small if proper neural networks are used. We can provide an empirical calculations for $\delta_0$ in the experiments.

We denote the approximation error at $\theta$ as:
\begin{align}
    \delta(\theta)= \int|\nabla\frac{\delta}{\delta \rho}\mathcal{F}\circ T_{\theta} (z) - \mathcal{K}_\theta[\nabla\frac{\delta}{\delta \rho}\mathcal{F}\circ T_{\theta}](z)|^2~d\lambda(z) .
\end{align}
It is easy to see that $\delta_0$ can be replaced by $\widehat{\delta}_0=\sup_{t\in[0, T]}\delta(\theta(t))$ in our error estimation, since only the information of $\delta(\theta)$ along the solution curve $\{\theta(t):t\in[0, T]\}$ is used in the proof.

Assume $\{\theta_k:0\leq k\leq K\}$ is the numerical solution from algorithm (\ref{alg:HFsolver}), where $K=T/h$ is the number of iterations. Denote the empirical bound on $\widehat{\delta}_0$ as:
\begin{align}
    \tilde{\delta}_0=\max_{0\leq k\leq K} \delta(\theta_k).
\end{align}
Let $\eta=\widehat{G}(\theta)^{\dagger}\nabla_{\theta}F(\theta)$ and notice that
\begin{align*}
    \mathcal{K}_\theta[\nabla\frac{\delta}{\delta \rho}\mathcal{F}\circ T_{\theta}](z)&=\partial_{\theta} T_\theta(z)\widehat{G}^{\dagger}(\theta) \int \partial_{\theta} T_\theta(z_0)^\top \nabla\frac{\delta}{\delta \rho}\mathcal{F}\circ T_{\theta}(z_0)\rho_\theta(z_0)~d\lambda(z_0)\\
    &=\partial_{\theta} T_\theta(z)\widehat{G}^{\dagger}(\theta) \nabla_{\theta}F(\theta)\\
    &=\partial_{\theta} T_\theta(z)\eta.
\end{align*}
Hence we can rewrite $\delta(\theta)$ as:
\begin{align*}
    \delta(\theta)&=\int|\nabla\frac{\delta}{\delta \rho}\mathcal{F}\circ T_{\theta}(z) -\partial_{\theta} T_\theta(z)\eta|^2~d\lambda(z)\\
    &=\int\left[\eta ^{\top}\partial_{\theta} T_\theta(z)^{\top}\partial_{\theta} T_\theta(z)\eta-2\nabla\frac{\delta}{\delta \rho}\mathcal{F}\circ T_{\theta}(z)^{\top}\partial_{\theta} T_\theta(z)\eta+|\nabla\frac{\delta}{\delta \rho}\mathcal{F}\circ T_{\theta}(z)|^2\right]~d\lambda(z)\\
    &=\eta^{\top}\widehat{G}(\theta)\eta -2\nabla_{\theta}F(\theta)^{\top}\eta + \mathbb{E}_{\rho_{\theta}}[|\nabla\frac{\delta}{\delta \rho}\mathcal{F}\circ T_{\theta}(z)|^2].
\end{align*}
With the above identity, we are able to evaluate the empirical error $\tilde{\delta}_0$. We report its value for each linear potential $V(x)$ in table~\ref{tab: empirical error}.

\begin{table}[htbp]
  \centering
  \caption{Empirical approximation error}
  \label{tab: empirical error}
  \begin{tabular}{|c|c|c|c|}
    \hline
    \textbf{Experiment}&
    \textbf{2d Geodesic} & \textbf{2d Harmonic Oscillator} & \textbf{10d Harmonic Oscillator} \\
    \hline
    Physical time $T$ & 4 & 40 & 10 \\
    \hline
    Number of iterations $K$& 2000 & 20000 & 10000 \\
    \hline
    Empirical error $\delta_0$&0 & 0.0035 & 0.0908 \\
    \hline
  \end{tabular}
\end{table}




\section{Discussion}
\label{sec:discussion}
The proposed method is also potentially applicable to many other important topics, for example, the nonlinear Schr\"odinger equation, the Schr\"odinger bridge problem, and the geodesic between two points on Wasserstein manifold. However, it needs further investigations to extend our algorithm to them, which is beyond the scope of this paper. We only provide a brief discussion here.

{\bf Geodesic between two points on Wasserstein manifold:} In the numerical experiments, we consider the initial value problem for Wasserstein geodesic equation with given $\rho_0, \Phi_0$. The commonly encountered geodesic problem is a 2-point boundary value problem, i.e., with given $\rho_0, \rho_T$ but no prior knowledge about the initial $\Phi_0$. In this case, there is not enough information for us to initialize the $p$ variable. A possible attempt is to choose a initial guess $\Phi(0, \cdot) $ and then apply shooting method on the ODE dynamics. Collocation or finite difference can be other options too. They deserve careful study in the future.  

{\bf Schr\"odinger equation:} Consider a potential $V\in C^{\infty}(\mathbb{R}^d)$ and functional $\mathcal{F}_1\in C^{\infty}(\mathcal{P}_{+}(\mathbb{R}^d))$ with variation $\frac{\delta}{\delta\rho}\mathcal{F}_1(\rho)=f(\rho)$, the nonlinear Schrodinger equation is given by:
\begin{align}
    i\frac{\partial}{\partial t}\psi(t, x)=-\frac{1}{2}\Delta\psi(t, x)+V(x)\psi(t, x)+f(\lvert{\psi}\rvert^2)\psi(t, x).
\end{align}
With the Madelung transform 
we can rewrite the complex wave function as $\psi(t, x)=\sqrt{\rho(t, x)}e^{i\Phi(t, x)}$, then the above equation in $\psi$ can be reformulated as a Wasserstein Hamiltonian flow with Hamiltonian:
\begin{align}
    \mathcal{H}(\rho, \Phi) = \int_{\mathbb{R}^d} \frac{1}{2}|\nabla \Phi(x)|^2\rho(x) dx+\int_{\mathbb{R}^d}V(x)\rho(x)dx+\mathcal{F}_1(\rho)+\frac{1}{8}\int_{\mathbb{R}^d}\lvert \nabla \log\ \rho(x)\rvert^2\rho(x)dx.
\end{align}

Possible challenges may arise due to the nonlinear potential $\int_{\mathbb{R}^d}\lvert \nabla \log\ \rho(x)\rvert^2\rho(x)dx$. This term involves the computation of space derivative of log density function, which is generally unavailable for multi-layer perceptron or the normalizing flow. One strategy is to use the Neural ODE as the push-forward map \cite{chen2018neural, grathwohl2018ffjord}, which supports such a computation. Another possible solution is to reformulate this functional as a mininization problem and estimate it via optimization techniques.

{\bf Schr\"odinger Bridge problem:} Similar to the Schrodinger equation, the Schr\"odinger Bridge problem can also be reformulated as WHF through the Hopf-Cole transform. Consider the Schrodinger bridge equation, sometimes also known as ``Schr\"{o}dinger system",
\begin{equation}
\label{SBE}
    \partial_t\eta_t=\frac{1}{2}\Delta \eta_t, \quad \partial_t\eta^*_t=-\frac{1}{2}\Delta \eta^*_t
\end{equation}
Here $\eta,\eta^*$ are two real valued functions. With the Hopf-Cole transformation:
\begin{equation}
    \eta = \sqrt{\rho}e^{\Phi/2}, \quad \eta^*=\sqrt{\rho}e^{-\Phi/2},
\end{equation}
equation \eqref{SBE} becomes a Wasserstein Hamiltonian flow with Hamiltonian:
\begin{align}
\label{SBE Hamiltonian}
    \mathcal{H}(\rho, \Phi) = \int_{\mathbb{R}^d} \frac{1}{2}|\nabla \Phi(x)|^2\rho(x) dx-\frac{1}{8}\int_{\mathbb{R}^d}\lvert \nabla \log\ \rho(x)\rvert^2\rho(x)dx.
\end{align}
The challenge in this problem is the combination of the difficulties in the two aforementioned examples. On the one side, we need to handle the computational challenge of $\nabla \log\ \rho$ term as mentioned in the Schrodinger equation. On the other side, Schr\"odinger bridge problem is a 2-point boundary value problem with given $\rho_0, \rho_T$ but no prior knowledge about $\Phi_0$. 

\section{Conclusion}
\label{sec:conclusion}
We close the discussion by summarizing that we developed a sampling based approach called PWHF for solving WHF in this work. PWHF is derived by applying Hamiltonian mechanics in the parameter space equipped with the pullback Wasserstein metric. Error estimates show that PWHF can approximate WHF with provable accuracy provided the pushforward map being efficient in approximation. Numerical examples demonstrate that our method is robust to the singularity of the equation and can scale up to high dimensional problems. There are still many work to be done about the WHF, which includes but not limited to: the application of PWHF to other Hamiltonian system such as Schr\"{o}dinger equation or Schr\"{o}dinger Bridge system; theoretical analysis on the quantities $\delta_0, \delta_1, \delta_2, \lambda_{\min}(\widehat{G})$; extension of PWHF to general Hamiltonian flow with on-quadratic kinetic energy. We hope the current study may serve as an starting point for furthering those investigations.





\section{Acknowledgments}
This research is partially supported by NSF grants DMS-1925263, DMS-2152960, DMS-2307465, DMS-2307466, and ONR grant N00014-21-1-2891. 

\appendix

\section{Derivation of Lagrangian $L$}\label{calculate L}
Recall that we introduce Lagrangian $L$ defined as
\begin{equation*}
  L(\theta, \dot\theta)=\mathcal{L}(T_{\theta \sharp}\lambda, (T_{\theta \sharp})_*\dot\theta).
\end{equation*}
We denote $\rho_\theta=T_{\theta\sharp}\lambda$, then $(T_{\theta \sharp})_*\dot \theta = \frac{\partial \rho_\theta}{\partial t}$. Actually, we can calculate the term $\frac{\partial \rho_\theta}{\partial t}$ as follows (c.f. proof of Theorem 3.4 of \cite{liu2022neural})
\begin{equation}
\frac{\partial \rho_\theta}{\partial t} = \frac{\partial\rho_\theta(x)}{\partial \theta}\cdot\dot\theta = -\nabla\cdot(\rho_\theta(x)\nabla\Psi_\theta(x)^\top\dot\theta).  \label{calculate dt rho theta}
\end{equation}

Now the Lagrangian $L$ can be computed as
\begin{align*}
    L(\theta, \dot\theta) & = \mathcal{L}(\rho_\theta, \frac{\partial \rho_\theta}{\partial \theta} \cdot \dot\theta)\\
    & = \frac{1}{2}\left(\int_{\mathbb{R}^d} -\nabla\cdot(\rho_\theta(x)\nabla\Psi_\theta(x)^\top\dot\theta) (-\Delta_{\rho_\theta})^{\dagger} (-\nabla\cdot(\rho_\theta(x)\nabla\Psi_\theta(x)^\top\dot\theta))~dx\right) - \mathcal{F}(\rho_\theta) \\ 
    & = \frac{1}{2}\left(\int_{\mathbb{R}^d} -\nabla\cdot(\rho_\theta(x)\nabla\Psi_\theta(x)^\top \dot\theta) \Psi_\theta^\top\dot\theta ~dx\right) - \mathcal{F}(\rho_\theta) \\
    & = \frac{1}{2}\left(\int_{\mathbb{R}^d} \dot\theta^\top \nabla\Psi_\theta(x)\nabla\Psi_\theta(x)^\top\dot\theta \rho_\theta(x)~dx \right) - \mathcal{F}(\rho_\theta) \\
    & = \frac{1}{2}\dot\theta^\top G(\theta)   \dot\theta-F(\theta).
\end{align*}

\section{Further discussion on geometric property of the map $\tau$}\label{append a }

Let us recall that $\tau$ is a map defined as
\begin{align*}
  \tau: & ~   T^*\Theta \longrightarrow T^*\mathcal{P}_\Theta\subset\mathcal{T}^*\mathcal{P},\nonumber \\
        & (\theta,~p) \longmapsto (T_{\theta\sharp}\lambda,~ \Psi_\theta^\top G(\theta)^{-1}p ).
\end{align*}
We are going to discuss the condition under which the map $\tau$  preserves the symplectic form. To express our idea clearly, let us first introduce the symplectic forms on the phase spaces $\mathcal{T}^*\Theta$ and $\mathcal{T}^*\mathcal{P}$. 



Let us recall that $\Theta$ is an open subset of $\mathbb{R}^m$. It is natural to treat $\mathcal{T}^*\Theta$ as a symplectic manifold equipped with the symplectic form 
$\omega_\Theta$ whose associated matrix representation is $$\Omega_\Theta = \left[\begin{array}{cc}
     & I_m \\
    -I_m  & 
\end{array}\right].$$ 
That is, $\omega_\Theta$ is a bilinear form defined on $\mathcal{T}(\mathcal{T}^*\Theta)$ such that for any two $C^1$ curves $\{\theta_t^1, p_t^1\}_{t\geq 0}$, $\{\theta_t^2, p_t^2\}_{t\geq 0}$ starting at the same point $(\theta_0, p_0)$, $\omega_\Theta$ is defined as $$\omega_\Theta((\dot\theta_0^1, \dot p_0^1), (\dot\theta_0^2, \dot p_0^2)) = \dot \theta_0^{1\top} \dot p_0^2 - \dot \theta_0^{2\top} \dot p_0^1.$$

On the other hand, we can also treat $\mathcal{T}^*\mathcal{P}$ as a symplectic manifold equipped with the symplectic form 
$\omega_{\mathcal{P}}$ whose associated matrix representation is $$\Omega_{\mathcal{P}} = \left[\begin{array}{cc}
     &  \textrm{Id} \\
    -\textrm{Id}  & 
\end{array}\right]$$. That is, $\omega_{\mathcal{P}}$ is a bilinear form defined on $\mathcal{T}(\mathcal{T}^*\mathcal{P})$ such that for any two $C^1$ curves $\{\rho_t^1, \Phi_t^1\}_{t\geq 0}$, $\{\rho_t^2, \Phi_t^2\}_{t\geq 0}$ both starting at $(\rho_0, \Phi_0)$, $\omega_{\mathcal{P}}$ is defined as $$\omega_{\mathcal{P}}((\dot\rho_0^1, \dot \Phi_0^1), (\dot\rho_0^2, \dot \Phi_0^2)) = \int_{\mathbb{R}^d} \partial_t\rho_0^1 \cdot \partial_t\Phi_0^2 - \partial_t \rho_0^2 \cdot \partial_t \Phi_0^1~dx.$$

We may treat both $(\mathcal{T}^*\Theta, \omega_\Theta)$ and $(\mathcal{T}^*\mathcal{P}, \omega_\mathcal{P})$ as symplectic manifolds.


We say a map $f: (M, \omega_M) \rightarrow (N, \omega_N)$ preserves the symplectic form if $f^*\omega_N = \omega_M$. Such geometric property is satisfied by a class of important maps in classical mechanics known as canonical transformations. For the sake of the completeness of our paper, we will investigate whether $\tau$ used in our method satisfies such a property.

Let us treat $M = \Theta$, $\omega_M = \omega_{\Theta}$ and $N = \mathcal{P}$, $\omega_{N}=\omega_{\mathcal{P}}$. In order to calculate $\tau^*\omega_\Theta$, we pick two arbitrary smooth curves $\{(\theta^{1}, p^1)\}$, $\{(\theta^2, p^2)\}$ on $\mathcal{T}^*\Theta$. Suppose the two curves intersect at $(\theta, p)$ when $t=0$. The the push-forward of vector fields $(\dot\theta^{i}, \dot p^i)$ ($i=1,2$) via $\tau$ is computed as
\begin{equation}
  \tau_* (\dot\theta^i, \dot p^i) = (-\nabla\cdot(\rho_{\theta}\nabla\Psi_{\theta}^\top \dot\theta^i),~  \Psi_{\theta}^\top G(\theta)^{-1}\dot p^i + \dot\theta^{i\top}\partial_\theta (\Psi_{\theta}^\top G(\theta)^{-1}) p^i )\in \mathcal{T}_{\tau(\theta, p)}\mathcal{T}^*\mathcal{P}_\Theta.  \quad i=1,2 \label{tau pushfwd velocity field}
\end{equation}

Then we compute
\begin{align}
  & \omega_\mathcal{P}(\tau_*(\dot\theta^1, \dot p^1), \tau_*(\dot\theta^2, \dot p^2))  \nonumber\\
  = &   \int_{\mathbb{R}^d} -\nabla\cdot(\rho_{\theta}\nabla\Psi_{\theta}^\top \dot\theta^1) \Psi_{\theta}^\top G(\theta)^{-1}\dot p^2~dx + \int_{\mathbb{R}^d}  -\nabla\cdot(\rho_{\theta}  \frac{\partial T_\theta}{\partial \theta}\circ T^{-1}_\theta(\cdot)  \dot\theta^1) 
 \dot\theta^{2\top}\partial_\theta ( p^{2\top} G(\theta)^{-1}\Psi_{\theta})~dx  \nonumber \\
  &   - \int_{\mathbb{R}^d} -\nabla\cdot(\rho_{\theta}\nabla\Psi_{\theta}^\top \dot\theta^2) \Psi_{\theta}^\top G(\theta )^{-1}\dot p^1~dx - \int_{\mathbb{R}^d}  -\nabla\cdot(\rho_{\theta}\frac{\partial T_\theta}{\partial \theta}\circ T^{-1}_\theta(\cdot) \dot\theta^2) 
 \dot\theta^{1\top}\partial_\theta (p^{1\top} G(\theta)^{-1}\Psi_{\theta})~dx  \nonumber
\end{align}
Notice that we replace $-\nabla\cdot(\rho_{\theta}\nabla\Psi_{\theta}^\top \dot\theta^i)$ by $-\nabla\cdot(\rho_{\theta}\frac{\partial T_\theta}{\partial \theta}\circ T^{-1}_\theta(\cdot) ~\dot\theta^i)$ for the second and the third term above.

Now the first integral equals $$  \dot\theta^{1\top} \left(\int_{\mathbb{R}^d}\nabla\Psi_{\theta} \nabla\Psi_{\theta}^\top\rho_{\theta}~dx\right) G(\theta)^{-1}\dot p^2=\dot\theta^{1\top} G(\theta) G(\theta)^{-1}\dot p=\dot\theta^{1\top} \dot p^2. $$
Similarly, the fourth term equals $\dot\theta^{2\top}\dot p^1$.

In order to analyze the second and the third term, we focus on the following third-order tensor
\begin{equation}
   \chi_\theta   = \int_{\mathbb{R}^d} \left[ {\frac{\partial T_\theta}{\partial \theta}(x)}^\top ~ \partial_\theta(\nabla (G( \theta )^{-1}\Psi_{\theta})_k )\circ T_\theta(x)    \right]_{k=1}^m ~d\lambda.
\end{equation}
Let us denote $\varphi_{\theta, k} = (G(\theta)^{-1}\Psi_\theta)_k$, i.e., $\varphi_{\theta,k}$ is the $k-$th component function of $G(\theta)^{-1}\Psi_\theta$. Notice that $\partial_\theta$ does not really take the derivative of $T_\theta$. Then one can verify that for all $k$, $1\leq k\leq m$,  
\begin{align*}
\left[ {\frac{\partial T_\theta}{\partial \theta}(x)}^\top ~ \partial_\theta(\nabla \varphi_{\theta,k} )\circ T_\theta(x)   \right]_{k=1}^m = & \underbrace{\partial_\theta\left[{\frac{\partial T_\theta}{\partial \theta}(x)}^\top {\nabla\Psi_{\theta}\circ T_\theta(x)}^\top G( \theta )^{-1}\right]}_{(1)}\\
 & - \underbrace{\left[{\frac{\partial^2 T_\theta}{\partial\theta^2}(x)}^\top \nabla \varphi_{\theta,k}\circ T_\theta(x)\right]_{k=1}^m}_{(2)}
 - \underbrace{\left[{\frac{\partial T_\theta}{\partial \theta}(x)}^\top ~ \nabla^2\varphi_{\theta,k}\circ T_\theta(x)\frac{\partial T_\theta(x)}{\partial\theta}\right]_{k=1}^m}_{(3)} 
\end{align*}
Let us integrate (1) w.r.t. $\lambda$, by swapping integration and $\partial_\theta$, it is not hard to verify that the integration of (1) equals $0$, since the integral inside the square brackets equals identity matrix $I_m$, which is independent of $\theta$. Thus we know that the tensor  $\chi(\theta)$ is the integration of the sum of terms (2) and (3), i.e.,
\begin{equation}
 \chi_\theta = \int_{\mathbb{R}^d} \left[{\frac{\partial^2 T_\theta}{\partial\theta^2}(x)}^\top \nabla \varphi_{\theta,k}\circ T_\theta(x) + {\frac{\partial T_\theta}{\partial \theta}(x)}^\top ~ \nabla^2\varphi_{\theta,k}\circ T_\theta(x)\frac{\partial T_\theta(x)}{\partial\theta} \right]_{k=1}^m 
 ~d\lambda.
\end{equation}
Thus one can verify that
\begin{equation}
   \omega_\mathcal{P}(\tau_*(\dot\theta^1, \dot p^1), \tau_*(\dot\theta^2, \dot p^2)) = \dot\theta^{1\top} \dot p^2 - \dot\theta^{2\top} \dot p^1  - \chi_\theta(\dot\theta^1, \dot\theta^2, p) - \chi_\theta(\dot\theta^2, \dot\theta^1, p).  \label{preserve symp }
\end{equation}
Here, for any $u,v,w\in\mathbb{R}^m$, the tensor-vector multiplication is defined as
\begin{align*}
\chi(\theta)(u,v,w) =  \int_{\mathbb{R}^d} & {u^\top\frac{\partial^2 T_\theta}{\partial\theta^2} v } \cdot \nabla (w^\top G(\theta)^{-1}\Psi_\theta)\circ T_\theta(x) \\
& + {\frac{\partial T_\theta(x)}{\partial \theta} u}^\top ~ \nabla^2(w^\top G(\theta)^{-1}\Psi_\theta)\circ T_\theta(x) ~\frac{\partial T_\theta (x)}{\partial\theta} v ~d\lambda.
\end{align*}
By definition of pullback of differential form, we have $\tau^*\omega_\mathcal{P}((\dot\theta^1, \dot p^1), (\dot\theta^2, \dot p^2)) = \omega_\mathcal{P}(\tau_*(\dot\theta^1, \dot p^1), \tau_*(\dot\theta^2, \dot p^2))$; 
We can further verify that $\dot\theta^{1\top} \dot p^2 - \dot\theta^{2\top} \dot p^1 =\omega_\Theta((\dot\theta^1, \dot p^1), (\dot\theta^2, \dot p^2))$. Also, $\chi_\theta$ is symmetric w.r.t. the first two components, i.e., $\chi_\theta(\dot\theta_1, \dot\theta^2, \cdot) = \chi_\theta(\dot\theta^2, \dot\theta^1, \cdot)$; Thus the above calculation \eqref{preserve symp } leads to
\begin{equation}
  \tau^*\omega_\mathcal{P}((\dot\theta^1, \dot p^1), (\dot\theta^2, \dot p^2))=\omega_\Theta((\dot\theta^1, \dot p^1), (\dot\theta^2, \dot p^2)) - 2\chi_\theta(\dot\theta^1, \dot\theta^2,   p).
\end{equation}
This implies that the symplectic matrix associated with $\tau^*\omega_{\mathcal{P}}$ takes the following form
\begin{equation}
   \Omega ( \theta, p ) = \left[\begin{array}{cc}
      -  2\chi_\theta(\cdot,\cdot, p)  &  -I_m  \\
       I_m  &    O_m 
    \end{array}\right].
\end{equation}


In most cases, it is not guaranteed that $\tau$ preserves the symplectic form $\omega_{\mathcal{P}}$ since the parametrized push-forward map $T_\theta$ may not guarantee that $\chi_\theta=0$.

It is worth mentioning that the preservation of the symplectic form is not the necessary condition for the convergence of our numerical method: Although our $\tau$ is not guaranteed to preserve the symplectic form, we still have theoretical guarantees on the numerical accuracy of our method (c.f. Section 3.4). 

We end our discussion with two interesting questions that may serve as future research directions. 
\begin{enumerate}
    \item Does there exist a special family of pushforward maps $T_\theta$ that vanish the tensor $\chi_\theta$ and thus preserve the symplectic form?
    \item We may recast our PWHF by using the directly pull-backed symplectic matrix $\Omega(\theta, p)$, i.e., we consider the modified PWHF
    \begin{align*}
      (\dot \theta , \dot p )^\top = \Omega(\theta, p)^{-1} \nabla H(\theta, p).
    \end{align*}
    Will the above modified PWHF gain better theoretical or numerical properties?
\end{enumerate}

\bibliographystyle{siamplain}
\bibliography{references}
\end{document}